\documentclass[11pt, usletter]{article}
\usepackage{amsmath, amsthm, amssymb, eucal, amscd, xypic}
\usepackage{mathpazo}
\newcommand{\partentry}[1]{\addtocontents{toc}
{\small\bfseries#1\hfill\thepage\par}}
\makeatletter
\def\@part[#1]#2{%
    \ifnum \c@secnumdepth >\m@ne
      \refstepcounter{part}
      \partentry{\protect\makebox[2em][l]{\thepart}#1}
\else
      \partentry{#1}
    \fi
    {\noindent\normalfont\Large\bfseries\thepart\hspace{1em}#2\par}
    \nobreak
    \vskip 3ex
    \@afterheading}
\def\@spart#1{%
    {\noindent\normalfont\Large\bfseries #1\par} 
     \nobreak
     \vskip 3ex
     \@afterheading}
\renewcommand\section{\@startsection{section}{1}{\z@}
{-3.5ex \@plus -1ex \@minus -.2ex}
{2ex \@plus.2ex}
{\large\bfseries}}
\renewcommand\subsection{
\@ifstar{\setcounter{subsection}{\value{equation}}
\@startsection{subsection}{2}{\z@}
                          {1.75ex \@plus.5ex \@minus.2ex}%
                           {-.4em} 
			{\itshape}*}
{\setcounter{subsection}{\value{equation}}
\stepcounter{equation}
\@startsection{subsection}{2}{\z@}
                          {1.75ex \@plus.5ex \@minus.2ex}%
                           {-.4em} 
{\itshape}}}
\def\@seccntformat#1{\@ifundefined{#1@cntformat}%
{\csname the#1\endcsname\quad} 
{\csname #1@cntformat\endcsname}}
\def\section@cntformat{\thesection.~}
\def\subsection@cntformat{(\thesubsection)\ }
\renewcommand*\l@section{\mdseries\small\@dottedtocline{1}{1.5em}{2em}}
\makeatother

\numberwithin{equation}{section}
\theoremstyle{plain}
\newtheorem*{maintheorem}{Theorem}
\swapnumbers
\newtheorem{theorem}[equation]{Theorem}
\newtheorem{corollary}[equation]{Corollary}
\newtheorem{lemma}[equation]{Lemma}
\newtheorem{proposition}[equation]{Proposition}

\theoremstyle{definition}
\newtheorem{definition}[equation]{Definition}
\theoremstyle{remark}
\newtheorem{remark}[equation]{Remark}
\newtheorem{example}[equation]{Example}

\newcommand{\ol}{\overline}

\newcommand{\ad}{\mathrm{ad}}

\newcommand{\cA}{\mathcal{A}}

\newcommand{\cC}{\mathcal{C}}
\newcommand{\cF}{\mathcal{F}}

\newcommand{\clL}{\mathcal{L}}

\newcommand{\frS}{\mathfrak{S}}

\newcommand{\BF}{\mathbf{F}}
\newcommand{\BH}{\mathbf{H}}

\newcommand{\BSp}{\mathbf{Sp}}
\newcommand{\bC}{\mathbb{C}}

\newcommand{\bP}{\mathbb{P}}
\newcommand{\bQ}{\mathbb{Q}}

\newcommand{\bT}{\mathbb{T}}
\newcommand{\bZ}{\mathbb{Z}}

\newcommand{\Hom}{\mathrm{Hom}}
\newcommand{\GL}{\mathrm{GL}}
\newcommand{\Id}{\mathrm{Id}}
\newcommand{\End}{\mathrm{End}}
\newcommand{\bbrak}[1]{[\![#1]\!]}
\newcommand{\ppar}[1]{(\!(#1)\!)}
\newcommand{\nodprop}{\supset\hspace{-2pt}\subset}

\newcommand{\eul}{\mathrm{eul}}
\newcommand{\provomega}{z}

\begin{document}                                                                                                   
\title{\textbf{The Structure of 2D Semi-simple Field Theories}
}
\date{August 9, 2011}
\author{Constantin Teleman} 
\maketitle
\begin{quote}
\abstract{
\noindent I classify the cohomological 2D field theories based on a 
semi-simple complex Frobenius algebra $A$. They are controlled by 
a linear combination of $\kappa$-classes and by an extension datum to 
the Deligne-Mumford boundary. Their effect on the Gromov-Witten 
potential is described by Givental's Fock space formulae. This leads  
to the \emph{reconstruction of Gromov-Witten (ancestor) invariants} 
from the quantum cup-product at a \emph{single} semi-simple point and 
the first Chern class of the manifold, confirming Givental's higher-genus 
reconstruction conjecture. This in turn implies the Virasoro conjecture 
for manifolds with semi-simple quantum cohomology. The classification uses 
the Mumford conjecture, proved by Madsen and Weiss \cite{madw}.}
\end{quote}

\section*{Introduction}
This paper studies structural properties of \emph{topological 
field theories} (TFT's), a notion introduced by Atiyah and Witten 
\cite{wit} and inspired by Segal's axiomatisation of Conformal Field 
Theory. A TFT extracts the topological information which is implicit 
in quantum field theories defined over space-time manifolds more general 
than Euclidean space. The first non-trivial example is in $2$ dimensions, a 
setting which has been the focus of much interest in relation to 
\emph{Gromov-Witten theory}: the latter captures the expected count 
of pseudo-holomorphic curves in a compact symplectic target manifold. 
The result proved here, the classification of semi-simple theories, 
shows that an important property of these invariants is a formal 
consequence of the underlying structure, rather than a reflection 
of geometric properties of the target manifold. Loosely stated, the 
property in question is that a count of rational curves  with three 
marked points, encoded in the \emph{quantum cohomology} of the target, 
determines the answer to enumerative questions about curves of all genera, 
when the quantum cohomology ring is semi-simple.\footnote{To be precise, this 
is true of the so-called \emph{ancestor} Gromov-Witten invariants. The 
complete, \emph{descendent} invariants require additional genus zero 
information, the \emph{$J$-function}.}  

My classification leaves some important questions open (see \cite{ecm} for 
more discussion). One of them is to extract the Gromov-Witten classification 
data from the geometry of the symplectic manifold. Finding even a single semi-simple 
quantum cup-product (when one exists) may require infinite information, if curves 
are counted degree-by-degree.  
Another, more precise question concerns the \emph{degeneration} of a semi-simple theory 
to the locus in its Frobenius manifold (the natural parameter space, see \S7) 
where the algebra acquires nilpotents. An example is the discriminant locus within 
the deformation space of an isolated singularity: it is unclear whether 
the higher-genus part of the associated TFT, the \emph{Landau-Ginzburg $B$-model} 
for a singularity, extends continuously there (the semi-simple classification data blow up).

\subsection{First definition.} A \textit{$2$-dimensional topological field theory} 
over a ring $k$ is a strong symmetric monoidal functor $Z$ from the $2$-dimensional 
oriented bordism category to the tensor category of finitely 
generated projective $k$-modules. This means that $Z$ assigns to every 
closed oriented $1$-manifold $X$ a $k$-module $Z(X)$, and to any 
compact oriented surface $\Sigma$, with independently oriented boundary 
$\partial\Sigma$, a linear ``propagator" 
\[
Z(\Sigma): Z(\partial_-\Sigma) \to Z(\partial_+\Sigma).
\]
The sign $\pm$ of a boundary component compares the orientation induced from 
$\Sigma$ with the independent one on $\partial\Sigma$; we call $\partial_-\Sigma$ 
the \textit{incoming} boundary and $\partial_+\Sigma$ the \textit{outgoing} one. 
The above definition requires that 
\begin{enumerate}\itemsep0ex
\item $Z$ is multiplicative under disjoint unions, $Z(X_1\amalg X_2) 
= Z(X_1)\otimes Z(X_2)$.
\item Sewing boundary components leads to the composition of maps. 
\end{enumerate}
\noindent  
Part (i) is the monoidal condition, while (ii) is the functorial property. 
Note that the cylinder ``$=$'' with one incoming and one outgoing end represents 
the identity. In the simplest definition of the bordism category, morphisms 
are surfaces modulo oriented homeomorphism (rel boundary); more sophisticated 
definitions remember the topology of the diffeomorphism group (Remark~\ref{sophtft}).  

\subsection{First classification.} A folk theorem (with non-trivial proof, 
see \cite{ab}) ensures that $Z$ is equivalent to the datum of a \emph{commutative 
Frobenius $k$-algebra} structure on the space $A:=Z(S^1)$. This last notion 
comprises a commutative $k$-algebra structure on $A$, together with an $A$-module 
isomorphism $\iota: A \xrightarrow{\sim} A^*:= \Hom_k(A,k)$. The Frobenius structure 
on $Z(S^1)$ can be read from the functor $Z$ as follows:
\begin{itemize}\itemsep0ex
\item the multiplication map $A\otimes{A} \to{A}$ is 
defined by the trinion with two incoming circles and an outgoing one;
\item the unit in $A$ is defined by the disk with outgoing boundary, 
$Z(\supset): k\to A$; 
\item the disk $\subset$ with incoming boundary determines the vector $\theta 
:= \iota(1)\in A^*$.  
\end{itemize}
(My pictures represent the projection outlines of surfaces, with their boundaries 
omitted. Also, the reader will have noticed that surfaces are `read' from 
right to left, matching the ordering convention for the composition of operators.) 
The form $\theta$, in turn, determines a symmetric pairing $\beta: A\times A \to k$, 
$\beta(a\times b)= \theta(ab)$, which is the partial adjoint to $\iota$ in one 
of the variables. Non-degeneracy of $\beta$ --- equivalently, the isomorphy condition 
on $\iota$ --- is also known as \emph{Zorro's lemma},\footnote{I believe the 
name was coined by Jacob Lurie.} and is proved by the diagram wherein a 
``Z''-shaped identity cylinder is factored into a ``right elbow" $\Supset$ 
(that is, a cylinder with two outgoing ends), sewed on to a left elbow $\Subset$ 
at one of its outputs: $Z(\Subset)$ represents $\beta$, and $Z(\Supset)$ 
provides an inverse co-form.

\subsection{Semi-simple case.}\label{eulercharformula}
 An easy but important special case concerns \textit{semi-simple} 
algebras $A$ over $k=\bC$. As algebras, these are isomorphic to $\bigoplus_i 
\bC\cdot P_i$ for projectors $P_i$, uniquely determined up to reordering. 
From the definition and non-degeneracy of $\beta$, the projectors are 
$\beta$-orthogonal and their $\theta$-values $\theta_i = \theta(P_i)$ must be 
non-zero complex numbers. Up to isomorphism, $A$ is classified by the (unordered) 
collection of the $\theta_i$. The TFT is easy to describe in the \emph{normalised 
canonical basis} of rescaled projectors $p_i:= \theta_i^{-1/2}P_i$, as follows. 
For a connected surface $\Sigma$ with $m$ incoming and $n$ outgoing boundaries, 
the matrix of the propagator $Z(\Sigma)$ has entry $\theta_i^{\chi(\Sigma)/2}$ 
linking $p_i^{\otimes{m}}$ to $p_i^{\otimes{n}}$, while all entries involving 
mixed tensor monomials in the $p_i$ are null. I leave it to the reader's care 
to supply the correct reading of this rule when $m$ or $n$ are zero. 

\subsection{Example: the Euler class.}\label{ss=invert}
A Frobenius algebra contains a distinguished vector, the \emph{Euler class}~$\alpha$, 
which is the output of a torus with one outgoing boundary. When $A$ is the 
cohomology ring of a closed oriented manifold with coefficients in a field and 
$\beta$ the Poincar\'e duality pairing, $\alpha$ is the usual Euler class. 
(Of course, $A$ will be a \emph{skew}-commutative, if there is any odd cohomology.) 
By contrast, in the semi-simple case, $\alpha$ is the \emph{invertible} element 
$\sum_i\theta_i^{-1}P_i$. The endomorphism of $A$ defined by a two-holed surface 
of genus $g$ is the multiplication by $\alpha^g$: in matrix form, $\mathrm{diag} 
[\theta_i^{-g}]$. In the semi-simple case, this observation allows the recovery 
of low-genus $Z$ from high genus, and will play a key role in the paper. 

There is actually a converse: invertibility of $\alpha$ implies 
semi-simplicity of $A$. (The trace on $A$ of the operator of multiplication 
by $x$ is $\theta(\alpha{x})$, so $\mathrm{Tr}_A$ defines a non-degenerate 
bilinear form on $A$; it follows that, over any residue field of the ground 
ring $k$, $A$ is a sum of separable field extensions.) This, and the importance 
of an invertible $\alpha$, were perhaps first flagged by Abrams, also in 
connection with quantum cohomology; the reader is referred to the nice paper 
\cite{ab2}. 
\vskip1.5ex

\subsection{What this paper does.} Here, I give an algebraic classification 
for \textit{family TFTs} (FTFTs), in which the surfaces vary in families and 
the functor $Z$ takes values in the cohomology of the parameter spaces, with 
coefficients in the space of maps between tensor powers of $A$. These theories 
are variants of the \emph{Cohomological Field Theories} (CohFT's) introduced by 
Kontsevich and Manin \cite{km}. ``Families" consisting of single surfaces recover 
the previous TFT notion, detecting the underlying Frobenius algebra $A$. My 
classification applies whenever $A$ is semi-simple and $k$ is a field of 
characteristic zero; I use $\bC$ for simplicity. 

\subsection{The Gromov-Witen case.} The theories of greatest interest involve 
nodal surfaces, the \emph{stable curves} of algebraic geometry, and come from 
\emph{Gromov-Witten invariants}. In this setting, I provide a structure formula
for the Gromov-Witten invariants of manifolds whose quantum cohomology is 
generically semi-simple. Such theories have additional structure,
a grading which stems from the fact that spaces of stable maps have 
topologically determined (expected) dimensions. This structure limits 
the freedom of choice considerably: the full FTFT is determined by the 
Frobenius algebra and the grading information. This affirms a conjecture 
of Givental's \cite{giv} on the reconstruction of higher-genus invariants, 
and in particular, as pointed out in \cite{giv3}, the Virasoro conjecture 
for such manifolds. Verification of this conjecture involves tracing 
through Givental's construction, with an improvement to the formulation 
which (I think) is originally due to M.~Kontsevich, and which we review in 
\S\ref{quadhamilt}. 

\subsection{Relation to ``open-closed'' theories.} With \emph{different} 
starting hypotheses, a vast extension of my classification has been reached 
by Kontsevich and collaborators in the framework of \emph{open-closed} FTFTs 
(see \cite{kkp} and sequels in preparation). From that perspective, I show 
that any semi-simple (closed string) CohFT might as well be assumed to come 
from an open-closed FTFT with a semi-simple category of boundary states. In
Gromov-Witten theory, this statement could even follow from a 
sufficiently optimistic formulation of Homological Mirror Symmetry: 
semi-simplicity of quantum cohomology suggests a Landau-Ginzburg 
B-model mirror with isolated Morse critical points of the potential, 
since (in the case of isolated singularities) the quantum cohomology 
ring is meant to be isomorphic to the Jacobian ring of the potential.
In this situation, the mirror category of boundary states (B-branes) is 
also semi-simple. Assuming all this, we could then invoke Kontsevich's 
classification. 

However, while it seems clear that the open-closed framework (or some
related $2$-categorical approach) is the right setting, Gromov-Witten 
theory is not quite ready for it, as the requisite assumptions on the 
Fukaya category of boundary states have only been checked in special 
cases; whereas the CohFT axioms are well-established. Examples of varieties 
with generically semi-simple quantum cohomology include: toric manifolds, 
most Fano three-folds with no odd Betti numbers \cite{ciolli}, as well 
as blow-ups of such varieties at any number of points \cite{bayer}. 
Of these, only for toric ones does the open-closed theory seem to be 
in convincing shape, thanks to work by Fukaya and collaborators 
\cite{fooo}. 


\section*{Acknowledgements}
I have lectured on successive approximations of this work since the 
fall of 2004. Comments from colleagues were essential in correcting 
and re-shaping the results, and it is a pleasure to thank them here. 
Early on, K.~Costello suggested a link between my classification and 
Givental's conjecture; T.~Coates, Y.~Eliashberg, S.~Galatius, E.~Getzler, 
M.~Kontsevich, Y.~Manin, J.~Morava and D.~Sullivan all contributed 
useful comments. Y.-P.~Lee offered numerous helpful comments on the 
early form of the manuscript. Special thanks are due to A.~Givental 
for patient explanation of his work, as well as to the anonymous referees 
for careful reading and numerous suggestions for improvement.

Substantial portions of this work were completed during the 2006 programme 
on ``New Topological structures in physics" at MSRI and much of the initial 
writing was done during a visit to the Max Planck Institute in Bonn in August 
2007. Partial support was 
provided by NSF grant DMS-0709448.

\section*{}
\begin{minipage}[t]{10cm}
\tableofcontents
\end{minipage}

\section{Summary of definitions and results}
\label{summary}
This section outlines the definition and classification of the various 
versions of FTFT's used throughout the paper, as well as the background 
of the two key results, Theorems\ 1 and 2, formulated towards the end 
of the section. It is not possible to cover \emph{all} the details in the 
space suited to an opening section, and the reader will often be referred 
to later paragraphs for clarification. For instance, classifying spaces of 
surface bundles are discussed in \S2; a refresher on $\kappa$- and $\psi$-classes 
is found in \S\ref{tautrefresh}; and the list of axioms for a DMT is only truly 
completed by spelling out the `nodal relations' in \S\ref
{nodalrelations}.  
 
\subsection{Functorial definition.} \label{functdef}
Family TFT's admit a categorical definition in the style of the introduction. 
I give it here for logical completeness; its meaning and use, in the several 
variants outlined in \S\ref{classif} below, will be spelt out more clearly 
in Section 2. 

Consider the following two contra-functors $\cC$ and $\cF$, defined 
over the category of topological spaces and continuous maps, and taking 
values in symmetric monoidal categories. On a topological space $X$, 
the first category $\cC(X)$ has as objects bundles of closed oriented 
$1$-manifolds over $X$, and as morphisms bundles of compact oriented 
$2$-bordisms, modulo boundary-fixing oriented homeomorphisms over $X$. 
Objects of the second category $\cF(X)$ are flat complex vector bundles 
over $X$, while the morphisms are the graded vector spaces 
\[
\mathrm{Hom}_{\cF(X)}(V,W) := H^\bullet\left(X; \mathrm{Hom}_X(V,W)\right).
\]
A FTFT is a symmetric monoidal transformation $Z$ from $\cC$ to $\cF$. 
Variants of this notion are obtained by  changing the defining features  
of $\cC$: we can require all circles in $\cC(X)$ to be parametrised (\S\ref
{classif}.i) or not (\S\ref{classif}.ii), allow Lefschetz fibrations 
as morphisms (\S\ref{classif}.iii), and finally, impose the 
Deligne-Mumford stability condition on the surface fibrations 
(\S\ref{classif}.iv).

\begin{remark}\label{sophtft}
The objects of $\cC$ and $\cF$ form \emph
{sheaves} over the site of topological spaces, but the morphisms do not. 
Morphisms of $\cF$ are the cohomologies of a differential-graded version 
of $\cF$, in which the objects are complexes of local coefficient systems 
over $X$ and the morphisms are co-cycles, instead of cohomology classes. 
There is a similar enhancement of $\cC$ to a sheaf of categories enriched 
over topological spaces: morphisms are classifying spaces of the 
homeomorphism groupoids of surface bundles. A (symmetric monoidal) natural 
transformation between these sheaves of categories is a possible definition 
of \emph{chain-level FTFT's}, and is closely related to Segal's definition 
of \emph{topological conformal field theory} \cite{seg, get}. We will 
not use this more refined notion in the paper. 
\end{remark}

\subsection{FTFT variants and their classification.}\label{classif}
We will consider several versions of family field theories; their 
classification increases in complexity. All four variants below are 
relevant to the eventual focus of interest, semi-simple cohomological 
field theories. 

\begin{enumerate}
\item In the simplest variant, the surfaces have parametrised boundaries. 
These theories are classified by a single, group-like class $\tilde{Z}^+$ 
in the $A$-valued cohomology of the stable mapping class group of surfaces 
(\S\ref{primclass}). As a result of the Mumford conjecture \cite{madw}, such 
a class is necessarily of the form $\exp\{\sum_{j> 0} 
a_j\kappa_j\}$, with arbitrary elements $a_j\in A$ coupled to the 
Morita-Mumford classes $\kappa_j$. The class $\tilde{Z}$ associated to 
a surface bundle acts diagonally on tensor monomials of the normalised 
canonical basis, as follows. if $a_j = \sum a_{ij}P_i$, $a_{ij}\in \bC$, then 
the entry $\theta_i^{\chi/2}$ in the propagator matrix \eqref{eulercharformula} 
which the Frobenius algebra assigns  to a single surface is now multiplied by 
the factor $\exp\{\sum_{j>0}a_{ij}\kappa_j\}$. Note that, as $\chi = -\kappa_0$, 
we could also account for the $\theta_i$ by including in our sum a term $j=0$, 
with $a_0 = \frac{1}{2}\log \alpha$. 

\item The second FTFT variant allows the boundaries to rotate freely. This 
introduces a new classification datum, a $\bC$-linear map $E: A\to A\bbrak
{\provomega}$ with $E \equiv \Id \pmod{\provomega}$. A free boundary theory is 
determined by $E$ and the earlier $\tilde{Z}^+$ as follows: twist the incoming 
states by $E^{-1}$ and the outputs by $E$, with $z$ specialised to the 
sign-changed Euler classes of the respective boundary circle bundles. 
(The awkward sign is reluctantly adopted here to avoid worse later; it stems 
from the sign mismatch between Euler and $\psi$ classes for inbound circles.) 
In-between, the fixed-boundary propagator of (i) applies. 

\begin{remark} The meaning of $E$ is obscured by our simplified, cohomological 
setting; it can be reverse-engineered from the context of open-closed and
chain-level FTFTs. In the functorial setting~\eqref{functdef}, the local system 
$Z(S^1)$ with fibre $A$ over $\bC\bP^\infty$ defined by the universal circle bundle 
is necessarily trivial, because $\bC\bP^\infty$ is simply connected. Our $E$ supplies 
a second, 'interesting' trivialization of the same, $(-\provomega)\in H^2(\bC\bP^\infty)$ 
is the universal Euler class. In the chain-level version of the theory, $A$ is the homology 
of a space $X$ (or a chain complex) with circle action, and the inputs and outputs 
at free boundaries belong naturally to the circle-equivariant cohomology, and 
$E$ is here to split the latter as $A\otimes\bC\bbrak
{\provomega}$. When the circle action on $X$ is trivialized for independent 
reasons, as happens with the Hochschild complex of a \emph{semi-simple} category of 
boundary states, $E$ expresses the difference between the `obvious' splitting 
and the one relevant the field theory. (See \cite{kkp} for more discussion.)  
\end{remark}

\item  Next in line are the \textit{Lefschetz theories}, where surfaces are 
allowed to degenerate nodally into the Lefschetz fibrations of algebraic geometry. 
A nodal surface can be deformed uniquely to a smooth one; the cohomological 
nature keeps $Z$ unchanged under this deformation, so adding \textit{single} 
nodal surfaces to the theory involves no new information. Things are different 
in a \textit{family}: up to homotopy, the automorphism group of the 
``nodal propagator" $\nodprop$, an incoming-outgoing pair of crossing disks, 
is the product $\bT\times \bT$ of the two independent circle rotation groups. 
This provides a new datum $Z(\nodprop)$, an $\mathrm{End}(A)$-valued formal 
series $D(-\omega_+,\omega_-)$ in the Euler classes $\omega_\pm$ of the two 
universal disk bundles. 

Keeping only the diagonal rotation, we can deform the node $\nodprop$ into a 
rotating cylinder. Since $Z(=)=\Id$ for a fixed cylinder, and it must remain 
a projector when the cylinder rotates, we conclude that $D=\Id \mod{(\omega_+ 
- \omega_-)}$. In addition, we will find a symmetry constraint relating $D$ 
and $E$; see \S\ref{dmsec1} for the precise relation. These are all the data 
and constraints: an involved, but explicit formula for the Lefschetz theory 
classes is given in \S\ref{lefschetzconst} from $\tilde{Z}, E$ and $D$, as a 
kind of ``time-ordered exponential integral'' along the surfaces in any family.

\item Lastly, we are interested in \textit{Deligne-Mumford theories} 
(DMT's): these are Lefschetz theories involving only \emph{stable} nodal
surfaces, the Deligne-Mumford stable curves of algebraic geometry. Excluding 
cylinders and disks (which are unstable) brings about the need for two additional 
axioms, the \emph{nodal factorisation rules} and \emph{vacuum axiom}, which 
in a Lefschetz theory follow from the other axioms (See \S\S\ref{nodexplained}--\ref
{vacax}). 

The best-known DMT's are the Cohomological Field theories \`a la Kontsevich 
and Manin, which satisfy $D =\Id$. It is more customary to state their 
structure in terms of surfaces with inputs only, but that is a matter 
of convenience. In CohFT's, the compatibility constraint on $E$ will become 
\emph{Givental's symplectic constraint} $E(\provomega)\circ E^*(-\provomega) 
=\mathrm{Id}$. The main examples of CohFT's are the Gromov-Witten cohomology 
theories of compact symplectic manifolds, which carry even more structure 
and constraints: see \S\S\ref{gromwit}--\ref{gwconst} below. 
\end{enumerate}

\noindent Functors  of the types (i), (ii) and (iii, iv) shall be denoted by 
$\tilde{Z}, Z$ and $\bar{Z}$, respectively. In the semi-simple case, we will
find that (iii) and (iv) have the same classification.

\subsection{Idea of proof.} For the first two types of FTFTs, the classification 
is an easy consequence of the Mumford conjecture, proved by Madsen and Weiss 
\cite{madw}. (We will also use an older result of Looijenga's on $\psi$-classes, 
\cite{lo}.) In the limit of large genus surfaces, the sewing axiom becomes 
an equation in the complex cohomology of the stable mapping class group. 
The latter is a power series ring in the tautological classes (see 
\S\ref{tautrefresh}), and we solve the equation there. Semi-simplicity 
of $A$ lets us retrieve  the low-genus answer from high genus thanks to 
invertibility of the Euler class $\alpha$. 

DMT's require an additional argument. The universal families of stable 
nodal surfaces are classified by orbifolds with a normal-crossing 
stratification. The argument above determines the classes $Z$ on each 
stratum, but there could be ambiguities and obstructions in patching 
these classes together. However, the Euler classes of certain boundary 
strata involving large-genus surfaces are \textit{not} zero-divisors in 
low-degree cohomology. This ensures the unique gluing of cohomology classes 
over suitably chosen strata. We find enough strata to cover all Deligne-Mumford 
moduli orbifolds, and prove the unique patching of the $Z$-classes to 
a global class $\ol{Z}$. This observation is the key contribution of 
the paper; the remainder falls in the ``known to experts" category. 
 
A more natural resolution of the gluing ambiguity involves the use of 
\emph{chains}, instead of homology classes. This point of view, pioneered 
by Kontsevich in the context of homological mirror symmetry, fits 
naturally with the notion of \emph{open-closed field theories} and their 
$A_\infty$-categories, and was successfully developed by Costello, 
leading in that setting to a beautiful classification result \cite{cost}. 
It also ties in nicely with the \emph{string topology} example of Chas and 
Sullivan \cite{sul}. From this angle, my result shows that the semi-simple 
case is considerably easier: open strings and chain-level structures are 
not needed. 

\subsection{Example: Gromov-Witten theory.}\label{gromwit}
Here, the Frobenius algebra $A$ is the quantum cohomology of a compact 
symplectic manifold $X$, at some chosen point $u\in H^{ev}(X)$. To apply 
my classification, we must choose a point $u$ where this ring is semi-simple  
(assuming such a point exists, which is a strong restriction on the manifold 
$X$). This $u$ may be the generic point --- which indeed may be the only 
option, if the series defining the quantum cup-product turns out to diverge. 
Semi-simplicity confines $H^\bullet(X)$ to even degrees, because odd classes 
are necessarily nilpotent. (More is true: it turns out that semi-simplicity 
of the \emph{even part} $H^{ev}$ of the quantum cohomology ring forces the 
vanishing of odd cohomology  \cite{hmt}.)

The Gromov-Witten theory of $X$ is constructed as follows. Denote by 
$X_{g,\delta}^n$ the space of Kontsevich stable maps to $X$ with genus 
$g$, degree $\delta\in H_2(X)$, and $n$ marked points. We obtain maps 
\[
GW{}_{g,\delta}^n: H^\bullet(X)^{\otimes n} \to H^\bullet
\big(\ol{M}{}_g^n\big)
\] 
to the cohomology of Deligne-Mumford spaces $\ol{M}{}_g^n$ by pulling 
back  classes on $X$ via the evaluation map $X_{g,\delta}^n \to X^n$, 
and then integrating along the forgetful map $X_{g,\delta}^n \to \ol{M}{}_g^n$. 
This last step uses the virtual fundamental class of $X_{g,\delta}^n$. The 
degree of each  map $GW_\delta$, in the natural grading on $H^\bullet(X)$, 
is determined by the  relative (virtual) dimension of moduli spaces: 
 \begin{equation}\label{dimf}
  \deg GW_{g,\delta} = 2(\dim_{\bC} M_g^n - \dim_{\bC} X_{g,\delta}^n) = 
  2(g-1)\dim_\bC X - 2\langle c_1(X) | \delta \rangle. 
 \end{equation}
Summing over homology degrees $\delta$ yields a class 
\[
GW{}_g^n := \sum_\delta GW_{g,\delta}^n \cdot \mathrm{e}^\delta, 
\]
with coefficients in (a completion of) the group ring $\bQ[H_2]$, 
called the \emph{Novikov ring} of $X$. For a fixed $u\in H^2(X;\bC)$, sending 
$\mathrm{e}^\delta \mapsto \exp{\langle u | \delta\rangle}$  
furnishes a ring homomorphism $\bQ[H_2]\to \bC$, and subject to convergence 
we get a $u$-dependent family of complex cohomology classes $GW_{u,g}^n$. 
We recall in Section~2 below why this is equivalent to a family of DMT's 
$GW_u$, in the sense of \ref{classif}.iv. It is no accident that we 
obtain an entire family of DMT's: in fact, a general deformation construction 
(Definition~\ref{zedu} below) produces a family parametrised by all $u$ in 
(an open, or possibly formal subset of) $H^{ev}(X)$. Example~\ref{gromwitextend} 
spells this out in the case of Gromov-Witten cohomology.

\subsection{Gromov-Witten cohomology constraints.}\label{gwconst}
The theories $GW$ just described meet three additional constraints.
They are specifically traced to the use of \emph{ordinary} cohomology 
(for instance, they do not apply in this form to the exotic Gromov-Witten 
theories of Coates and Givental \cite{tomsasha}).
\begin{enumerate}
\item The \emph{Cohomological Field Theory} (CohFT) condition $D=\Id$;
\item The \emph{flat vacuum} condition: inserting the identity 
$\mathbf{1}\in A$ as the first input in $GW_u^n$ leads to the same class as
the pull-back of $GW_u^{n-1}$ along the first forgetful morphism 
$\ol{M}{}_g^n\to\ol{M}{}_g^{n-1}$;
\item \emph{Homogeneity} of the family $GW_u$ under the \emph{Euler vector 
field} $\xi$ on $H^{ev}(X)$. Along $H^2(X)$, $\xi$ is the constant 
vector field $c_1(X)$, but more generally
\[
\xi_u := c_1(X) + \sum (i-1) u_{2i}\quad\text{at}\quad u=\sum u_{2i}
\in\bigoplus\nolimits_i H^{2i}(X).
\]  
\end{enumerate}

\noindent In GW theory, condition (i) reflects the factorisation of the 
(virtual) fundamental class of $X_{g,\delta}^n$ at the boundary of 
Deligne-Mumford space \cite{giv2}. Condition (ii) is the base 
change formula in the square of forgetful morphisms\footnote{This is not 
altogether trivial, because the square is not quite Cartesian, due to the 
contractions of the universal curve. Again, it is conditioned by our 
use of ordinary cohomology.} 
\[
\begin{array}{ccc}
 X_g^n & \to  & X_g^{n-1}  \\
 \downarrow &   &  \downarrow \\
 \ol{M}{}_g^n & \to  &   \ol{M}{}_g^{n-1}
\end{array}
\]
Readers may know that (ii) implies the \emph{flatness of the identity} in 
the associated Frobenius manifold \cite[III]{man}; we will revisit this 
in \S\ref{vacdiffeq}. Finally, the homogeneity condition (iii), to be 
reviewed in more detail in \S\ref{eulerfield} (see also \cite[\S{I}.3]{man}), 
encodes the degrees \eqref{dimf} of the maps $GW_{g,\delta}^n$: see Example 
\ref{gromwitextend}. 

These constraints can be axiomatised in the setting of abstract DMT's,  
and imposing them narrows down the classification of semi-simple theories. 
In CohFT's, the operator $E$ of \S\ref{classif}.ii satisfies $E(\provomega)
\circ E^*(-\provomega) =\mathrm{Id}$. The flat vacuum condition determines 
the $\tilde{Z}^+$ (of \S\ref{classif}.i) from $E$, as in Proposition~\ref
{flatvac} below. Confirming a prediction of Givental's 
\cite{giv}, we will see that semi-simple CohFT's are determined by their 
\emph{genus-zero part}, the restriction to families of genus zero curves, 
save for an ambiguity related to the \emph{Hodge bundle}. (See \S\ref
{ambiginhom} for the precise statement.) Homogeneous theories (iii) have 
no such ambiguity, and we can then give an explicit \emph{reconstruction 
procedure} from the Frobenius algebra $A$ alone and the homogeneity constraint, 
as we explain after reviewing the following example.

\subsection{Example: the Manin-Zograf conjecture.} 
A simple illustration of the classification concerns the cohomological field 
theories of rank one: $\dim A =1$, so $A$ is necessarily semi-simple. (These 
theories are the units for a natural 
tensor structure on the category of CohFT's.) In this case, my classification 
affirms an older conjecture of Manin and Zograf \cite{manzog}: $\ol{Z}$ is 
an exponentiated linear combination of $\kappa$- and $\mu$-classes (the latter 
being the Chern character components of the Hodge bundle). The coefficients of 
the $\mu$-classes are easily related to those of $\log{E}(z)$ (\S\ref{manzogsect}), 
and this example illustrates nicely the ambiguities in reconstruction, as follows. 
Genus zero CohFT's of rank one are described using $\kappa$-classes alone, 
\cite[\S{III}.6]{man}, because the Hodge bundle is trivial in genus zero, 
where the $\mu$-classes are therefore invisible. On the other hand, the flat 
vacuum CohFT's are precisely those involving $\mu$-classes only (Proposition~\ref
{manzogconj}). Imposing all three conditions in \S\ref{gwconst} leads to 
$\tilde{Z}^+ =1$ and $E=\Id$, leaving only one choice: the Frobenius algebra 
structure on $A$, determined by the single complex number $\theta(1)$.

\subsection{Reconstruction from genus zero.}\label{18}
We now outline the reconstruction result of semi-simple homogeneous CohFTs from 
their underlying Frobenius algebra; full details are given in \S7 and \S8. 

For any CohFT $\ol{Z}$, a formal construction (Definition~\ref{zedu}) 
produces a family $\ol{Z}{}_u$ of CohFT's, parametrised by $u\in U$, an 
open (or formal) nieghbourhood of $0\in A$. In Gromov-Witten theory, the 
$H^2$ part of this family was described in \S\ref{gromwit}. The Frobenius 
algebra structure on  $A$ varies in this family and leads to a so-called 
\emph{Frobenius manifold} structure on $U$; see \S\ref{frobman} below for 
a minimal discussion, or \cite{dub, man, leepan} for an extensive one. A 
reconstruction theorem \cite{man} determines the genus-zero part of the CohFT 
from this Frobenius manifold. This fact has no known  analogue for the 
higher-genus part of the theory, largely because the cohomologies $H^\bullet
\big(\ol{M}_g^n\big)$ are unknown. 

However, for semi-simple theories, Givental \cite{giv} conjectured a formula 
for the classifying datum $E$ from genus-zero information. (The conjecture 
was framed in the sightly different setting of \emph{potentials}, to be recalled 
in \S\ref{potential} below.) Specifically, he characterised $E$ by a system of 
linear ODE's on $U$ (Dubrovin's \emph{first structure connection}), and from 
there, the homogeneity constraint \S\ref{gwconst}.iii led to a unique solution. 
In the final sections of this paper, I verify the ODE's for $E$ in the abstract 
setting of CohFT's (along with a companion ODE for $\tilde{Z}$) and conclude 

\begin{maintheorem}[1]
A semi-simple Cohomological Field theory satisfying the homogeneity constraint 
\ref{gwconst}.iii is uniquely and explicitly reconstructible from genus zero 
data. For homogeneous theories with flat vacuum, the Euler vector field and 
the Frobenius algebra structure suffice for reconstruction.
\end{maintheorem}

\noindent 
Reconstruction takes the form of a recursion for the Taylor coefficients 
of $E(\provomega)= \sum_k E_kz^k$. Let us spell this out in Gromov-Witten cohomology, 
when $A= H^\bullet(X)$ with the quantum cup-product at some point $u\in H^{ev}(X)$, 
assumed to define a semi-simple multiplication. Denote by $\mu$ the shifted 
degree operator $(\deg -\dim_\bC(X))/2$ on $A$, and by $(\xi\cdot_u)$ that of quantum 
multiplication by the Euler vector $\xi_u$ at $u$. Then, the recursion 
\[
 \left[(\xi\cdot_u), E_{k+1}\right] = (\mu+k)\cdot{E}_k
\]
determines $E(z) $ uniquely form $E_0=\Id$. (See the proof of Theorem~\ref
{gwreconst}.) Thus, all Gromov-Witten classes $GW_{g,\delta}^n\in H^\bullet
(\ol{M}{}^n_g)$ are  constructible from $c_1(X)$ and the quantum multiplication 
operator $(\xi\cdot_u)$ at a single (semi-simple) point $u$.

\begin{remark}
The series $E(\provomega)$ has an interpretation already flagged 
by Dubrovin \cite{dub}. Namely, the formal expression $E(z)\cdot\exp
(-\xi\cdot_u/z)$ gives the asymptotics at $z=0$ of solutions of an ODE with 
irregular (quadratic) singularities there (see \S\ref{frobreconst}). In 
the case of quantum cohomology, genuine solutions have unipotent, but 
non-trivial monodromy around $0$. (The monodromy logarithm is the 
operator of classical multiplication  by $c_1(X)$, cf.~\cite{dub}, 
which does not vanish for manifolds with semi-simple quantum cohomology.) 
Because the asymptotic formula is single-valued, it cannot represent 
a genuine solution and so the series $E(z)$ cannot converge. This makes 
the prospect of expressing $E$ in terms of immediate geometric data 
of the symplectic manifold problematic; this last question is very much 
open. 
\end{remark}

\subsection{Potential of a DMT.} \label{potential}
Let $\ol{Z}_g^n: A^{\otimes n} \to H^\bullet(\ol{M}{}_g^n)$ be the 
class associated by the DMT to the universal stable curve over the 
Deligne-Mumford space $\ol{M}{}_g^n$. The \emph{primary invariants} 
are the integrals of the $\ol{Z}$'s on the $\ol{M}{}_g^n$'s. However, 
$\ol{M}{}_g^n$ also carries the Euler classes $\psi_1, \dots, \psi_n$ 
of the cotangent lines to the universal curve at the marked points, 
and more information about $\ol{Z}$ is recovered by including $\psi$'s 
before integration. The resulting numbers are encoded in a generating 
series, the \emph{(ancestor) potential}, a function of a series $x(\provomega)
= x_0 +x_1\provomega + \dots \in A\bbrak{\provomega}$: 
\begin{equation}\label{pot}
\cA(x) = \exp\left\{\sum_{g,n} \frac{\hbar^{g-1}}{n!}
	\int\nolimits_{\ol{M}{}_g^n} \ol{Z}{}_g^n\left(x(\psi_1), \dots,
	x(\psi_n)\right)\right\}; 
\end{equation}
the sum excludes the values $(g,n) = (0,0), (0,1), (0,2)$ and $(1,0)$ 
for which $\ol{M}$ is not an orbifold. The series in \eqref{pot} need not 
converge analytically, but converges at least formally as a power 
series in $\{x,\hbar, x^3/\hbar\}$; so its exponential is well-defined 
in \emph{some} space of $\bC\ppar{\hbar}$-valued functions of $x$.

\begin{example}\label{trivexample}
The \textit{trivial} $1$-dimensional theory has $A=\bC$ and $\ol{Z} 
=\mathrm{1}$ for all $g$ and $n$; the integrand is ${x}(\psi_1)\wedge 
\dots\wedge x(\psi_n)$ and $\cA$ is the \emph{$\tau$-function} of 
Kontsevich and Witten. 

More generally, any Frobenius algebra $A$ can be coupled to the trivial 
cohomological field theory, by letting each ${}^q\ol{Z}{}_g^p$ be the 
degree zero class specified by the surface of genus $g$ with $p$ 
inputs and $q$ outputs. The potential is then expressible in terms of 
Kontsevich integrals. 
\end{example}

The potentials $\cA_u$ corresponding to the family  $GW_u$ of Gromov-Witten 
cohomology theories of a compact symplectic manifold are parametrised by 
$u\in A= H^{ev}(X)$ --- or in a formal version thereof, since the convergence 
question seems open in general. They are known as the \emph{ancestor GW 
potentials} of $X$. Their relation to the more customary \emph{descendent 
potential}, defined using the $\psi$-classes and integration over the 
spaces $X_{g, \delta}^n$, was determined\footnote{For clarification,  
recall that the descendent potential carries additional ``calibration'' 
information from the $1$-point, or \emph{J-} function, a choice of solution 
of the quantum ODE, which is not contained in our notion of a CohFT.} by 
Kontsevich and Manin \cite{km2}. The ancestor-descendent relation was reframed 
by Givental in the setting of loop group actions, which we now recall. 
 
\subsection{Givental's loop group conjecture.}\label{loopgroup}
For clarity, let us focus here on Cohomological Field Theories ($D=\Id$), 
postponing discussion of the general case until \S\ref{quadhamilt}. Let 
$\BF\ppar{\hbar}$ be the space of $\bC\ppar{\hbar}$-valued polynomials 
on $A\bbrak{\provomega}$; the potentials $\cA$ in \eqref{pot} live in some 
completion of this, such as the space of power series described in \S\ref
{potential}. (The choice of completion is not material, as our constructions 
and group actions will reduce to recursively defined operations on power 
series coefficients; see~\S6.) Define a symplectic form on the space $A\ppar
{\provomega}$ of formal Laurent series, 
\[
\Omega(x,y) = \oint \beta\left (x(-\provomega), y(\provomega)\right)d\provomega, 
\]
using the Frobenius bilinear form $\beta$. We view $\BF\ppar{\hbar}$ as a 
Fock representation of the Heisenberg group $\BH$ built on $\{A\ppar
{\provomega}, \hbar\Omega\}$. The symplectic group $\BSp$ on $A\ppar
{\provomega}$ acts projectively on suitably chosen completions of $\BF\ppar{\hbar}$ 
by the intertwining metaplectic representation. (The Lie algebra acts on 
$\BF\ppar{\hbar}$, and the portion of the action which we will need is 
integrable on a space containing the potentials; see~\S6.) The Laurent series 
loop group $\GL(A)\ppar{\provomega}$ acts point-wise on $A\ppar{\provomega}$. 
Consider the following subgroups of $\BSp$:
\begin{itemize}
\item $\BSp_L:= \BSp\cap \GL(A)\ppar{\provomega}$, the symplectic part of 
$\GL(A)\ppar{\provomega}$;
\item $\BSp_L^+:= \BSp\cap (\Id +\provomega\cdot\End(A)\bbrak{\provomega})$. 
\end{itemize}
The term ``symplectic loop group" is sometimes used for $\BSp_L$, 
but it really is the \emph{twisted form} of the loop group of $\GL(A)$. 
The subgroup $\BSp_L^+$ contains the matrix series $E(\provomega)$ of \S\ref
{gwconst}. In \cite{giv, giv3}, Givental described the Kontsevich-Manin relation 
between descendent and ancestor potentials of Gromov-Witten theory in terms 
of the action of $\BSp_L$, without assuming semi-simplicity. In addition, 
he proposed (and proved, for toric Fano manifolds) a formula for the value 
of the GW ancestor potential $\cA_u$ at a semi-simple point $u$ of quantum 
cohomology. This was formulated in terms of the action $\BSp^+_L$, using 
ingredients which appear in Dubrovin's isomonodromy description \cite{dub} of 
semi-simple Frobenius manifolds. 

Call $A^{DM}$ the subspace of those vectors in the cohomology $\prod_{g,n} 
H^\bullet(\ol{M}{}_g^n; (A^*)^{\otimes n})$ of all Deligne-Mumford spaces 
which are invariant under the symmetric groups. A DM field theory $\ol{Z}$ 
defines a vector in $A^{DM}$, by restricting to surfaces with no output points. 
(Furthermore, if the nodal factorisation rule $D$ is specified, $\ol{Z}$ is 
in turn determined by this vector.) A distinguished vector $I_A\in\prod H^0$ 
represents the trivial CohFT based on $A$. Let $\BH^+, \BH^{++}$ denote 
the subspaces $\provomega{A}
\bbrak{\provomega}$ and $\provomega^2{A}\bbrak{\provomega}$ in the 
Heisenberg group $\BH$, acting on $\BF\ppar{\hbar}$ by translation. In 
\S\ref{quadhamilt}, we describe an action of $\BSp^+_L\ltimes \BH^+$ 
on $A^{DM}$, which lifts the metaplectic and translation actions on 
potentials.\footnote{A construction along similar lines was alluded 
to in \cite{cks}.} Let $T_x$ denote the translation by $x\in\BH^+$, 
$(T_x\cF)(y) = \cF(y-x)$, and write $T_z$ short for $T_{z\mathbf{1}}$, 
for the unit $\mathbf{1}\in A$. My classification of DMT's will imply the 
following. 

\begin{maintheorem}[2]
The CohFT's with underlying semi-simple Frobenius algebra $A$ constitute 
the $\BSp_L^+\ltimes \BH^{++}$-orbit of the trivial theory $I_A$. 
The theories with flat vacuum form the orbit of the subgroup 
$T_\provomega\circ\BSp_L^+\circ{T}_{\provomega}^{-1}$.
\end{maintheorem}

\noindent The group element of $\BSp_L^+\ltimes \BH^{++}$ taking $I_A$ to the 
theory with classification data $\{A, E(z),\tilde{Z}{}^+\}$ in \S\ref{classif} 
is $E(z)\cdot \zeta$, with 
\[
\zeta = \provomega\exp\left(-\sum\nolimits_{j>0}a_j\provomega^j\right)
	-\provomega \in \BH^{++}. 
\]
This formula is closely related to the coordinate changes studied by 
Kabanov and Kimura\footnote{I am grateful to V.~Tonita for pointing 
this out.} \cite{kabkim}. 

Note that this $\zeta$ contains no $\provomega$-linear 
term. Adding a term $\zeta_1\provomega$, with $\zeta_1 = \sum_i\zeta_{i1}P_i$ 
turns out to change the structure constants $\theta_i$ of $A$, scaling 
them by $(1+\zeta_{i1})^2$  (Proposition~\ref{action}). Every 
complex semi-simple Frobenius algebra can be obtained in this way from 
a sum of copies of the trivial one, $\bC$ with $\theta(1)=1$. It is 
tempting to say that all semi-simple CohFT of the same rank constitute a 
single $\BSp_L^+\ltimes \BH^{+}$-orbit, but there is trouble when 
some $\zeta_{i1}=-1$: in other words, the action of the linear modes $zA\in\BH^+$ 
on $A^{DM}$ has some singularities, so this re-formulation of the first 
part of Theorem~2 requires some care.

Translation by $\provomega$ is the \emph{dilaton shift} of the literature; 
it encodes the expression of $\zeta$ from $E$ in flat vacuum theories. With 
a general vacuum vector $\mathbf{v}(z)$ (as in \S\ref{morevac}), 
we are instead looking at the set $T_{z\mathbf{v}(z)}\circ\BSp_L^+\circ{T}_
{\provomega}^{-1} (I_A)$; cf.~\S\ref{flatidagain}. Even more generally, 
abandoning the CohFT condition to allow $D\neq \Id$ enlarges the space 
of DM theories to the orbit of a larger subgroup $\BSp^+\subset 
\BSp$; this requires a slightly different setup and will be discussed 
in \S\ref{quadhamilt}, where the proof of Theorem~2 is completed.

\begin{remark}
The translation action of $\BH^+$ on the space of CohFT's has an analogue 
for the zero-modes $A\in A\bbrak{\provomega}$: this leads to the 
Frobenius manifold mentioned in \S\ref{18}. The interaction of these translations 
with the group $\BSp^+_L\ltimes\BH^+$ is rather complicated, given by a system of 
ODE's which we will derive in \S\ref{basdiff}. For instance, $A$-translations 
and $\BH^+$-translations do not commute. In the setting of open-closed 
theories, translation along the Frobenius manifold and that by $\BH^+$ 
correspond to deformations of the TFT coming from independent sources: 
to wit, deformation of the category of boundary states, versus deformation 
of the cyclic trace.\footnote{Unfortunately, the author does not know of a 
written reference detailing this point of view.}  
\end{remark}

\section{Field Theories from universal classes}\label{1}
We now review the definitions of FTFTs from the perspective of 
classifying spaces of oriented surface bundles. In the process, we 
complete the definition of Cohomological Field theories; however, the 
list of axioms for more general DM theories is only completed in 
\S\ref{nodalrelations}, after listing some explicit conditions.

We may switch between oriented topological, smooth, metric and Riemann 
surfaces as convenience dictates, because these structures are related by 
contractible spaces of choices (the spaces of Riemannian metrics, or 
metrics up to conformal equivalence), so their classifying spaces --- 
the bases of universal surface bundles --- are homotopy equivalent. 
Similarly, we can describe boundary circles more economically as follows. 

\subsection{Points versus boundaries.}\label{bdpoints}
Call a surface $(m,n)$-pointed if it carries a set of $m+n$ distinct unordered 
points, separated into $m$ \emph{incoming} and $n$ \emph{outgoing} ones. 
Given a vector space $A$, the base $X$ of an $(m,n)$-pointed surface bundle 
$\Sigma_X$ carries local systems $A^{(m)}, A^{(n)}$ with fibres $A^
{\otimes{m}}, A^{\otimes{n}}$, permuted by the monodromy in the base. 
Removing open disks centred at the special points shows that, up to 
a contractible space of choices, points contain the same information 
as un-parametrised boundary circles. Moreover, since 
$\mathrm{Diff}_+(S^1)$ is homotopy equivalent to its subgroup of rigid 
rotations, we may capture the parametrisation information, again up to a 
contractible space of choices, by specifying unit tangent vectors, or 
tangent rays. More precisely, there is a torus bundle $\tilde{X}
\twoheadrightarrow X$ with fibre $\bT^m\times\bT^n$, the product of unit 
tangent spaces at the special points.\footnote{$\tilde{X}$ is a \emph
{principal} bundle only if there is no monodromy, that is, if the special 
points can be ordered over $X$.} Up to homotopy, $\tilde{X}$ parametrises 
the surfaces in the family $\Sigma_X$, together with all parametrisations 
of their boundary circles.

\subsection{FTFT's reviewed.} 
Let us recall the functorial definition of FTFT's, and then convert 
the data to a collection of cohomology classes on the classifying 
spaces of surface bundles. This is especially necessary for DMT's, where 
we must formulate the nodal factorisation and vacuum axioms mentioned 
in \S\ref{classif}.iv.

\begin{itemize}\itemsep0ex
\item A \textit{family TFT with fixed boundaries} and coefficients in $A$ 
assigns to each family $\Sigma_X \mapsto X$ of closed oriented 
$(m,n)$-pointed surfaces a class 
\[
\tilde{Z}(\Sigma_X) \in H^\bullet(\tilde{X}; \Hom(A^{(m)}, A^{(n)})).
\] 
This must be functorial in $\tilde{X}$ and subject to the condition that 
sewing together any collection of incoming-outgoing boundary pairs 
gives the corresponding composition of linear maps. 
\item In a \emph{free boundary FTFT}, the class $Z(\Sigma_X)$ lives in $H^\bullet
(X; \Hom(A^{(m)}, A^{(n)}))$, is functorial in $X$, and the sewing 
condition must hold for \emph{any} given identification over $X$ of an 
incoming-outgoing boundary pair. 
\item A \emph{Lefschetz} FTFT assigns 
such $\ol{Z}$'s functorially to (chiral) Lefschetz fibrations of closed 
oriented pointed surfaces. 
\item Finally, a \textit{Deligne-Mumford FTFT} is a Lefschetz FTFT for 
\emph{stable} surfaces, satisfying a \emph{nodal factorisation rule} and 
a \emph{vacuum axiom}. We describe these in \S\S\ref{dmfactrule}--\ref{vacax} 
below, after introducing the universal classes ${}^p\ol{Z}{}^q$. 
\end{itemize}

\begin{remark}\label{stab}
\begin{enumerate}\itemsep0ex
\item Single surfaces define a commutative Frobenius algebra structure on $A$.
\item ``Sewing'' of pointed surfaces in a family is well-defined, up to 
homotopy, from an identification of tangent rays at the matched points.
\item As usual, nodes and special points must avoid each other. 
\item \textit{Stability} of surfaces leads to an \emph{orbifold} 
description of the moduli of nodal surfaces, but this does not play a 
role here. More important is the connection with Gromov-Witten theory, 
which forces us into the setting of Deligne-Mumford spaces and cohomological 
field theories. The classification of semi-simple theories remains unchanged 
for Lefschetz theories, which allow \textit{pre-stable curves}.
\end{enumerate}
\end{remark}

\subsection{Reformulation using universal classes.} 
Let ${}^q{M}_g^{p}$ denote the classifying space of the universal 
surface with $p+q$ distinct ordered points, and denote by ${}^q\tilde
{M}{}_g^p$ (or alternatively, ${}_q{M}_{g,p}$, as is common in the 
literature) the principal torus bundle defined by all choices of tangent 
rays at those points. Functoriality reduces a fixed-boundary FTFT to the 
specification of universal classes 
\[
{}^q\tilde{Z}_g^p\in H^\bullet_{\frS_p\times\frS_q}
	\left({}^q\tilde{M}_g^p; \mathrm{Hom}(A^{\otimes p}; 
	A^{\otimes q})\right), 
\]
where the symmetric groups $\frS_p,\frS_q$ act on ${}^q{M}_g^{p}$ by 
permuting marked points and simultaneously on $A^{\otimes{p,q}}$ by 
permuting the factors. Over $\bC$, equivariance under finite groups 
simply means invariance. With free boundary theories, we obtain classes 
${}^q Z_g^p$ over ${}^q M_g^p$, and in the case of DM theories, 
${}^q\ol{Z}{}_g^p$ over the Deligne-Mumford compactifications ${}^q\ol{M}
{}_g^p$.

The classifying space for the universal Lefschetz fibration has model which 
is perhaps less familiar, as a finite-dimensional complex algebraic Artin 
stack ${}^q\ol{A}{}_g^p$ of infinite type, classifying nodal curves with 
arbitrary chains and trees of rational curves. This has an infinite descending 
normal-crossing stratification, reflecting the unlimited bubbling that can 
occur in families. 

\subsection{Sewing conditions.} \label{propdesc}
Sewing two specified boundary components together defines maps, uniquely 
up to homotopy (with $x = x'+x''$ for $x=g,p,q$)
\begin{equation}\label{propmap}
s: {}^{q'}\tilde{M}_{g'}^{p'} \times {}^{q''}\tilde{M}_{g''}^{p''} 
	\to {}^{q-1}\tilde{M}_{g}^{p-1},
\end{equation}
and similar maps where several pairs of boundaries are simultaneously 
identified. The FTFT sewing condition is 
\[
s^*\left( {}^{q-1}\tilde{Z}^{p-1}\right) = {}^{q'}\tilde{Z}^{p'}
		\circ {}^{q''}\tilde{Z}^{p''},
\]
with composition of the appropriate entries. Self-sewing in a family of 
single surfaces is also permitted, but the result could be re-expressed 
by means of sewing on elbows.

Free boundary FTFT's are different, in that the sewing maps \eqref{propmap} 
does not descend to the base moduli spaces $M, M', M''$ for surfaces with 
free boundaries: sewing requires an identification of the boundaries. A  
natural circle bundle $\pi: \partial{N} \twoheadrightarrow M' \times M''$ 
parametrises the possible identifications. This $\partial{N}$ is also 
(the pull-back to $M'\times{M}''$ of) the circular neighbourhood of a 
divisorial boundary stratum  in $\ol{M}$, image of $M'\times{M}''$ under 
a \emph{boundary map} (see \eqref{dmop} below). Functoriality 
stipulates that, after contracting out the $A,A^*$ factors 
of the two sewing indices, the pull-back $\pi^*({}^{q'}Z_{g'}^{p'} \times 
{}^{q''} Z_{g''}^{p''})$ agrees with the restriction of the class 
${}^{q-1}Z_{g}^{p-1}$ to $\partial{N}$. 

\subsection{Nodal factorisation in Lefschetz theories.} \label{nodexplained}
Every Lefschetz theory carries a \emph{nodal factorisation rule}, which describes 
$\ol{Z}(\Sigma)$, for a nodal surface family $\Sigma$, in terms of the 
normalised family $\tilde{\Sigma}$. This rule is a consequence of the sewing 
condition: cutting out the pair of crossing disks near a chosen node expresses 
$\ol{Z}(\Sigma)$ as a contraction of $\ol{Z}(\tilde\Sigma)$ with the crossing 
disk family. Functoriality describes the latter by a universal formula in the 
Euler classes of the two tangent spaces at the node. Thus, for of a pair of 
marked points of opposite type, the relevant operator is the nodal propagator  
$Z(\nodprop) = D(-\omega_+,\omega_-)\in \End(A)\bbrak{\omega_\pm}$ mentioned 
in \S\ref{classif}.iii. 
Similarly, the effect of attaching two output points of $\tilde\Sigma$ 
into a node is controlled by a bilinear form $B$ on $A$ with values 
in $k\bbrak{\omega_\pm}$, while inputs involve a co-form $C\in (A\otimes A)
\bbrak{\omega_\pm}$.

The tensors $B,C$ and $D$ are not independent: each of them determines 
the other two, by a formal game with connecting elbows. In addition, $B$ and $C$ 
are symmetric under a switch of the two disks, and this also translates into a 
symmetry constraint on $D$. We will list the explicit formulae in \S\ref{BCD} 
below.

\begin{remark}
 When lifted from $M'\times{M}''$ to $\partial{N}$, the nodal 
factorisation law  becomes precisely the smooth surface sewing axiom, by virtue 
of the identity $D(-\omega,\omega)=\Id$.   
\end{remark}

\subsection{Deligne-Mumford factorisation rules.}\label{dmfactrule}
The nodal factorisation condition can also be formulated in a DMT, but as 
the pair of crossing disks is an unstable surface, the cutting argument used 
to derive it in Lefschetz theories is no longer valid. We therefore adopt 
these rules as an additional DMT axiom, and now spell them out.
 
Universally on Deligne-Mumford spaces, attaching marked points 
define the following boundary morphisms, differing only in the type 
of the attaching points:
\begin{eqnarray}\label{dmop}
&& b_2^D:{}^{q'}\ol{M}_{g'}^{p'} \times {}^{q''}\ol{M}_{g''}^{p''} 
	\to {}^{q-1}\ol{M}_{g}^{p-1}, \qquad 
	b_1^D:{}^{q}\ol{M}_{g}^{p} \to {}^{q-1}\ol{M}_{g+1}^{p-1},\nonumber \\
&&b_2^C:{}^{q'}\ol{M}_{g'}^{p'} \times {}^{q''}\ol{M}_{g''}^{p''} 
	\to {}^{q}\ol{M}_{g}^{p-2}, \qquad 
	b_1^C:{}^{q}\ol{M}_{g}^{p} \to {}^{q}\ol{M}_{g+1}^{p-2},\\
&&b_2^B:{}^{q'}\ol{M}_{g'}^{p'} \times {}^{q''}\ol{M}_{g''}^{p''} 
	\to {}^{q-2}\ol{M}_{g}^{p}, \qquad 
	b_1^B:{}^{q}\ol{M}_{g}^{p} \to {}^{q-2}\ol{M}_{g+1}^{p}. \nonumber
\end{eqnarray}
(There are corresponding maps for the Artin classifying stacks $^q\ol{A}{}^p$ 
of Lefschetz fibrations.) 

DMT's are required to satisfy a factorisation rule under each of these maps, 
involving contraction with specified tensors $B, C$ and $D$. Thus, for $b_2^B: 
{}^1\ol{M}{}_{g'}^{n'}\times {}^1\ol{M}{}_{g"}^{n"} \to \ol{M}{}_g^{n}$, we 
require  
\begin{equation}\label{Bcont}
\big(b_2^B\big){}^*\ol{Z}{}_g^{n} = 
	B(\omega',\omega")\left({}^1\ol{Z}{}_{g'}^{n'}\otimes 
	{}^1\ol{Z}{}_{g"}^{n"}\right),
\end{equation}
where $\omega', \omega''$ are the two Euler classes at the node; similarly 
for the other maps, with $D(-\omega',\omega'')$ and $ C(-\omega',-\omega'')$, 
respectively. (The choice of signs here is adapted to our later use of 
$\psi$-classes, in lieu of Euler classes.) 

The tensors $B,C,D$ should satisfy the consistency constraints already mentioned 
(and spelt out in \S\ref{BCD}), which are guaranteed in a Lefschetz theory. 
As it turns out, these constraints are also guaranteed in a semi-simple DM 
theory; so we could omit them from the axiom in this case. Refer to~\S\ref
{nodalrelations} below for more detail.

\begin{remark}
In the familiar case of CohFT's, we require that $D=\Id$, $B=\beta$, and 
$C$ is the inverse co-form. Factorisation rules with interesting $B$ appear 
in generalised-cohomology Gromov-Witten theory \cite{tomsasha} (although the 
dependence on $\omega',\omega''$ has a very special form there, $D$ is scalar). 
\end{remark}

\subsection{Vacuum axiom.}\label{vacax}
In a Lefschetz theory, a distinguished vector $\mathbf{v}(z)\in{A}\bbrak{z}$, 
defined by the universal sphere with a single output, has the following property: 
for any Lefschetz fibration $\Sigma_X$ with $n>0$ input points and the associated 
family $\Sigma'_X$ which ignores the first input, we have 
\[
\ol{Z}(\Sigma_X)\big(\mathbf{v}(\omega_1),x_2,\dots,x_{n}\big) =
	\ol{Z}(\Sigma'_X)(x_2,\dots,x_{n}),
\]
where $\omega_{1}$ is the the first input  Euler class. 
This is the smooth surface sewing rule at work. 

Again, this story fails in a DMT, so our final DMT axiom is the specification 
of a ``vacuum'' vector $\mathbf{v}(z)$, which must satisfy the following condition 
in the case of a CohFT. Let $\varphi: \ol{M}{}_g^{n} \to \ol{M}{}_g^{n-1}$ 
be the morphism of Deligne-Mumford spaces induced by forgetting the first marked 
point. Then, 
\[
\ol{Z}{}_g^{n+1}\big(\mathbf{v}(\omega_1),x_2,\dots,x_{n}\big) =  
\varphi^* \ol{Z}{}_g^n(x_2,\dots,x_{n}).
\]
The vacuum condition is more complicated in general DMT's with $D\neq\Id$, where 
it gets corrected by boundary terms. The reason is that $\varphi$ does \emph{not} 
classify the point-forgetting map on nodal surfaces: the universal 
curve over $\ol{M}{}_g^{n-1}$ lifts to a \emph{contraction} of the one over 
$\ol{M}{}_g^{n}$. This contraction turns out to be inoffensive for $\ol{Z}$ 
in a CohFT, but not so in general. We will not use the vacuum in general 
DMT's, so will not spell out the correction terms. 

Later, we will concentrate  on the special class of CohFT's with flat vacuum, 
when $\mathbf{v}=\mathbf{1}$.

\subsection{PROP description.} 
The sewing maps \eqref{propmap} assemble to a \emph{PROP structure} 
on the spaces ${}^q\tilde{M}_{g}^p$, which carries over to their homology. 
In this language, an FTFT is equivalent to an algebra over the homology PROP. 
Similarly, the Deligne-Mumford boundary morphisms \eqref{dmop} give a PROP 
structure on $H_\bullet(\ol{M}_g)$.  In this case, self-sewing of single 
surfaces enhances this to a \emph{wheeled PROP} (a notion introduced in 
\cite{mms}). Cohomological field theories are algebras over the associated 
homology PROP of DM spaces, but to capture the full CohFT structure, we 
must add a \emph{cyclic} structure, permuting inputs and outputs. (We 
lost the ability to switch inputs and outputs by means of elbows.) 
DMT's with general $D$ are algebras over a twisted form of the DM homology PROP. 
Free boundary FTFT's do not fit into PROP language, for the reason explained 
in \S\ref{propdesc}.

\subsection{Tautological classes.}\label{tautrefresh}
The classification will describe the various field theories in terms 
of the \emph{tautological classes} on the moduli of surfaces. We 
briefly recall the generating tautological classes on $\ol{M}{}_g^n$; 
those on $M_g^n, \tilde{M}{}_g^n$ are obtained by restriction.  
Let $\varphi: \ol{M}{}_g^{n+1} \to \ol{M}{}_g^n$ be the map 
forgetting the last marked point. The marked points define $n$ 
sections $\sigma_i$ of $\varphi$, with smooth divisors $[\sigma_i]$ 
as their images. Let $T^*_\varphi$ be the relative cotangent 
complex of $\varphi$ and define 
\[
\psi_i:= \sigma_i^* c_1(T^*_\varphi), \qquad 
	\kappa_j = \varphi_* \big(\psi_{n+1}^{j+1}\big),
\]
(where $\psi_{n+1}$ on $\ol{M}{}_g^{n+1}$ is defined using $\sigma_{n+1}$ 
and $\ol{M}{}_g^{n+2}$). These classes satisfy the relations 
\[\begin{split}
 [\sigma_i]\cdot\psi_i = &[\sigma_i]\cdot\psi_{n+1} = 0, \\
 \psi_i^k = \varphi^* \psi_i^k + \sigma_{i}{}_*(\psi_i^{k-1})&, 
\qquad  \kappa_j = \varphi^*\kappa_j + \psi_{n+1}^j.
\end{split}\]
The correction term $\sigma_{i}{}_*(\psi_i^{k-1})$ in the first 
relation is only visible on $\ol{M}{}_g^n$, but the one for 
$\kappa$ also appears on $M_g^n$. Thus defined, the $\kappa_j$ are 
\emph{primitive}: that is, under the boundary maps \eqref{dmop},
\[
b_2^*(\kappa_j) = \kappa_j' + \kappa_j'', \qquad b_1^*(\kappa_j) = \kappa_j.
\] 
Additional tautological classes on Deligne-Mumford spaces arise by the recursive 
pushing forward of polynomials in the $\kappa-$ and $\psi-$classes from boundary 
divisors.

\subsection{The stability theorems.}
The key to the classification are two stability theorems, due to Harer 
\cite{har} (later improved by Ivanov \cite{iv}), and to Madsen and Weiss 
\cite{madw}, respectively. For the first theorem, let $M_{g,m}^n$ be the base 
family of the universal surface of genus $g$ with $m+n$ ordered points, 
equipped with unit tangent vectors at the first $m$ special points. 

\begin{theorem}[``Harer stability" \cite{har,iv}]
The maps $M_{g,m}^n \to M_{g,m-1}^n$ and $M_{g,m}^n \to M_{g+1,m}^n$, 
defined (up to homotopy) by sewing in a disk, respectively by sewing 
on a two-holed torus, induce homology isomorphisms in degree less than 
$(g-1)/2$. 
\end{theorem}

\noindent
An important consequence describes the homological effect 
of adding marked points: 

\begin{corollary}[Looijenga, \cite{lo}]\label{lofree}
In the stable range of total degree $< (g-1)/2$, we have 
\[
H^\bullet(M_g^n;\bQ) \cong H^\bullet(M_g;\bQ)[\psi_1, \dots,\psi_n].
\]
\end{corollary}
\noindent We reproduce the easy proof. The circle bundle $\pi: M_{g,1} 
\twoheadrightarrow M_g^1$ presents $H^\bullet(M_{g,1};\bQ)$ as the cohomology 
of the differential graded algebra $\{H^\bullet(M_g^1;\bQ)[\eta], d\}$, with  
$\deg\eta=1$ and $d\eta =\psi_1$. Now, thanks to Harer, a right inverse 
of $\pi^*: H^\bullet(M_g^1;\bQ) \to H^\bullet(M_{g,1};\bQ)$ in the stable 
range is provided by the forgetful pull-back $H^\bullet(M_g;\bQ) \to 
H^\bullet(M_g^1;\bQ)$. Therefore, $\pi^*$ is onto, in the stable range. 
But then, $\psi_1$ is not a zero-divisor in that range: if $\psi_1x =0$, 
then $\eta{x}$ is a class which is not in the image of $\pi^*$. From the DGA 
presentation, we conclude that, in the stable range, 
\[
H^\bullet(M_g;\bQ) \cong H^\bullet(M_g^1;\bQ)/(\psi_1)
\]
so that $H^\bullet(M_g)[\psi_1]$ surjects and injects to $H^\bullet(M_g^1)$, 
giving the corollary for $n=1$. Repeat for the other $\psi$.

\begin{theorem}[``Mumford conjecture'' \cite{madw}]
In the stable range, we have  
\[
H^\bullet(M_g;\bQ) = \bQ[\kappa_j],\quad j=1,2,\dots.
\]
\end{theorem}

\subsection{Primitive and group-like classes.}\label{primclass}
We conclude by spelling out the role of $\kappa$-classes in our context. 
Genus-stabilisation $M_{g,m}^n \to M_{g+1,m}^n$ defines a limiting homotopy 
type $M_{\infty,m}^n$. This agrees with the classifying space of the \emph
{stable mapping class group} $\Gamma_{\infty, m}^n$ of a surface with 
$m$ fixed and $n$ free boundaries. Harer stability makes the fixed boundaries 
invisible in the homology of the classifying space, while the homological 
effect of free boundaries is described by Corollary~\ref{lofree}; so we 
focus on $M_{\infty,1}$. Sewing two surfaces, with one fixed boundary each, 
into a fixed pair of pants defines a map
\begin{equation}\label{monoidal}
m:{}M_{g,1}\times{M}_{h,1} \to M_{g+h,1},
\end{equation}
which gives a homotopy-commutative monoidal structure on 
$\coprod_g M_{g,1}$ and, in the limit, on $M_{\infty,1}$. The latter 
becomes a group-like topological monoid, and its cohomology $H^\bullet
(M_{\infty,1}; \bQ)$ acquires a (commutative and co-commutative) Hopf 
algebra structure.  By the Milnor-Moore theorem, this must be the 
free power series algebra in the \emph{primitive} cohomology classes, 
that is, the classes $x$ satisfying $m^*(x) = x\otimes{1} +1\otimes x$. 
The $\kappa$'s do have that property (\S\ref{tautrefresh}), so the 
Madsen-Weiss theorem has the following important consequence.

\begin{corollary} \label{primkappa}
All primitive rational cohomology classes on 
$M_{\infty, 1}$ are linear combinations of the $\kappa$'s.  
\end{corollary}

\begin{remark}
Corollary~\ref{primkappa} is equivalent to the rational Mumford 
conjecture. Madsen and Weiss prove an \emph{integral} version, 
identifying the homotopy type of the group completion of the 
topological monoid $\amalg{M}_g$ with the infinite loop space $\Omega^
\infty\bC\bP_{-1}^\infty$ of the \emph{Madsen-Tillmann spectrum} 
\cite{madtil}. An integral, in fact spectrum version of Looijenga's 
theorem was found earlier by B\"odigheimer and Tillmann \cite{botil}.
\end{remark}

Another important notion is that of a \textit{group-like} class 
$X\in H^\bullet(M_{\infty,1};\bQ)$, a non-zero class for which $m^*X = X 
\otimes{X}$. It is easy to see that the group-like classes are precisely 
those of the form $\exp(x)$, with primitive $x$.

\section{Smooth surface theories}

Armed with the boundary maps between the $M_g$ and the tautological 
classes, we proceed to classify  FTFT's of the first two types, involving 
smooth surfaces with parametrised or with free boundaries. This might be 
the place to confess to a minor gap in the classification: the definitions 
do not seem to determine the value of the universal $\tilde{Z}_g$ without 
marked points, although a valid choice can always be made from my data. For 
free boundaries, the same ambiguity applies to $Z_1^0$. This last, genus one 
problem persists for Lefschetz theories, but not for DMT's, since $\ol{M}_1$ 
does not exist.

\subsection{Fixed boundary theories.} \label{fixedbd}
With $g=g'+g''$, consider the effect on $\tilde{Z}$-classes of the operation 
of sewing onto the general surface of genus $g'$ a fixed $2$-holed surface of 
genus $g''$: 
\[
{}^1\tilde{Z}_{g} = \alpha^{g''}\cdot{}^1\tilde{Z}_{g'} \quad\text{on } {}^1\tilde{M}_{g'}
\] 
where $\alpha\in A$ is the Euler class of \S\ref{ss=invert}, and the left-side 
class has been restricted to ${}^1\tilde{M}_{g'}$. 
When $\alpha$ is invertible, it follows that $\alpha^{-g}\cdot{}^1\tilde{Z}_g$ 
stabilises, as $g\to\infty$, to a class $\tilde{Z}^+\in H^\bullet(M_\infty; A)$.  
The sewing axiom, applied to the multiplication map \eqref{monoidal}  
and corrected by the same power $\alpha^{-(g+h)}$ on both sides, implies 
that $\tilde{Z}^+$ is \textit{group-like}. It follows that 
\[
\tilde{Z}{}^+ = \exp\left\{\sum\nolimits_{j>0} a_j\kappa_j\right\}, 
\quad\text{ for certain } a_j\in A.
\]
We have used the superscript ``$+$'' to flag the lack of a $\kappa_0
$-contribution, present in the classes $\tilde{Z}$.

\begin{remark}
Integrally, $\tilde{Z}{}^+$ would be a group-like class in the 
$A$-valued cohomology of $\Omega^\infty\bC\bP_{-1}^\infty$. Additively, 
there exist additional primitive classes, the Dyer-Lashof descendants of 
the $\kappa$'s \cite{til}; quite likely, analogous group-like classes exist 
as well. The new classes could perhaps be ruled out by imposing the FTFT axioms 
at \textit{chain level}.  
\end{remark}

Clearly, ${}^1\tilde{Z}_g$ is the restriction to ${}_1M_g$ of $\alpha^g
\cdot\tilde{Z}^+$; let us find ${}^m\tilde{Z}_g^n$. Sewing on large genus 
surfaces to one boundary allows us to assume $g$ is large, without loss of 
information. Map now ${}^1\tilde{M}_g$ to ${}^m\tilde{M}_g^n$ by sewing on 
to the universal surface a fixed sphere with $n+1$ inputs and $m$ outputs. 
This sphere determines the map ${}_mS_{n+1}: A^{\otimes(n+1)} \to A^{\otimes m}$, 
multiplication to $A$ followed by the $m$th co-power $A\to A^{\otimes m}$. 
Thanks to Harer, the map ${}^1\tilde{M}_g \to {}^m\tilde{M}_g^n$  is a homology 
equivalence in a range of degrees, so we detect ${}^m\tilde{Z}{}_g^n$ by 
pulling back to  ${}^1\tilde{M}_g$, where we see the result of feeding 
${}^1 \tilde{Z}_g$ as one of the inputs in ${}_mS_{n+1}$. Thus, 
\[
{}^m\tilde{Z}_g^n = {}_mS_1(\alpha^g \cdot\tilde{Z}^+ \cdot {}_1S_n)
\]
and we conclude the desired classification, with freely chosen elements 
$a_j = \sum_ia_{ij}P_i$ of $A$:

\begin{proposition}[Fixed-boundary FTFTs]\label{paramformula}
If $m, n$ do not both vanish, then the matrix for ${}^m\tilde{Z}_g^n$ is diagonal 
in the tensor monomials of the normalised canonical basis. All entries are null, 
save for those relating $p_i^{\otimes n}$ to $p_i^{\otimes m}$; these have 
the form 
\[
\theta_i^{\chi/2}\cdot \exp\left
	\{\sum\nolimits_{j>0} a_{ij} \kappa_j\right\}, 
\]
for fixed complex numbers $a_{ij}$. Each $p_i^{\otimes 0}$ stands for $1\in \bC$, if 
$m$ or $n$ (but not both) vanish. 

Finally, every choice of numbers $\{a_{ij}\}$ gives rise to an FTFT by this rule, 
if, in addition, we define $\tilde{Z}_g$ defined by summing the above expression over
 $i$. \qed
\end{proposition}

\begin{remark}\label{nopoints}
The argument fails when $m=n=0$, and the axioms don't seem to determine 
$\tilde{Z}_g$ for closed surfaces, except in the stable range of homology,  
where we can detect it by lifting to $M_{g,1}$. 
\end{remark}

\subsection{Free boundaries and $E$.}\label{freebd}
Restricting to surfaces with fixed boundaries determines a $\tilde{Z}$ 
as above. Let now ${}^1 Z_{g, 1}$ denote the lift of ${}^1 Z_g^{1}$ to  
${}^1M_{g,1}$. Recall that the latter is a circle bundle over ${}^1M_g^1$ 
and classifies surfaces with a fixed incoming boundary and a free outgoing 
one. Sewing a fixed surface of genus $g''$ into the fixed incoming 
boundary of the general surface over ${}^1 M_{g',1}$ tells us that
\[
{}^1 Z_{g, 1} \in H^\bullet\left({}^1M_{g, 1}; 
	\End(A)\right)\quad\text{restricts to}\quad {}^1 Z_{g', 1} 
	\circ(\alpha^{g''}\cdot\,)  \in  H^\bullet\left({}^1 
	M_{g',1}; \End(A)\right),
\]  
with $(x\cdot\,)$ denoting the operator of multiplication by $x\in A$. 
Again, we get a stable class 
\begin{equation}\label{eztilde}
{}^1Z^+_1(\kappa, \psi_+) := {}^1 Z_{g, 1} \circ (\alpha^{-g} \cdot\,)
	\in H^\bullet\left({}^1 M_{g,1}; \End(A)\right) 
	\quad\text{as}\quad g\to\infty,
\end{equation}
minding that the cohomology ring is freely generated by the $\kappa_j$  
($j>0$) and the class $\psi_+$ at the outgoing point. Similarly, switching 
the roles of the boundary circles defines a stable class 
\begin{equation} \label{eztilde2}
{}_1Z^{+,1}(\kappa, \psi_-) := \lim_{g\to\infty}
			(\alpha^{-g} \cdot\,)\circ{}_1 Z_g^1 
\end{equation}
Setting the $\kappa$'s to $0$ in \eqref{eztilde} defines a formal 
Taylor series $E(-\psi):= {}^1Z^+_1(0, \psi)\in \End(A)\bbrak{\psi}$.  

\begin {lemma}\label{38}
We have 
\[
{}^1 Z_1^+(\kappa,\psi_+) = E(-\psi_+)\circ(\tilde{Z}^+(\kappa)\cdot\,)
\quad \text{and}\quad{}_1 Z^{+,1}(\kappa,\psi_-) = 
(\tilde{Z}^+(\kappa)\cdot\,)\circ{E}(\psi_-)^{-1}.
\]
More generally, in any genus $g$, 
\[
{}^1 Z_g^1 = E(-\psi_+)\circ(\alpha^g\tilde{Z}^+(\kappa)\cdot\,)
\circ E^{-1}(\psi_-).
\]
\end{lemma}


\begin{proof}
Modify the sewing above by letting \textit{both} surfaces vary, while 
the sewing circle rotates freely. This takes place over  
\[
\partial{N}= {}^1 {M}{}_{g',1} \underset{\bT}\times {}_1 {M}{}_{g''}^1 
\]
where the circle $\bT$ simultaneously rotates the two boundaries being 
sewn together. The sewn surface is classified by a map $\partial{N} \to 
{}^1{M}{}_{g}^1$. Pull-backs to $\partial{N}$ being understood, we have
\begin{equation}\label{movingcircle}
{}^1 Z_{g'}^1\circ {}^1 Z_{g''}^1 =  {}^1Z{}_g^1.
\end{equation}
In a moment, we will proceed by fixing the incoming or outgoing 
boundaries, as convenient. In any case, $\partial{N} \to {}^1 
{M}{}_{g'}^1 \times {}^1 {M}{}_{g''}^1$ is a circle bundle with 
Chern class $-(\psi'+\psi'')$, using the $\psi$-classes at the node. 
On the total space $\partial{N}$, $\psi'' = -\psi'$, the common 
value representing the Euler class of the sewing circle. The Leray 
sequence and our knowledge of stable cohomology show that $H^\bullet
(\partial{N})$, below degree $(g''-1)/2$, is freely generated over 
$H^\bullet({}^1M_{g'}^1)$ by the $\kappa_j''$. Similarly, it is freely 
generated over $H^\bullet({}^1M_{g''}^1)$ by the $\kappa_j'$, below degree 
$(g'-1)/2$. Let now both $g'$ and $g''$ be as large as needed, and lift 
\eqref{movingcircle} to ${}_1M_{g,1}$; we obtain from \eqref{eztilde} 
and \eqref{eztilde2}, after cancelling powers of $\alpha$:
\begin{equation}\label{addpsi}
{}_1 Z^{+,1}(\kappa',\psi') \circ {}^1 Z^+_1 (\kappa'',-\psi') 
	= \tilde{Z}{}^+(\kappa).
\end{equation}
Using the relation $\kappa = \kappa' + \kappa''$ and the algebraic 
independence of $\kappa',\kappa'',\psi'$, we obtain the second formula 
in the lemma by setting $\kappa''=0$, and from there, the first formula by 
setting $\kappa'=0$.

For the final and more general formula, return to \eqref{movingcircle} 
and let only $g''$ be large. Fixing the incoming circle leads to  
\[
{}^1 Z_{g'}^1 =  {}^1Z{}_{g,1} \circ \big({}^1 Z_{g'',1}\big)^{-1}
\]
with both factors on the right now known. Minding that $\psi'=-\psi''$ 
gives the formula.
\end{proof}

\begin{proposition}[Free boundary FTFTs]\label{officialZ}
For $(g,m,n)\neq(1,0,0)$, we obtain ${}^n Z_g^m$ as follows: each input is 
transformed by $E^{-1}(\psi)$, with the respective $\psi$ class; 
the product of these is multiplied by $\alpha^g\cdot\tilde{Z}^+(\kappa)$, 
the result is co-multiplied out to  $A^{\otimes n}$, where each factor is 
transformed by the respective $E(-\psi)$. The unit $\mathbf{1}$ substitutes 
for the empty input, and the Frobenius trace $\theta$ is applied if 
there is no output. 
\end{proposition}
\begin{proof}
If at least one marked point is present, we can repeat the final argument in 
the proof of the previous lemma: 
for each output or input point, compose with a large-genus ${}_{ 1} Z_G^1$ or 
${}^1 Z_{G, 1}$, respectively, to arrive at the known operator ${}_m Z_{>G, n}$. 
Since $^1Z_G^1$  is invertible and known, we are done. The case $m=n=0$ is 
handled as follows. Pull back $Z_g$ along the forgetful map $\varphi:M_g^1 
\to M_g$. The universal closed surface bundle splits, when lifted to $M_g^1$, 
into an open surface and a disk sewed along their common (moving) boundary, 
and we can compute $\varphi^*Z_g$ from the known formulae to get the desired
\[
\varphi^*Z_g = \theta\left(\alpha^g\cdot\tilde{Z}^+(\varphi^*\kappa)\right),
\]
having used the primitivity of $\kappa$-classes. (More precisely, the $\kappa
$-classes of the unpointed disk precisely undo the $\varphi^*\kappa$-correction 
of \S\ref{tautrefresh}; see the discussion of the vacuum below, if more help 
is needed.) Now, when $g\neq1$, the map $\varphi^*$ is split in rational cohomology 
by integration against the $\psi$-class down to $M_g$, so we recover $Z_g$ as hoped.
\end{proof} 

\subsection{The vacuum.} \label{morevac}
The universal disk with outgoing boundary 
defines the \emph{vacuum vector} $\mathbf{v}(z)\in A\bbrak{z}$, where we take 
$z$ to be the \emph{opposite} of the boundary Euler class (and of the 
$\psi_{out}$ at the output, in the pointed sphere model). Let $\varphi: 
{}^pM_g^{q+1}\to {}^pM_g^q$ be the map which forgets the first input; 
capping the first boundary in the universal surface with a disk shows that 
\begin{equation}\label{charvac}
 {}^pZ_g^{q+1} (\mathbf{v}(\psi_1),\dots) = \varphi^* ({}^pZ_g^q)(\dots).
\end{equation}
Fixing the disk shows that $\mathbf{v} \equiv \mathbf{1} \pmod{z}$. We say 
that that $Z$ has \emph{flat vacuum} if in fact $\mathbf{v} = \mathbf{1}$. 

\begin{proposition}\label{flatvac}
In the semi-simple free boundary FTFT build from data $\{E,Z\}$, the 
vacuum is given by 
\[
\mathbf{v}(z) = E(z)\left(\exp\{-\sum\nolimits_{j>0} a_jz^j\}\right),
\]
and $Z$ has a flat vacuum precisely when $\exp\left\{-\sum_{j>0} 
a_jz^j\right\} =  E(z)^{-1}(\mathbf{1})$.
\end{proposition}

\begin{proof}
This is the formula in Lemma~\ref{38} together with the equality 
$\kappa_j = -(-\psi_{out})^j$ on ${}^1M_0$. One way to see the latter is 
to use the correction formula in \S\ref{tautrefresh} for the pull-back 
to ${}^1M_0^1$, on which space all $\kappa$'s vanish, and the two 
$\psi$-classes are opposite.  
\end{proof}

\begin{remark}\label{mum} 
Unlike Harer stability and Looijenga's result on $\psi$-classes, 
the Mumford conjecture has not seriously been used: in the discussion so far, 
the $\kappa$'s could have been replaced by the primitives in the Hopf algebra 
$H^\bullet(M_{\infty,1})$. However, later on, unknown primitive classes would 
break the argument for reconstruction from genus $0$. 
\end{remark}

\subsection{Comment on stable surfaces}
Deligne-Mumford theories, which we aim to classify, can be restricted to 
families of smooth curves and they define free boundary FTFTs for \emph
{stable} surfaces only; so we must track the role of stability in the 
arguments of this section. The discussion applies with two exceptions: 
the construction of the vacuum, and the determination of $Z_g^0$ (in the proof 
of Proposition~\ref{officialZ}). In a general DMT, the vacuum (specified 
in an extra axiom) can be used to determine $Z_g^0$. In semi-simple 
theories, we can detect $\mathbf{v}$ --- and establish its existence --- by 
going to large-genus surfaces in the contraction formula \eqref{charvac}; 
invertibility of $\tilde{Z}{}^+$ allows us to replicate the conclusion of 
Proposition~\ref{flatvac}. This helps explain why there will be no 
distinction later between classification of semi-simple Lefschetz and 
DM theories.

\section{Lefschetz and DM theories: construction}\label{dmsec1}

Restricting a Lefschetz theory to families of smooth surfaces gives 
a free-boundary theory. In the semi-simple case, this is parametrised by 
\begin{enumerate}
\item the Frobenius algebra $A$,
\item the class $\tilde{Z}^+ = \exp\big\{\sum_{j>0} a_j\kappa_j\big\}$ of 
\S\ref{fixedbd}, 
\item the formal Taylor series $E(\provomega)= \Id + \provomega{E}_1 + \provomega^2
{E}_2 +\dots\in \End(A)\bbrak{\provomega}$ of \S\ref{freebd}. 
\end{enumerate} 
New information arises from the universal pairs of crossing disks, in the 
form of 
\begin{enumerate}
\setcounter{enumi}{3}
\item the ``nodal propagator" $\ol{Z}(\nodprop) = D(-\omega_+,\omega_-)$  
and the companion quadratic tensors $B,C$ of \S\ref{nodexplained}, all of 
which are formal series in the Euler classes of the two disks. 
\end{enumerate}
Ingredients (iii) and (iv) are subject to consistency constraints. 
In particular, we will see that $B,C$ and $D$ determine each other in any 
Lefschetz theory, whether or not $A$ is semi-simple. After spelling out 
the constraints on $B,C,D$ --- and the compatibility condition with $E$, in 
the semi-simple case --- we will construct a Lefschetz theory from those 
data. Restriction to stable curves gives a DMT. Unlike the proof of 
uniqueness in the next section, the construction does not require 
semi-simplicity of $A$.

We will switch henceforth from the Euler classes $\omega$ of boundary circles 
to the $\psi$-classes \emph{at the node}, and in doing so must mind the signs: 
$\omega=-\psi$ at the center of an outgoing disk, but $\omega=\psi$ for an 
incoming one. We use $z$'s to denote universal $\psi$ classes.

\subsection{Relating on $B, C$ and $D$.}\label{BCD}
The discussion in this sub-section applies to any Lefschetz theory, not 
necessarily semi-simple. The pairing 
\[
B: A\otimes A\to k\bbrak{\provomega_{1,2}}
\] 
defined by two disks with incoming boundaries and crossing at their centres 
must be \emph{symmetric} under simultaneous swap of the $A$ factors and of the 
nodal $\psi$-classes $\provomega_{1,2}$. The same symmetry holds for the co-form 
output by two crossing disks,
\[
C\in (A\otimes{A})\bbrak{\provomega_{1,2}}.
\] 
Each of these pairs of disks can be constructed from $\nodprop$ and from the 
left or right elbows $\Subset$, $\Supset$. To simplify notation in converting to 
algebra, we use the Frobenius pairing $\beta$ to express quadratic tensors 
as endomorphisms: define $B'$ by $\beta(a_1, B'(a_2)) = B(a_1 \otimes a_2)$, 
and similarly define $Z'(\Subset)\in \End(A)\bbrak{z}$ from $Z(\Subset)$, 
with $z$ standing for the Euler class of the second input circle. We then have  
\begin{equation}\label{bilform}
B'(\provomega_1,\provomega_2) = Z'(\Subset)(\provomega_1)
		\circ{D}(\provomega_1, \provomega_2);
\end{equation}   
similarly, defining the operator $C'$ by $\beta(a_1, C'(a_2)) = 
\beta^{\otimes2}(a_1\otimes a_2, C)$, and $Z'(\Supset)(z)$ by the same rule 
(with $z$ the Euler class of the second circle) leads to
\begin{equation}\label{coform}
C'(\provomega_1,\provomega_2) = D(\provomega_2,\provomega_1)\circ 
			Z'(\Supset)(\provomega_1), 
\end{equation} 
and these endomorphisms must satisfy $B'(\provomega_2,\provomega_1) = B'
(\provomega_1,\provomega_2)^*$ and  $C'(\provomega_2,\provomega_1) = C'
(\provomega_1,\provomega_2)^*$. 

We can eliminate the operators $Z'(\Subset)$ and $Z'(\Supset)$ from formulas 
\eqref{bilform} and \eqref{coform}:

\begin{proposition} \label{relBCD}
We have
\[
Z'(\Subset)(\provomega) = B'(\provomega,-\provomega) 
	 = D^*(0,\provomega)\circ{D}^{-1}(\provomega,0) 
	 = C'(-\provomega, \provomega)^{-1} = Z'(\Supset)(-z)^{-1}.
\] 
\end{proposition}
\begin{proof}
The first identity arises by setting $\provomega = \provomega_2 = -\provomega_1$ 
in \eqref{bilform}: this results  in the specialisation $D(z,-z) =\Id$, 
as in \S\ref{classif}.iii. For the second, set one of the 
arguments to $0$ and the other to $\provomega$ in \eqref{bilform} to get 
\[
B'(0,\provomega)= D(0, \provomega), \quad B'(\provomega,0) = 
	Z'(\Subset)(\provomega)\circ D(0,\provomega),
\]
and now use the symmetry of $B$. The last two identities are reached by 
the same route, but using equation \eqref{coform}. Incidentally, equality of 
the outer terms in \eqref{relBCD} is the equivariant form of Zorro's lemma.
\end{proof} 

Equations \eqref{bilform}, \eqref{coform} and \eqref{relBCD} allow us to express 
$B,C$ and $D$ in terms of each other. In particular, note the equivalent \emph
{Cohomological Field theory constraints} $B'=\Id \Longleftrightarrow C'=\Id 
\Longleftrightarrow D=\Id$. 

\subsection{Deligne-Mumford data and constraints.}\label{nodalrelations}
In a DMT, the data $B,C$ and $D$ are supplied in the nodal factorisation axioms, 
controlling the behaviour of $\ol{Z}$-classes at the boundary of $\ol{M}{}_g^n$. 
The arguments of \S\ref{BCD} are now disallowed, because the elbows  $\Subset, 
\Supset$ are unstable surfaces. Instead, the relations between $B,C,D$ are 
imposed as consistency constraints on the data. (In the process of interpreting 
the formulas form the previous section, we \emph{define} $Z(\Subset), Z(\Supset)$ 
from $B$ and $C$ by Proposition~\ref{relBCD}.) Equivalent constraints can 
be formulated for each datum separately, as follows: 
\begin{enumerate}
\item symmetry of $B$; 
\item symmetry of $C$; 
\item the identity $D(z,-z) =\Id$, together with an awkward 
adjointness condition on $D$.
\end{enumerate}
(We shall not use this adjointness condition on $D$, and leave it to the 
reader to spell it out.) With these constraints, the list of axioms of a DMT 
is finally complete!

Let us note, on the side, that it is unnecessary to impose the constraints 
in semi-simple theories, as they can be inferred from the other axioms by a 
different method. Namely, one considers the universal curve with a single node 
and two components of large genus, each carrying a marked point, and computes 
its $\ol{Z}$-class in the three possible nodal factorisations, using  $B,C$ 
and $D$. This suffices to detect the identities of \S\ref{BCD} for the usual 
reasons: the surface operators are invertible, and the nodal $\psi$-classes 
are free algebra generators in the stable range. (The details of the argument 
closely parallel the proof of Lemma~\ref{38}, and are left to the reader.)

\subsection{Constraint on $E$, and the CohFT condition.}\label{gromwitconds}
Thanks to Proposition~\ref{officialZ}, in a semi-simple theory we have  
\[
Z'(\Subset)(z) =  E^{-1}(-\provomega)^*\circ{E}^{-1}(\provomega),
\]
whence equation \eqref{relBCD} gives the compatibility constraint between $E$ 
and each datum $B,C$ and $D$. For example, fixing a symmetric $B$ subjects 
$E$ to the constraint
\begin{equation}\label{symofB}
B'(\provomega,-\provomega) = 	E^{-1}(-\provomega)^* E^{-1}(\provomega),
\end{equation}
and determines $E$ up to left multiplication by any $\End(A)$-valued 
Taylor series $F(\provomega)=\Id + O(\provomega)$ which preserves 
the symplectic form on $A\ppar{\provomega}$
\[
\Omega_B(a_1,a_2) = \operatorname{Res}_{\provomega=0} B(-\provomega,\provomega)
	\left(a_1(-\provomega), a_2(\provomega)\right)d\provomega. 
\]

In particular, in a CohFT with $B=\Id$, \eqref{symofB} becomes the standard
\textit{symplectic condition} 
\[
E^*(\provomega) = E^{-1}(-\provomega),
\]
which says that  $E(\provomega)$ preserves the symplectic form
\[
\Omega(a_1,a_2) := \operatorname{Res}_{\provomega=0} 
	\beta\left(a_1(-\provomega)a_2(\provomega)\right)d\provomega.
\]
The constraint on $E$ must applies in any semi-simple DMT: we can detect 
the requisite identity in the tubular neighbourhood of a 
nodal stratum in large genus. 

\subsection{Alternative parameters for semi-simple theories.} 
The following alternative description will be useful in \S\ref{quadhamilt}.
Since $D(\provomega,-\provomega)\equiv\Id$, we can write 
\begin{equation}\label{altB}
C'(\provomega_1,\provomega_2) = E\left(\provomega_1)\circ(\Id + 
	(\provomega_1+\provomega_2)W'(\provomega_1,\provomega_2)\right)
	\circ{E}^*(\provomega_2)
\end{equation}
for a uniquely determined $W'$ satisfying the straight-forward 
symmetry constraint 
\[
W'(\provomega_1,\provomega_2)^*= W'(\provomega_2, \provomega_1),
\]
corresponding to a symmetric $W\in (A\otimes{A})\bbrak{\provomega_{1,2}}$.
The triple $(\tilde{Z}{}^+, W, E)$ will be an alternative set of parameters 
for a semi-simple DMT or Lefschetz theory, with symmetry of $W$ as the 
only constraint. For example, in these parameters, the CohFT condition becomes 
\[
 W'(\provomega_1,\provomega_2) = \frac{E^{-1}(\provomega_1)E^{-1}(\provomega_2)^*
 	-\Id}{\provomega_1+\provomega_2}, 
\]
which can be met precisely for symplectic $E$.

Finally, we give the promised construction of a Lefschetz theory with compatible 
data (i)--(iv).
   
\begin{proposition}\label{lefschetzconst} 
Given any Frobenius algebra $A$ and data $\tilde{Z}^+$, $E$ and $B$,   
subject to the constraint~\eqref{symofB}, there exists a Lefschetz theory 
with nodal bilinear form $B$, and which on smooth surface families is given 
by Proposition~\ref{officialZ}.
\end{proposition}
\begin{proof}
Here is a recipe to produce a field theory; for definiteness, we write it 
on $\ol{M}{}_g^n$ but $\ol{A}{}_g^n$ would work as well. For a single surface 
$\Sigma$, the smooth-surface and nodal factorisation rules leave no choice: 
resolve the surface, viewing all nodal points as outgoing say, then 
apply the free boundary formula to each component, and finally use 
$B(\psi',\psi'')$ to contract the two factors of $A$ at each node (formula 
\ref{Bcont}). Clearly, this recipe works in any family which does not vary 
the topological type of the surface, and in particular over any stratum of 
$\ol{M}{}_g^n$. However, patching these classes together when attaching 
the strata requires more comment.

For any boundary stratum, the recipe just given can also be applied to nearby 
smoothings of our nodal surface $\Sigma$. These smoothings have a distinguished 
handle which degenerates to the node; we can cut this handle and use, in contracting 
with $B$, the Euler class of the cutting circle, with the two choices of sign, 
in lieu of the nodal $\psi$-classes. Let us call this the \emph{nodal recipe}. 
The nodal recipe is unavailable as we move farther into the bulk of Deligne-Mumford 
space, where the handle is lost; the \emph{smooth recipe}, based 
on the true topology of the surface, must 
take over. Constraint \eqref{symofB} ensures the agreement of the 
smooth and nodal recipes, at the level of cohomology, in the region where 
both can be used. However, to produce a well-defined cohomology class on 
$\ol{M}{}_g^n$, we must exhibit \emph{cocycle-level representatives}, 
such as differential forms, for the local $\ol{Z}$-classes, and check 
their agreement on overlaps.  (Choose the overlaps to be (poly-) 
annular neighbourhoods of the Deligne-Mumford strata.) 

For this purpose, we choose differential forms $\tilde\psi$ representing the 
$\psi$'s over 
$\ol{M}{}_g^{n+1}$, such that:
\begin{enumerate}
\item $\tilde\psi_{n+1}$ vanishes near the sections $[\sigma_i]$ and near the 
nodes of the universal curve
\item The closed forms $\tilde\psi', \tilde\psi''$ at a node are also defined 
on a tubular neighbourhood of the locus of nodal curves, and $\tilde\psi'= 
-\tilde\psi''$ in an annular neighbourhood.
\end{enumerate}
This is possible because the line bundle $\det \sigma_{n+1}^* T^*_\varphi$ 
is trivial near the $[\sigma_i]$ and flat near the nodes, so its curvature 
forms in any metric which is constant near $[\sigma_i]$ and near the nodes 
will work.  

Apply now the nodal recipe for $\ol{Z}$ with differential forms, using 
$\int_\varphi \tilde\psi_{n+1}^{j+1}$ for each occurrence of $\kappa_j$ in 
the cohomological formula. Vanishing of $\tilde\psi_{n+1}$ near the nodes 
allows us to omit nodal neighbourhoods of the surface when computing the 
integral, and gives a well-defined differential form expression for $Z$ 
of the cut surface $\Sigma^{cut}$ with values in $A\otimes A$, over a small 
neighbourhood of the boundary of $\ol{M}{}_g^n$. We can then contract with 
$B(\tilde\psi',\tilde\psi'')$. In other words, we can continue to use the 
nodal recipe in a neighbourhood of any given boundary stratum, using $\Sigma^{cut}$  
and the forms $\tilde\psi', \tilde\psi''$ as substitutes for the boundary Euler 
classes. Moving now a little further away, into the annular neighbourhood where 
$\tilde\psi'= -\tilde\psi''$, constraint~\ref{symofB} shows that contraction with $B$ simply 
has the effect of cancelling the output $E(\tilde\psi)$-twists in the formula 
for $Z(\Sigma^{cut})$. As a result, the nodal and smooth recipes agree 
at the level of forms. This gives the desired patching.  
\end{proof}

\begin{remark}
Another construction of the classes $\ol{Z}$ will be given in \S\ref
{quadhamilt}, in terms of a  group action on cohomology of the 
Deligne-Mumford spaces.
\end{remark}

\subsection{The vacuum in Lefshetz theories.} \label{nodalvac}
Existence of a vacuum (\S\ref{vacax}) follows from the Lefschetz theory 
sewing rule. In the 
theory of Proposition~\ref{lefschetzconst}, $\mathbf{v}(z)$ is given by 
the formula of Proposition~\ref{flatvac}, and the \emph{flat vacuum} 
condition $\mathbf{v}(z)= \mathbf{1}$  amounts to
\begin{equation}\label{zfrome}
\exp\left\{-\sum\nolimits_{j>0} a_j\provomega^j\right\} = E^{-1}(\provomega)
(\mathbf{1}).
\end{equation}
In the semi-simple case, large genus surfaces detect the vacuum, so the 
restricted Deligne-Mumford theory will also have a flat 
vacuum precisely when \eqref{zfrome} holds.

\section {Deligne-Mumford theories: uniqueness}
\label{dmsec2}

This section contains the key argument of the paper: we show that semi-simple 
DMT's are uniquely determined by the nodal propagator $D$ and by the associated 
free-boundary theory on smooth curves.\footnote{In the context of chain-level 
theories, this fact is true without the semi-simplicity assumption; but in 
that situation, it can be made obvious with the right definitions.} 
The argument also applies to Lefschetz theories, but we focus on the DM case. 
A reformulation of the main result, suggested by one of the referees, 
is found in the appendix to this section. 

\subsection{Extending $Z$-classes over Deligne-Mumford strata.} 
\label{41}
Let $j: S\hookrightarrow M$ be the divisor parametrising a (locally versal) 
nodal degeneration of a family $\Sigma_M\to M$ of marked Riemann 
surfaces. The normal bundle $\nu_S$ to $S$ in $M$ is the tensor 
product $L'\otimes L''$ of the complex tangent lines at the two 
exceptional points $p',p''$ of the normalised surface $\tilde \Sigma$ 
over $S$; as to its Euler class, $\eul(\nu_S) = -(\psi'+\psi'')$. 

Since $p'$ and $p''$ may be switched by the monodromy over $S$, we 
view them both as outgoing. Over $S$, and hence over a tubular 
neighbourhood $N$, $Z(\Sigma)$ is the contraction of $Z(\tilde\Sigma)
\in H^\bullet(\partial{N}; A^{(2)})$ by $B(\psi',\psi'')$. The 
Mayer-Vietoris sequence
\[
\cdots\to H^{\bullet-1}(\partial{N}) \xrightarrow{\:\delta\:} 
	H^\bullet(M) \to H^\bullet(M\setminus N)\oplus H^\bullet(N)\to 
			H^\bullet(\partial{N})\xrightarrow{\:\delta\:}\cdots,
\]
shows that cohomology classes over $M\setminus S$ and $N$ patch into 
one over $M$, if they agree over the circular neighbourhood 
$\partial N$; but an ambiguity arises from the $\delta$-image of 
$H^{\bullet-1}(\partial{N})$. More precisely, if $\eta$ is a 
connection form on the circle bundle $\partial{N}\to S$, then 
$H^\bullet(\partial{N})$ is computed as the cohomology of the DGA 
$H^\bullet(S)[\eta]$, with differential $d\eta = \eul(\nu_S)$. 
For $a\in H^{\bullet-1}(\partial{N})$, $\delta(a)$ is given by the 
differential of any extension of $a$ to $N$ as a co-chain. This kills 
classes pulled back from $S$, while a class $b\eta$, with $b$ from $S$, 
is sent to $j_*(b)$. Now, $b\eta$ is a co-cycle iff $b\cdot\eul(\nu_S)=0$, 
so the patching ambiguity is precisely the Thom push-forward $j_*$ of 
the annihilator of $\eul(\nu_S)$ in $H^{\bullet-2}(S)$.

This observation applies to Deligne-Mumford strata $S$ of any 
co-dimension $c$: a class in $H^\bullet(M)$ with known restrictions 
to $M\setminus{S}$ and $S$ is ambiguous only up to addition of some 
$j_*(b)$, with $b\in H^{\bullet-2c}(S)$ annihilated by $\eul(\nu_S)$.
We see this from the long exact cohomology sequence
\[
\dots\to H^{\bullet-2c}(S) \xrightarrow{j_*} H^\bullet(M) \to 
	H^\bullet(M\setminus{S}) \xrightarrow{\delta} H^{\bullet-2c+1}(S) 
	\to\dots,
\] 
(where we have used the Thom isomorphism $j_*: H^{\bullet-2c}(S)\cong H^\bullet
(M,M\setminus{S})$) and from the fact that $j_*(b)|_S = \eul(\nu_S)\cdot {b}$. 
Note that $\eul(\nu_S)$ is the product of  Euler factors for the 
Deligne-Mumford divisors containing $S$.

\subsection{Uniqueness for large genus: the main idea.} \label{euler}
If $M\setminus{S}$ is the universal family of smooth surfaces of large 
genus and $S$ a boundary divisor in its DM compactification, Looijenga's 
theorem \eqref{lofree} ensures that $\eul(\nu_S) = -\psi' - \psi''$ is 
not a zero-divisor within a range of degrees, as one component of 
$\tilde\Sigma$ must have large genus. Classes then patch uniquely. 
This applies to strata of any co-dimension, and even if the family 
$M$ includes nodal and reducible surfaces, the only requirement 
being that each node defining the degeneration to $S$ should belong 
to at least \textit{one} large genus component. This is the germ 
of an inductive proof of unique extension of $Z(\Sigma_M)$ to the 
Deligne-Mumford boundary. The induction requires a careful 
stratification of the Deligne-Mumford spaces $\ol{M}{}^n_g$. 

\subsection{Stratification of $\ol{M}{}^n_g$.}\label{stratif}
Assume that $n>0$, and call the irreducible 
component of the universal curve containing the marked point $n$ 
\textit{special}. We now decompose $\ol{M}{}^n_g$ following the 
\textit{topological type $\tau$} of the special component. A partial 
ordering on the resulting strata is defined by stipulating that higher special
types can only degenerate to lower ones (plus extra components, which 
cease to be special). We extend this to some complete ordering; an example 
is the dictionary order on geometric genus, number of nodes and total number 
of marked points 
of the special component. (Nodes linking the special component to other 
components should be counted for this purpose as marked points, not nodes.) 
The smooth stratum $M^n_g$ is by itself. Every stratum in the decomposition 
is isomorphic to $(M_\gamma^\nu\times \ol{M})/F$, where $\gamma$ and 
$\nu$ pertain to the special component, while $\ol{M}$ parametrises the 
complementary components, and $F$ is the group of symmetries of the 
modular graph describing the topological type our curves. 

\begin{example}
With $g>2$ and $n=1$, if we split off an elliptic curve crossing 
the special component at two nodes, $\gamma = g-2, \nu = 2$, $\ol{M} 
= \ol{M}_1^2$ and $F = \bZ/2$, switching the two nodes. 
\end{example}

Our decomposition $M_\tau$ of $\ol{M}^n_g$ is not a stratification 
in the strict sense: it is not compatible with the dimensional 
ordering. However, we have the following:
\begin{enumerate}
\item Each $M_\tau$ is a union of Deligne-Mumford strata.
\item Every descending union $\amalg_{\tau'\ge \tau} M_{\tau'}$ of 
strata is open. 
\item Each $M_\tau$ is a closed sub-orbifold of $\amalg_{\tau'\ge \tau} M_{\tau'}$.
\item The normal bundle to $M_\tau$ is (locally) a sum of lines 
$L'\otimes L''$ for tangent line pairs at the nodes which belong 
to the special component (and possibly one other component).
\end{enumerate}
Parts (i) and (ii) are clear by construction. To see (iv), choose 
a surface $\Sigma$ in $M_\tau$. It belongs to a DM stratum $M_\Sigma$, 
which is wholly contained within $M_\tau$. The deformation space of 
$\Sigma$ is smooth, and its tangent space is the sum of the lines 
$L'\otimes L''$, over \textit{all} nodes, with the tangent space to 
$M_\Sigma$. The nodes which lie on the special component give deformations 
changing the topology of the special component, hence they represent  
normal directions to $M_\tau$; whereas the other nodes correspond to 
deformations of the complement of the special component, which are 
tangent to $M_\tau$. An automorphism of $\Sigma$ preserves 
the special component, and cannot interchange tangent and normal lines. 
This shows that the symmetry group $F$, acting on the tubular 
neighbourhood of $M_\tau$, preserves the decomposition into tangent 
and normal directions; so $M_\tau$ has no self-intersections, 
proving smoothness in (iii).

\subsection{Unique patching.}
Let us now prove uniqueness of the patched class on every $\ol{M}{}_g^n$ 
($n>0$). Attach to the marked input point $n$ a moving smooth surface 
$\Sigma_G$ of large genus $G$ with an incoming point marked $``-''$ and an 
outgoing one marked $``+''$ (the latter attached to $n$). This embeds 
$S:= \ol{M}{}^n_g\times {}^1 M_G^1$ as part of the boundary of $\ol{M}{}_
{g+G}^n$.  Let, as before, $N$ be a tubular neighbourhood of $S$ and 
$\partial{N}$ its boundary. 

\begin{lemma}\label{46}
The projection $\partial{N} \to \ol{M}{}^n_g\times {M}_G^1$ forgetting 
the point $+$ gives an isomorphism in degree less than 
$(G-1)/2$:
\[
H^\bullet(\partial{N}) \cong  H^\bullet(\ol{M}{}_g^n) 
\otimes H^\bullet(M_G)[\psi_-].
\] 
\end{lemma}

\begin{proof}
The description of $\partial{N}$ as a circle bundle over $S$ gives the 
description of $H^\bullet(\partial{N})$ in the stable range as the cohomology 
of the differential graded algebra
\[
 H^\bullet(\ol{M}{}_g^n) \otimes H^\bullet(M_G)[\psi_+,\psi_-,\eta] 
 \quad \text{with}\quad d\eta = \psi_n+\psi_+, 
\]
which implies our statement.
\end{proof}

Now, $S$ parametrises nodal degenerations at $n=+$ of those surfaces
corresponding to the open union of $U$ of DM strata in $\ol{M}{}_{g+G}^n$ 
which meet $\partial{N}$. We carry over our type decomposition of \S\ref
{stratif}  to $U\subset \ol{M}{}_{g+G}^n$ with special point $-$, 
and observe that properties 
(i)--(iv)  continue to hold. In addition, the special component now has 
geometric genus $G$ or higher. All the normal Euler classes in (iv) are 
then products of free generators of the cohomology ring.
The classes $Z$ over the $U_\tau$ then patch uniquely. But each $U_\tau$ 
factors as $M_{G+\gamma}^\nu\times \ol{M}$, and $\ol{M}$ parametrises 
surfaces whose type is \textit{strictly lower} than that of geometric 
genus $g$, with $n$ marked points. We can inductively assume their 
$Z$-classes to be known; the factorisation rule gives the $Z$-class 
on each $U_\tau$, therefore on all of $U$ and then also on $S$. The 
class on $S$ is $\ol{Z}{}_g^n\circ{D}(-\psi_n,\psi_+)\circ{}^1Z_G^1$, 
with $D$ fed into the $n$th entry of $\ol{Z}{}^n$. Lifting to $\partial
{N}$ recovers $\ol{Z}{}_g^n$, by Lemma~\ref{46}.

\subsection{Pre-stable surfaces.} Restriction to stable surfaces may 
seem unnatural from the axiomatic point of view. There are Artin stacks 
$\ol{A}{}_g^n$ parametrising all  \textit{pre-stable} curves, nodal 
curves with no condition on the rational components: they arise from 
stable curves by inserting chains of $\bP^1$'s at a node (leading to 
semi-stable curves) and trees of $\bP^1$'s at smooth points. However, 
these stacks also have normal-crossing stratifications \`a la 
Deligne-Mumford, and the inductive argument applies as before, 
ensuring uniqueness of the extension to $\ol{A}{}_g^n$.

\subsection{Appendix: An infinite-genus Deligne-Mumford space.}
One referee observed that the splitting result of this section 
has a re-formulation in the guise of a homological splitting of a certain 
``infinite-genus Deligne-Mumford space'' $\ol{M}{}_{n\cdot\infty}^n$ into 
its constituent strata. This space is a partial completion of the classifying 
space $B\Gamma_\infty^n$ of the infinite-genus mapping class group, and can be 
obtained by the addition of certain boundary strata. Roughly speaking, 
$\ol{M}{}_{n\cdot\infty}^n$ parametrizes infinite-genus nodal surfaces with 
$n$ marked points such that each irreducible component which carries a marked 
point has infinite genus, but the other components have \emph{finite}  genus. 

A geometric construction of the requisite DM space, as well as its moduli 
interpretation, require some effort; so I shall only outline the story. 
While it is true that we need the spaces only up to homotopy in order to know 
their cohomology, we need to describe $\ol{M}{}_{n\cdot\infty}^n$ as a \emph
{stratified} homotopy type, with normal structure to the strata. In this format, 
the space can be assembled from its constituent strata, which are products of 
various $M_g^k$ and factors 
of $B\Gamma_\infty^l$, in the manner in which $\ol{M}{}_g^n$ is assembled 
from its Deligne-Mumford strata, and with the same normal-crossing 
structure. Readers familiar with the structure of Deligne-Mumford boundary 
divisors should have no trouble supplying the details for this case.

A point in $\ol{M}{}_{n\cdot\infty}^n$ represents a nodal curve $C$; to this, 
we associate its stable graph $\tilde{\gamma}(C)$ in the usual way (a genus-labeled 
vertex for each component, an edge for each node, a labeled external edge for 
each marked point), and the modified graph $\gamma(C)$ which collapses all the
edges which link vertices of finite genus. We now stratify $\ol{M}{}_{n\cdot\infty}^n$ 
according to the modified graph. (For this purpose, one must take care that the 
`infinite' genera of components of the curve are really very large numbers, 
to be stabilized later;  for instance, splitting off some finite genus piece 
from a large genus surface changes the graph. This book-keeping must be 
built into the construction of $\ol{M}{}_{n\cdot\infty}^n$.) For a single marked 
point, we recover the stratification of \S\ref{stratif} by topological type of 
the special component (now stabilised to infinite genus). Call $c_\gamma$ the 
complex co-dimension of $M_\gamma$.

There is a partial ordering on strata, compatible with degeneration 
of the infinite-genus components: $\gamma \ge \gamma'$ if 
the closure of the stratum  $M_\gamma$ contains $M_{\gamma'}$. (This happens 
as soon as the former meets the latter.) This gives an 
increasing filtration of $\ol{M}{}_{n\cdot\infty}^n$ by the open subsets 
$F_\gamma:= \coprod_{\gamma'\ge\gamma} M_{\gamma'}$. The following proposition, 
suggested by the referee, has the same proof as Lemma~\ref{46}.

\begin{proposition}
\begin{trivlist}\itemsep0ex
\item (i)
The cohomology spectral sequence associated to the filtration $F_\gamma$ collapses 
at the first page: 
\[
\mathrm{gr}H^\bullet(\ol{M}{}_{n\cdot\infty}^n) = \bigoplus H^\bullet(M_\gamma)[2c_\gamma].
\]
\item (ii) Every cohomology class of $\ol{M}{}_{n\cdot\infty}^n$ is uniquely determined 
by its restrictions to all the strata $M_\gamma$. \qed
\end{trivlist}
\end{proposition}

\section{A group action on DM field theories}\label{quadhamilt}
This section reformulates the classification of semi-simple DMT's 
in terms of the action of a subgroup of the symplectic group on the 
cohomology of Deligne-Mumford spaces. This construction, which lifts 
some of Givental's quadratic Hamiltonians, may have been first flagged 
by Kontsevich \cite{cks} (see also the recent \cite{kkp}), and plays 
a substantial role in his study of deformations of open-closed field 
theories. Here, it is merely a convenient way to rephrase my classification, 
but it does provide the link with Givental's original conjecture, which 
was formulated in terms of CohFT potentials.
The context is more general than in the Introduction: we allow $D\neq 
\Id$, and this requires us to review the notation.

\subsection{Definitions.} 
Let $\Delta$ be the completed second symmetric power of $A\bbrak{\provomega}$; 
we may view it as the space of (symmetric) $2$-variable Taylor series in 
$A^{\otimes 2}\bbrak{\provomega_{1,2}}$. The group $\GL(A)\bbrak{\provomega}$ 
acts on $V\in\Delta$ point-wise, 
\[
\mathrm{Ad}_g(V)(z_{1,2}) := (g(\provomega_1)\otimes g(\provomega_2))\circ{V}
(\provomega_{1,2}). 
\]
Let $\GL^+\subset \GL(A)\bbrak{\provomega}$ be the congruence subgroup $\equiv \Id\pmod
{\provomega}$, and define $\BSp^+:= \GL^+\ltimes \exp(\Delta)$, the second 
factor denoting the vector Lie group with Lie algebra $\Delta$. Call 
$\BF$ the space of polynomial functions on $A\bbrak{\provomega}$, introduce 
a formal parameter $\hbar$ and consider, on the space $\BF\ppar{\hbar}$
\begin{itemize}
 \item the \emph{translation action} of $A\bbrak{\provomega}$: $(T_x\cF)(y) = \cF(y-x)$;
 \item the \emph{geometric action} of $\GL^+$: $( g\cF)(x) = \cF(g^{-1}x)$;
 \item the action of $\exp(\Delta)$, exponentiating 
 the quadratic-differentiation action of $\hbar\Delta$.
\end{itemize}
Together, these assemble to an action of $\BSp^+\ltimes{A}\bbrak{\provomega}$. 
When $A\bbrak{\provomega}$ is doubled to a symplectic vector space, $\BF$ can 
be regarded as the Fock representation of its Heisenberg group $\BH$ constructed 
therefrom, $\BSp^+$ is a subgroup of the symplectic group $\BSp$, acting on $\BH$, 
and $\Delta$ is the ``upper right corner" of the Lie algebra of $\BSp$. The 
(projective) metaplectic representation of $\BSp$ on $\BF$ induces on $\BF\ppar{\hbar}$ 
the action of $\BSp^+$ that we have just described, except that we have chosen 
to rescale $\Delta$ by~$\hbar$. To be precise, only the Lie algebra of $\BSp$ 
acts on polynomial functions on $A\bbrak{z}$; integrating the action to the group 
$\BSp$ requires one to complete $\BF$ in some way. Nonetheless, $\Delta
$-differentiation does exponentiate on $\BF\ppar{\hbar}$, and our $\hbar$ scaling 
will match the action we need on DMT potentials. 

Note that we have \emph{not} committed to an identification of the symplectic 
space $A\bbrak{\provomega}\oplus A\bbrak{\provomega}^*$ with $(A\ppar{\provomega},
\Omega)$ as in \S\ref{loopgroup}. Most importantly, the geometric action of 
(symplectic group elements) $g\in \GL^+$ does \emph{not} agree with the metaplectic 
one, induced from its point-wise action on $A\ppar{\provomega}$ in \S\ref{loopgroup}: 
rather, the latter comes from a \emph{different} embedding of (the symplectic part 
of)~$\GL^+$ in~$\BSp^+$; see Proposition~\ref{2emb} below. 

\subsection{Action on DMT's.}
A Deligne-Mumford theory defines a vector in the 
space of $\frS_n$-invariant cohomologies
\[
A^{DM}:= \prod_{g,n} H^\bullet\left(\ol{M}{}_g^n; (A^*)^{\otimes n}
\right){}^{\frS_n}.
\]
To any $\ol{Z} \in A^{DM}$, not necessarily one coming from a DMT, we 
assign as in \S\ref{potential} its \emph{potential}
\begin{equation}
\cA(x) = \exp\left\{\sum_{g,n} \frac{\hbar^{g-1}}{n!}
	\int\nolimits_{\ol{M}{}_g^n} \ol{Z}{}_g^n\left(x(\psi_1), \dots,
	x(\psi_n)\right)\right\},
\end{equation}
living in a completion of $\BF\ppar{\hbar}$. It converges as a formal power series 
in $\hbar, x$ and $x^3/\hbar$, but is in fact of a very restricted kind, thanks to 
the dimensions $3g+n-3$ of the spaces $\ol{M}{}_g^n$. Thus, the exponent is a formal 
series in $\{x,\hbar, x^3/\hbar\}$ for $x\in A\oplus Az$, whose coefficients are 
polynomials in the $\provomega^2 A\bbrak{\provomega}$ variables. This shows that 
the differentiation by $z^2A\bbrak{z}$ and of $\hbar\Delta$ can be exponentiated
to a linear enlargement of $\BF\ppar{\hbar}$ which contains the potentials. 

Let $\BH^+$ and $\BH^{++}$ be the natural lifts of $zA\bbrak{z}, z^2A\bbrak{z}$ 
in $\BH$. I will define an action of $\BSp^+\ltimes \BH^{++}$ on $A^{DM}$ which 
lifts the action on potentials. Now, a distinguished point $I_A\in A^{DM}$ represents 
the trivial theory based on $A$; its $(g,n)$-component is the $n$th co-power 
of $\alpha^g$ (interpreted as the Frobenius trace, if $n=0$). We will 
verify the following DMT version of Theorem~2: the semi-simple DM theories 
constitute the $\BSp^+\ltimes \BH^{++}$-orbit of $I_A$. Specializing to 
Cohomological Field theories will lead to the original version of 
Theorem~2.

As we will see, this lifted action extends infinitesimally to the larger 
group $\BSp^+\ltimes\BH^+$, but the exponentiated action of the linear modes 
$\provomega{A}\subset\BH^+$ has singularities. We will compute the action 
explicitly in the case of semi-simple DMT's, and will see that these linear 
modes vary the algebra structure of $A$, re-scaling the projectors. On the 
other hand, a similarly-defined translation by zero-modes is more complicated, 
and does \emph{not} commute with the rest of $\BH^+$; see \S\ref{reconst}.

\subsection{Translation.}\label{trans}
Let $\ol{Z} \in A^{DM}$ be any class. For $a(\provomega)\in \provomega 
\cdot{A} \bbrak{\provomega}$, define a new class ${}_a\ol{Z}$ by setting
\[
{}_a\ol{Z}{}_g^n(x_1,\dots,x_n) =: 
	\sum_{m\ge 0} \frac{(-1)^m}{m!} 
	\int_{\ol{M}{}_g^{n+m}}^{\ol{M}{}_g^{n}} \ol{Z}{}_g^{n+m}
	(x_1,\dots,x_n,a(\psi_{n+1}),\dots,a(\psi_{n+m})).
\]
All $\psi$-classes are on $\ol{M}{}_g^{n+m}$. With $a=0$, we recover 
$\ol{Z}$. For dimensional reasons, the sum is \emph{finite} 
if $a\in \provomega^2\cdot{A} \bbrak{\provomega}$, but linear components 
$\provomega{A}$ can cause convergence problems, and should \emph
{a priori} be treated as formal variables. For semi-simple DMT's, we 
will see below that ${}_a\ol{Z}$ depends \emph{rationally} on $a\in zA$. 

We claim that ${}_a\left({}_b\ol{Z}\right) = {}_{a+b}\ol{Z}$. Indeed, the 
second-order infinitesimal variation, capturing the linear effect of an 
infinitesimal $b$-translation followed by that of an $a$-translation, is 
\begin{equation}\label{secondvar}
 \frac{\delta^2\ol{Z}{}_g^n}{\delta{a}\delta{b}}(x_1,\dots,x_n) 
 	= \int_{\ol{M}{}_g^{n+1}}^{\ol{M}{}_g^{n}}
	\int_{\ol{M}{}_g^{n+2}}^{\ol{M}{}_g^{n+1}}
	\ol{Z}{}_g^{n+2}\left(x_1,\dots,x_n,a(\varphi^*\psi_{n+1}), 
	b(\psi_{n+2})\right),
\end{equation}
where $\varphi$ is the morphism forgetting the point $n+2$. The 
difference $a(\psi_{n+1})-a(\varphi^*\psi_{n+1})$ is a multiple 
of $[\sigma_{n+1}]$ (cf.~\S\ref{tautrefresh}); it is killed by 
$\psi_{n+2}$, therefore also by $b(\psi_{n+2})$. As a result, the 
right-hand side is symmetric in $a,b$. 

The same argument, using the presence of $\psi$-classes in $a$, 
gives the binomial expansion 
\[
\int_{\ol{M}{}_g^{n}}\ol{Z}{}_g^{n}\left(x + a(\psi_1), \dots, 
	x+a(\psi_n)\right) = \sum\nolimits_k \binom{n}{k}
	\int_{\ol{M}{}_g^{k}}{}_a\ol{Z}{}_g^{k}\left(x, \dots, x\right).
\] 
Defining a potential $\cA_a$ from ${}_a\ol{Z}$ as in \eqref{pot} 
leads to  
\[
\cA_a(x) = \cA(x-a)\quad\text{for}\quad a\in\provomega{A}\bbrak{\provomega}.
\]
In other words, $\ol{Z}\mapsto {}_a\ol{Z}$ lifts to DMT classes the 
translation action of $a$ on $\BF\ppar{\hbar}$. 

\subsection{The $\BSp^+$-action.} \label{symplift}
It is clear how the action of elements $g(\provomega)\in \GL(A)
\bbrak{\provomega}$ can be lifted to $A^{DM}$: the $i$th input of $\ol{Z}$ 
is transformed by $g^{-1}(\psi_i)$. The quadratic differentiations in 
$\Delta$ can be implemented by the addition of boundary terms, as I 
now describe. 

Recall first that $\ol{M}{}_g^n$ has one boundary divisor $D_{ir}$ parametrising 
irreducible nodal curves of genus $g-1$, and additional divisors corresponding 
to reducible nodal curves. The latter ones are labelled by tuples 
$(g',g'',n',n'',\sigma)$, where $(g',n')+(g'',n'') =(g,n)$ and the 
partitions $\sigma$ of marked points range over co-sets in $\frS_n/
(\frS_{n'}\times \frS_{n''})$. As usual, forbidden values of $(g',n')$ 
or $(g'',n'')$, giving unstable degenerations, are excluded. Our labelling 
double-counts the boundaries because of the interchange $(g',n') 
\leftrightarrow (g'',n'')$; in the case when $g'=g''$ and $n'=n''$, 
this becomes an involution of the respective boundary stratum,  
interchanging the branches at the node. In other words, a label 
determines a boundary stratum together with an ordering of the two branches. 
(This also applies to $D_{ir}$, which is a $\bZ/2$-quotient of 
$\ol{M}{}_{g-1}^{n+2}$.) Denote by $\psi',\psi''$ the two $\psi$-classes 
at the node. Call $\Lambda$ the set of labels for reducible degenerations, 
let $\Theta_\lambda$ be the Thom class of the boundary $D_\lambda$, $\lambda\in 
\Lambda$ and $\Theta_{ir}$ the one for $D_{ir}$. 

\begin{definition}\label{quadaction}
The infinitesimal action of $\delta{V}=v'\provomega^p \otimes{v}''
\provomega^q + {v''}\provomega^q\otimes {v'}\provomega^p\in\Delta$ on 
$\ol{Z}\in A^{DM}$ 
is given by
\[\begin{split}
\delta\ol{Z}{}_g^n (x_1,\dots,x_n) =& 
	 -\sum_{\lambda\in \Lambda} \Theta_\lambda\wedge\ol{Z}{}_{g'}^{n'+1}
	(x_{\sigma(1)},\dots,x_{\sigma(n')}, v')\wedge\psi'^p \\
	&\qquad\wedge\ol{Z}{}_{g''}^{n''+1}
	(x_{\sigma(n'+1)},\dots,x_{\sigma(n)}, v'')\wedge\psi''^q\\
	&- \Theta_{ir}\wedge\ol{Z}{}_{g-1}^{n+2}
	(x_1,\dots,x_n,v', v'')\wedge\psi'^p\wedge\psi''^q.
\end{split}\]
\end{definition}
\noindent
(An extension of the boundary class $\ol{Z}$ to a small tubular neighbourhood has 
been implied.) 

This is a \emph{non-linear} action --- notice the quadratic term -- that is, 
a vector field on~$A^{DM}$. To see that this really defines an action of $\Delta$, 
we must check that the effects of any two $\delta{V},\delta{W}$ commute. Now, 
the second variation, computed in either order, is expressed as a sum over all 
boundary strata of complex co-dimension $2$ in $\ol{M}{}_g^n$. These strata 
are labelled by stable curves with two distinguished nodes, and a stratum $S$ 
contributes the following term: the Thom class of~$S$, times the product of 
$\ol{Z}$-classes, with one factor for each irreducible component of the curve, 
and having the pair of entries at the two nodes contracted with~$\delta{V}$, 
respectively with~$\delta{W}$. We are exploiting the facts that nodal 
$\psi$-classes of boundary strata restrict to their counterparts on second 
boundaries, and that the Thom push-forwards, from these same second boundaries, 
factorise into two successive Thom push-forwards of the type appearing in 
Definition~\ref{quadaction}. This is the desired symmetry of the second variation. 
 
Let us now show that the actions just defined on $A^{DM}$ 
assemble to an action of $\BSp^+\ltimes\BH^+$. 

\begin{proposition}\label{68}
The action of $\GL(A)\bbrak{\provomega}$ intertwines naturally with 
those of  $\BH^+$ and $\Delta$, which commute with each other. Moreover, 
the resulting action of $\BSp^+\ltimes\BH^+$ lifts the metaplectic 
action on potentials. 
\end{proposition}
\begin{proof}
The statement about $\GL$ is clear, as it merely transforms the input 
arguments. We now check the infinitesimal ommutation of $\BH^+$ with $\Delta$. 
Recall that the derivative $(\partial_a\ol{Z})_g^n$ in the direction 
$a\in\BH^+$ is the integral along the universal curve of the 
$a$-contraction $a\vdash\ol{Z}{}_g^{n+1}$. Call $\varphi:\ol{M}{}_g^{\bullet+1}\to 
\ol{M}{}_g^\bullet$ the last forgetful morphism. Omitting the obvious symbols in
Definition~\ref{quadaction}, we have
\[\begin{split}
\delta_{a}\delta_{V}\ol{Z} &=
	\sum\Theta_\lambda\wedge\left(\big(\partial_a\ol{Z}\big)'{\psi'}^p
 	\wedge \ol{Z}{}''{\psi''}^q + \ol{Z}{}'{\psi'}^p
 	\wedge \big(\partial_a\ol{Z}\big)''{\psi''}^q \right) + 
	(\partial_a\ol{Z}){}_{g-1}^{n+2}\wedge\Theta_{ir}
	{\psi'}^p{\psi''}^q,\\
\delta_{V}\delta_{a}\ol{Z} &= 
\int_\varphi a\vdash\left(\sum \Theta_\lambda\wedge\ol{Z}{}'{\psi'}^p
 	\wedge \ol{Z}{}''{\psi''}^q \right) + 
	\int_\varphi {a}\vdash\big(\Theta_{ir}\wedge\ol{Z}{}_{g-1}
	{\psi'}^p{\psi''}^q\big);
\end{split}\]
we must show the agreement of the two. 

Now, each $\partial_a\ol{Z}$ in the first formula represents an integral 
$\int_\varphi{a}\vdash \ol{Z}$, but when extracting this operation out to the 
front of the sum,  several discrepancies arise with respect to the second formula:
\begin{enumerate}\itemsep0ex
\item The sum in $\delta_{a}\delta_{V}$ ranges over the boundary  
divisors of $\ol{M}{}{}_g^n$, that in $\delta_{V}\delta_{a}$ 
over those of $\ol{M}{}{}_g^{n+1}$.
\item The Thom classes in  $\delta_{a}\delta_{V}$ are those of  
the boundary divisors downstairs. In  $\delta_{V}\delta_{a}$, we use the 
Thom classes of the boundaries upstairs.
\item The nodal $\psi', \psi''$ classes are the ones from $\ol{M}{}_g^n$ in 
$\delta_{a}\delta_{V}$, but are those on $\ol{M}{}_g^{n+1}$ 
in $\delta_{V}\delta_{a}$.
\end{enumerate}
To establish the commutation of $\BH^+$ with $\Delta$, we must resolve these 
discrepancies.  
Concerning (i), note that each $\varphi^{-1}(D_\lambda)$, from $\ol{M}{}_g^n$,  
is the union of a pair $D_{\lambda'}\cup D_{\lambda''}$ of boundary divisors 
upstairs:\footnote{With the usual exception $g'=g'', n'=n''$ 
when we get a self-intersecting divisor, just as we do for $D_{ir}$.}  they correspond 
to the components of the universal curve $\ol{C}{}^n_g$, and are distinguished 
by the component which contains the marked point absorbing $a$. Therefore, each 
$\lambda$ in the first sum has two matching terms $\lambda', \lambda''$ in 
the second sum. Moreover, because $\varphi^*\Theta_\lambda = \Theta_{\lambda'} 
+ \Theta_{\lambda''}$,  the Thom push-forward operations in the two formulae 
match after integrating down along $\varphi$. We are therefore only left to 
account for the boundary components  $[\sigma_i]$ in the second sum  
(the sections of $\varphi$), which have no counterparts in $\delta_a\delta_V$, 
as well as the discrepancy (iii). However, all of these vanish for the same 
reason: they are killed by the positive powers of $\psi_{n+1}$ present in $a$. 

Finally, let us compare this action with the metaplectic action on potentials. 
Translation was checked earlier. It is clear that the $\GL$-action lifts the 
geometric action on $\BF\ppar{\hbar}$. The analogue for the metaplectic 
action of $\Delta$ is seen in the following interpretation of the power 
series expansion of $\cA$: it is the integral over the moduli of \emph{all, 
possibly disconnected} stable nodal surfaces, with individual components of 
the moduli space weighted down by the automorphisms of their topological type. 
In this expansion of the potential $\cA$, differentiation in the input $x$ 
involves replacing one $x$-entry in a $\ol{Z}$-factor in each term 
by the direction of differentiation, and summing over all choices of 
doing so. Quadratic differentiation is the same procedure, but applied to 
all pairs of entries. Thanks to the Thom classes in formula \eqref
{quadaction}, we can re-interpret the integral of $\delta\ol{Z}{}_g^n$ 
there over $\ol{M}{}_g^n$ as a sum of integrals over the relevant boundaries 
instead. Book-keeping confirms that we thus supply all requisite 
terms for the quadratic differentiation in the expansion of 
$\cA$. \end{proof}

\begin{proposition}\label{dmtransforms}
If $\ol{Z}$ defines a DMT, then so do all of its transforms under 
$\BSp^+\ltimes \BH^+$. More precisely, upon transforming by $\mathrm{e}
^{V(\provomega_{1,2})}\in\exp(\Delta)$,  the nodal co-form $C$ is changed 
to $C(\provomega_{1,2}) + (\provomega_1+\provomega_2) V(\provomega_{1,2})$.  
$\BH^+$-translation does not change $C$. Finally, $\GL^+$ has the 
obvious effect on $C$ via its action on $\Delta$.
\end{proposition}

\begin{proof}
For the action of $\GL^+$, this is clear from first definitions. For 
$\Delta$ and $\BH^+$, we will check that the infinitesimal action gives 
a first-order deformation of a field theory; in the process, we spell out 
its effect on the co-form $C$, and will do so first in the more delicate case 
of $\Delta$. 

More precisely, we claim that for the variation $\delta\ol{Z}$ resulting 
from $\delta{V}$, $\ol{Z} + \epsilon\cdot\delta\ol{Z}$ is a DMT over the 
ground ring $k[\epsilon]/\epsilon^2$, with nodal co-form $C + \epsilon\delta{C}$, 
where $\delta{C}(\provomega_{1,2}) = (\provomega_1 +\provomega_2)\cdot\delta{V}
(\provomega_{1,2})$. Write the DMT factorisation rule \eqref{dmop} at a 
boundary divisor $D_{\lambda_0}$, corresponding to a splitting node and labelled 
by $\lambda_0\in \Lambda$, as
\[
b_2^*\big(\ol{Z}\big) = \ol{Z}{}' \dashv C(\psi',\psi'') \vdash \ol{Z}{}'',
\]
where the two contractions $\dashv$ and $\vdash$ absorb the left and right 
factors of $C$ into the nodal slots of $Z', Z''$. In a DMT over $k[\epsilon]/
\epsilon^2$, the $\epsilon$-linear part of factorisation becomes a ``Leibniz rule''
\begin{equation}\label{leibnizfact}
b_2^*\big(\delta\ol{Z}\big) = \delta\ol{Z}{}' \dashv C(\psi',\psi'') \vdash \ol{Z}{}'' 
	+ \ol{Z}{}' \dashv \delta{C}(\psi',\psi'') \vdash \ol{Z}{}''
	+  \ol{Z}{}' \dashv C(\psi',\psi'') \vdash \delta\ol{Z}{}'',
\end{equation}
which we must verify for our specific $\delta\ol{Z}$ and proposed $\delta{C}$. 

To do so, restrict formula~\eqref{quadaction} for $\delta\ol{Z}$ to 
$D_{\lambda_0}$. Since $b_2^*\Theta_{\lambda_0}$ is the Euler class 
$-(\psi'+\psi'')$ of $D_{\lambda_0}$, the term $\lambda=\lambda_0$ in the 
sum becomes  
\[
\epsilon (\psi'+\psi'')\wedge\ol{Z}{}_{g'}^{n'+1}(x_{\sigma(1)},\dots, 
	x_{\sigma(n')}, v')\wedge\psi'^p \wedge\ol{Z}{}_{g''}^{n''+1}
	(x_{\sigma(n'+1)},\dots,x_{\sigma(n)}, v'')\wedge\psi''^q.
\]
This is precisely the contribution to \eqref{leibnizfact} of the variation 
$\delta{C}$ we posited above. On the other hand, the $\lambda\neq\lambda_0$ 
and $D_{ir}$ terms in \eqref{quadaction} correspond to boundary divisors 
on the Deligne-Mumford moduli space underlying $D_{\lambda_0}$; the nodal 
factorisation rule for $\ol{Z}'$ or $\ol{Z}{}''$ identifies those terms, 
restricted by $b_2^*$, with the $\ol{Z}{}'\dashv{C} \vdash\delta\ol{Z}{}'' 
+ \delta\ol{Z}{}'\dashv{C}\vdash\ol{Z}{}''$ terms in our Leibniz factorisation 
\eqref{leibnizfact}. A similar discussion applies to the boundary 
divisor $D_{ir}$, proving our Leibniz rule.

For an infinitesimal translation by $a(z)$, the first variation $\delta_a\ol{Z}$
is the integral $\int_\varphi {a}(\psi)\vdash\ol{Z}$ along the universal curve. 
Restricting now to $D_{\lambda_0}$, we can split the integral into two 
terms, coming from  the two irreducible components of the curve, to get 
$\ol{Z}{}'\dashv{C}(\psi',\psi'')\vdash\delta\ol{Z}{}'' + \delta\ol{Z}
'\dashv{C}(\psi',\psi'')\vdash \ol{Z}''$, and there is now no additional 
term that could provide a $\delta{C}$ contribution.  

This last argument argument conceals a subtlety: thanks to the 
presence of a $\psi$-factor, contraction with $a(\psi)$ kills the difference 
between the nodal $\psi',\psi''$-classes pulled back from $D_{\lambda_0}$ 
and those on the universal curve, over which integration is taking place. 
(Compare with the proof of Proposition~\ref{68}.)
\end{proof}

\begin{remark}\label{subtle}
If $a(z)$ contains a constant term  \emph{and} the co-form $C$ carries a 
dependence on $\psi',\psi''$, there will be a $\delta{C}$-term  accounting 
for the difference between nodal $\psi',\psi''$-classes on the curve and 
their pull-backs from $D_{\lambda_0}$. However, this does not happen in 
Cohomological Field theories, where $C$ is constant. We will exploit this 
observation in \S\ref{basdiff} below.
\end{remark}

\subsection{The action on semi-simple DMT's.} \label{dmtaction}
Let us now determine the action of a general group element $g\cdot{e}^V 
\cdot \zeta \in \GL^+\ltimes\left(\exp\Delta\times \BH^+\right)$ on 
semi-simple DMT's, in terms of their classification. The natural 
description involves the alternative parameters $(\tilde{Z}, W, E)$ 
of  \eqref{altB}. We will meet a restriction on the $z$-linear term 
of $\zeta$.

Write $\zeta = \sum_{j> 0} \zeta_j\provomega^j$. If $\zeta_1= 0$, we 
will not change the algebra structure on $A$, and the reader can skip 
straight to the statement of the Proposition below, ignoring the primes. 
However, if $\zeta_1\neq 0$, let $A'$ be the Frobenius algebra which is 
identified with $A$ as a vector space with quadratic form $\beta$, but 
with the multiplication re-defined in such a way that the new projectors 
are $P_i'= (1+\zeta_1)P_i$. Thus, the new multiplication is $x\cdot'y := 
x\cdot{y}\cdot(1+\zeta_1)^{-1}$, the new identity is $\mathbf{1}' = 
\mathbf{1} +\zeta_1$, and the Euler class is now $\alpha' =\alpha\cdot(
\mathbf{1} + \zeta_1)^{-1}$. However, note that the vector $(\alpha')^{1/2}$, 
with the square root in the prime algebra, agrees with the old $\alpha^{1/2}$. 
The construction breaks down when $(\mathbf{1}+\zeta_1)$ is \emph{not} a 
unit in $A$, so we must exclude that case. 

\begin{proposition}\label{action}
The trivial DMT $I_A$ transforms under $g\cdot\mathrm{e}^V \cdot \zeta 
\in \GL^+\ltimes(\exp(\Delta)\times\BH^+)$ into the semi-simple theory based 
on the algebra $A'$, with alternative parameters 
\[
\tilde{Z} = \exp'\left\{\sum\nolimits_{j\ge 0} 
	a'_j\kappa_j\right\}, \:{E(z) = g(z)},\:{W(z_{1,2}) = V(z_{1,2})}.
\]
Here, $\sum_{j\ge 0} a_j'\provomega^j$ is the Taylor expansion of 
$\log'\alpha^{1/2} -\log'(\mathbf{1} + \zeta/\provomega)\in A'\bbrak
{\provomega}$, and the logarithm and exponential are computed in $A'$.
\end{proposition}
\begin{remark}
Since $\log'(\mathbf{1} + \zeta_1) = \log'(\mathbf{1}') = 0$, we have 
$\exp'a'_0 = \alpha^{1/2}$. In the original 
algebra $A$, we can expand $\log\alpha^{1/2} - \log(\mathbf{1} + \zeta/
\provomega) = \sum_{j\ge 0} a_j\provomega^j$; the relation $\exp'x' 
= (1+\zeta_1)\cdot\exp{x}$ for $x' = (1+\zeta_1)\cdot{x}$ shows that the 
Taylor coefficients are then related by $a_j' = (\mathbf{1}+\zeta_1)a_j$. 
The operators of multiplication by $\exp\big\{\sum\nolimits_{j\ge0} a_j 
\kappa_j\big\}$ on $A$ and by $\exp'\big\{\sum\nolimits_{j\ge0} a'_j 
\kappa_j\big\}$ on $A'$ coincide, when we identify the two vector spaces as 
above. (However, the customary relation $a_0 =\log\alpha^{1/2}$ is broken 
if $\zeta_1\neq 0$ since involves the `wrong' $\log$.) 
\end{remark}

\begin{proof}
Note that $E$ and $W$ do not change the Frobenius algebra structure, 
which is determined by $\beta$ and by the tensor $Z_0^3: A^{\otimes3} 
\to \bC$. The effect of $\zeta$ will be checked in a moment. In particular, 
semi-simple theories remain semi-simple and we are merely looking for 
the change in parameters. 

The effect of $E$ is clear from its definition, while that of $\mathrm{e}^V$ 
was explained in Proposition~\ref{dmtransforms} above: on a theory with $E=\Id$, 
$W\mapsto W+V$. To understand $\zeta$, note first that translation cannot 
affect the $E$ and $W$ parameters of a DMT, because of the group law in 
$\BSp^+\ltimes\BH^+$. To find its effect on $\tilde{Z}$, it suffices to 
take $n=1$ and compute its first-order variation over $\tilde{M}{}_g^1$ 
under $\delta\zeta$. This leads to a differential equation governing the 
action of $\zeta$, which we solve. We omit the $\zeta$-subscript from the 
notation for tidiness (so $\tilde{Z}$ should really be ${}_\zeta\tilde{Z}$,
 etc.) and let $C_{g,1} \to M_{g,1}$ denote the universal curve. Then,
\[
\delta\tilde{Z}(\kappa_j) = -\int_{C_{g,1}}^{M_{g,1}}
	\alpha^{-1/2}\cdot\tilde{Z}(\kappa_j)\cdot \delta\zeta(\psi_2) = 
	-\alpha^{-1/2}\tilde{Z}(\kappa_j)\int_{C_{g,1}}^{M_{g,1}} 
	\tilde{Z}(\psi_2^j)\delta\zeta(\psi_2),
\]
where $\tilde{Z}(\kappa_j) = \exp\big\{\sum_{j\ge 0} c_j\kappa_j\big\}$ 
with the $c_j$ as yet unknown, $\tilde{Z}(\psi_2^j) = \exp\big\{
\sum_j c_j \psi_2^j\big\}$ and we have used the fact that $\kappa_j$ 
inside the integral is $\kappa_j$ outside plus $\psi_2^j$. Integration 
converts $\psi_2^{j+1}$ to $\kappa_j$. Quadratic and 
higher terms in $\delta\zeta$ do not give rise to $\kappa_0$ 
and so do not affect the multiplication in $A$. Assuming first 
that $\zeta_1=0$, we specialise to $\kappa_j\mapsto\provomega^j$:
\[
 \delta\tilde{Z}(\provomega^j) = -\alpha^{-1/2}\tilde{Z}
 	(\provomega^j)^2\cdot\frac{\delta\zeta(\provomega)}{\provomega}\,,
\]
which is solved by
\[
 {}_\zeta\tilde{Z}(\provomega^j) = \frac{\alpha^{1/2}}
 {\mathbf{1} + \zeta(\provomega)/\provomega}
\]
since we know the initial value $\tilde{Z}=\alpha^{1/2}$. 
Now, $\log\tilde{Z}$ is linear homogeneous in the $\kappa_j$, so we recover 
the true $\tilde{Z}$ from our specialisation by substituting 
$\provomega^j\mapsto \kappa_j$ in $\log\tilde{Z}$, and then 
exponentiating. 

Finally, the effect of $\zeta_1$-translation on the trivial 
$A$-theory can be determined directly from the formula
\[
\int_{\ol{M}{}_g^{n+1}}^{\ol{M}_g^1} 
	\psi_1\wedge\dots\wedge\psi_n = (2g+n-2)\cdot\dots\cdot(2g-1),
\]
giving
\[
{}^1_\zeta\tilde{Z}_g = \alpha^g\sum_n \frac{(-\zeta_1)^n}{n!} 
	\int_{\ol{M}{}_g^{n+1}}^{\ol{M}{}_g^1} 
	\psi_1\wedge\dots\wedge\psi_n  
	=\alpha^g\sum_n \binom{1-2g}{n} \zeta_1^n = 
	\frac{\alpha^g}{(\mathbf{1}+\zeta_1)^{2g-1}}.
\] 
This introduces no higher $\kappa$-classes, but changes the multiplication 
on $A$ in the manner claimed. 
\end{proof}

\subsection{Cohomological Field theories.}\label{cohftcase}
We now deduce Theorem~1 from Proposition~\ref{action} by identifying the 
subgroup of $\BSp^+\ltimes\BH^+$ which preserves the Cohomological Field 
theory constraint (\ref{gwconst}.i). Recall from \S\ref{gromwitconds} 
that this constraint takes the equivalent forms $B'=\Id$, $C'=\Id$ and 
$D=\Id$. In terms of $E$ and $W$, we need the identity
\begin{equation}\label{wfrome}
W'(\provomega_1,\provomega_2) = W'_E := \frac{E(\provomega_1)^{-1}
	E(-\provomega_2)-\Id}{\provomega_1 + \provomega_2},
\end{equation}
together with the symplectic condition $E(\provomega)^* E(-\provomega) 
\equiv \Id$ of \S\ref{gromwitconds}. In \S\ref{loopgroup}, we wrote  
$\BSp^+_L$ for the subgroup of symplectic matrix series $E\in\GL^+$. 
It follows from Proposition~\ref{dmtransforms} that the group homomorphism 
$E(z)\mapsto E(z)\cdot{e}^{W_E(z_{1,2})}$ identifies $\BSp^+_L$ with the 
stabiliser of $C'=\Id$ in $\BSp^+$; it is a new, a new sheared embedding 
of $\BSp^+_L$ in $\BSp^+$. 
We now use the symplectic form $\Omega$ of \S\ref{loopgroup} to 
identify the symplectic double of $A\bbrak{\provomega}$ with $A\ppar{\provomega}$. 
The group $\GL^+$ acts on $A\ppar{\provomega}$, point-wise in $\provomega$;  
its subgroup $\BSp^+_L$, by definition, preserves $\Omega$ and lies in $\BSp$. 
We write $E\mapsto\widehat{E}$ for this point-wise embedding of $\BSp^+_L$ in $\BSp$.

\begin{proposition}\label{2emb}
The two embeddings of $\BSp_L^+$ into $\BSp$ agree:
 $\widehat{E}= E\cdot{e}^{W_E}$.
\end{proposition}
\begin{proof}
We verify this on Lie algebras. Let $\delta{E} = \sum_{n>0} 
\delta{E}_n\provomega^n$; then,  
\[
\delta{W_E}(\provomega_{1,2}) =  \frac{\delta{E}(-\provomega_2)- 
	\delta{E}(\provomega_1) }{\provomega_1+\provomega_2} = 
	-\sum_{p,q\ge 0} \;\delta{E}_{p+q+1} (-\provomega_2)^p\provomega_1^q.
\]
In the monomial decomposition $\{z^n\cdot{A}\}_{n\in \bZ}$ of $A\ppar
{\provomega}\cong A\bbrak{\provomega}\oplus A\bbrak{\provomega}^*$, 
the geometric action of $\delta{E}$ is given by the operator with 
$(p,q)$ blocks 
\[
O_{p,q} = \begin{cases} 
	-\delta{E}_{p-q} & \text{for $p>q\ge 0$} \\
	(-1)^{p+q-1}\delta{E}_{p-q}^* & \text{for $0>p >q$} \\
	0 & \text{otherwise}
	\end{cases}
\]
The symplectic condition is $(-1)^{p+q} \delta{E}_{p-q}^* = \delta{E}_{p-q}$. 
On the other hand, the matrix corresponding via the symplectic form 
$\Omega$ to the quadratic differentiation operator $\delta{W_E}(\provomega_
{1,2})$ has $(p,q)$-blocks $-\delta{E}_{p-q}$ in positions $q<0\le p$. 
This supplies precisely the missing $p\ge 0 >q$ blocks for the 
point-wise multiplication action of the operator $\delta{E}(\provomega): 
A\ppar{\provomega}\to A\ppar{\provomega}$. Our statement follows. 
\end{proof}

\subsection{Flat vacuum.}\label{flatidagain}
Let us identify the vacuum vector (\S\ref{vacax}) of the theory in terms 
of the group element $\widehat{E}\cdot\zeta$. In particular, we will identify 
the subgroup of $\BSp_L^+\ltimes\BH^+$ whose action on $I_A$ preserves the 
flat vacuum condition (\ref{gwconst}.ii) with the conjugate of $\BSp_L^+$ 
by the translation $T_\provomega$ by $\provomega\mathbf{1}$. 
This will conclude the proof of Theorem~2.

By equation~\eqref{zfrome} and Proposition~\ref{action},  
\[
E^{-1}(\provomega)(\mathbf{v}) = \exp'\left\{-\sum\nolimits_{j>0} 
	a'_j\provomega^j\right\} = \mathbf{1}+\zeta/\provomega
\]
so $\zeta = \provomega(E^{-1}(\provomega)(\mathbf{v})- \mathbf{1})$. 
Clearly, the CohFTs with vacuum $\mathbf{v}$ constitute the orbit 
\[
\left( T_{z\mathbf{v}(z)}\cdot\widehat{E}\cdot T_{\provomega}^{-1}\right) (I_A),
\] 
with $E$ ranging over the symplectic $\End(A)$-valued series considered.  
(Note that the action of $T_z$ on $I_A$ is singular, but the conjugate 
$T_z\widehat{E}T_{z}^{-1}$ makes good sense, so that the group element 
in parentheses acts.) 
In particular, notice that changing the vacuum of a theory with fixed 
underlying algebra and symplectic parameter $E$ is accomplished by 
$\BH^{++}$-translation.

\section{Frobenius manifolds and homogeneity}\label{homog}

We now enrich a given DMT $\ol{Z}$ 
into a family of DMT's parametrised by a (possibly formal) neighbourhood $U$ 
of $0\in A$. When starting with a cohomological field theory, the genus zero 
part of this family defines on $U$ the  structure of a \emph{Frobenius manifold}, 
a notion introduced by Dubrovin \cite{dub}. The family of DMT's will allow 
us to incorporate the grading information of Gromov-Witten theory in the 
form of a \emph{homogeneity} condition under a vector field on $U$. The 
reader may consult \cite[\S{I}]{man} or \cite{leepan} for a broader 
account of the subject. 

\begin{definition}\label{zedu}
Given a DMT $\ol{Z}$, define for $u\in U$
\[
{}_u\ol{Z}{}_g^n(x_1,\dots,x_n) := \sum_{m\ge 0} \frac{(-1)^m}{m!} 
	\int_{\ol{M}{}_g^{n+m}}^{\ol{M}{}_g^{n}} \ol{Z}{}_g^{n+m}
	(x_1,\dots,x_n,u,\dots,u).
\]
\end{definition}
\noindent
Restriction to $U$ may be required for convergence, but for convenience 
we will treat $u$ as a genuine parameter in our formulae. 
It is straightforward to verify the DMT axioms for ${}_u\ol{Z}$ from those 
for $\ol{Z}$; the construction is formally similar to the \emph{translation} 
of \S\ref{trans}, but in this case we are using the subspace $A \subset 
A\bbrak{\provomega}$ of the Heisenberg group. However, while the effect of 
translation by $\provomega{A}\bbrak{\provomega}$ was easily expressed in 
terms of $\kappa$-classes, the structure resulting now is more complicated, 
because the new translation interacts with the boundary terms, and fails 
to commute with $\BH^+$. Microscopically, the absence of a $\psi$-factor 
in $u$ breaks the calculations in the proof of Proposition~\ref{dmtransforms}. 
Conceptually, in the case of open-closed field theories, which are controlled 
by linear categories with a cyclic trace, the $u$-parameter is related to 
deformations of the category of boundary states, whereas translation 
by $\BH^+$ is tied to the (easier) deformation of the trace. There is, 
however, one easy fact to state, which was already mentioned in Remark~\ref{subtle}.

\begin{proposition}\label{cohftpreserved}
If the DMT $\ol{Z}$ is actually a CohFT, then so is every ${}_u\ol{Z}$; 
moreover, the Frobenius bilinear form $\beta$ remains unchanged. \qed
\end{proposition}

\subsection{Frobenius manifold of a CohFT.}\label{frobman}
The previous proposition does conceal something: the product and the Frobenius 
trace $\theta$ on $A$ \emph{will} vary with $u$. We obtain a 
$u$-dependent family of Frobenius algebra structures on $A$, viewed 
as a fixed vector space with bilinear form $\beta$. Spelt out, we get for 
$g=0, n=3$ a map
\[
{}_u\ol{Z}{}_0^3: A^{\otimes 3} \to \bC.
\] 
Converted to a map $A^{\otimes 2}\to A$ by means of $\beta$, this gives a 
$u$-dependent multiplication $\cdot_u$ on $A$. This multiplication is 
evidently commutative, because of the symmetry of $\ol{Z}$, but must 
be associative as well, since it is part of a CohFT structure. (Explicitly, 
we can apply the nodal factorisation rule to the several boundary 
restrictions of the map ${}_u\ol{Z}{}_0^4: A^{\otimes 4} \to H^*\big(
\ol{M}{}_0^4\big)$. Since $\ol{M}{}_0^4 = \bP^1$ is connected, these 
restrictions define the same map $A^{\otimes 4}\to \bC$, so that 
$\beta(a\cdot_u{b}, c\cdot_u{d})$ is symmetric in the four variables.)  

We write $A_u$ when referring to the algebra structure at $u$, and  
identify each $A_u$ with the tangent space $T_uU$ using the linear structure. 
The multiplications satisfy an integrability condition, which is captured 
by the observation that ${}_u\ol{Z}{}_0^3$ is the third total 
partial derivative of a function ${}_u\ol{Z}{}_0^0$. This function, the 
\emph{potential} of the Frobenius manifold, is expressed by the series in 
Definition~\ref{zedu} with $g=n=0$, after omitting the $m\le 2$ terms. This 
integrable family of Frobenius algebras on $U$, together with the (flat) metric 
$\beta$, is called a \emph{Frobenius manifold} structure. The linear structure 
on $U\subset A$ is characterized by the \emph{flat coordinates} under $\beta$. 

We say that the Frobenius manifold has \emph{flat identity} if the unit vector 
field $\mathbf{1}$ is flat in the metric (constant in flat coordinates). It is 
shown in \cite[III]{man} that this is follows from the flat vacuum condition on 
$\ol{Z}$; we will also verify that as part of Proposition~\ref
{vacdiff} below. A Frobenius manifold is in fact equivalent to the datum of 
a genus-zero CohFT (the collection of classes $\ol{Z}{}_0^n$, satisfying 
the CohFT axioms), by an explicit reconstruction \cite{man}. 

\subsection{The basic differential equations.} \label{basdiff}
Semi-simplicity of $A$ ensures that of the nearby $A_u$, so nearby theories are 
classified by $u$-dependent data $ \tilde{Z}_u, E_u, B_u$. Assuming that $\ol{Z}$ 
is a CohFT, I describe the changes in $\tilde{Z}$ and $E$ by means of differential 
equations.

To isolate the effect of the varying multiplication, we will express it 
in the (moving) normalised canonical basis $p_i = \theta_i^{-1/2} P_i$,   
in which the product can be computed entry-wise. Let $\Pi_u:A_0\to A_u$ 
be the map identifying the normalised canonical bases in the two spaces. 
In the normalised canonical identification $\bC^N\cong A_0$, this 
gives the normalised canonical framing of $TU$. Let $*$ denote the 
entry-wise multiplication of column vectors, and $\cdot_u$ the 
multiplication in $A_u$; we have 
\begin{equation}\label{pimult}
\Pi_u(x*y) = \alpha_u^{-1/2}\cdot_u\Pi_u(x)\cdot_u\Pi_u(y).  
\end{equation}
Also define the following column vector depending on $u$ and on the $\kappa$-classes,
\[
Y_u =Y_u(\kappa) := \Pi_u^{-1}(\alpha^{1/2}\tilde{Z}_u), 
\]
whose entries are the eigenvalues of multiplication by $\tilde{Z}_u$: that is, 
$\Pi\circ (Y_u *)\circ\Pi^{-1} = (\tilde{Z}_u\cdot)$. (The $i$th entry of $Y$ is 
$\exp\{\sum_{j\ge0} a_{ij}\kappa_j\}$, with $u$-dependent coefficients $a_{ij}$.) 
Write $Y_u(z)$ for the result of the substitution $\kappa_j\mapsto z^j$. Since 
$\log{Y}(\kappa)$ is linear homogeneous in the $\kappa$'s, $Y(z)$ determines 
$Y(\kappa)$. We can now write the propagator ${}_u^1Z^n_g: A_u^{\otimes{n}} 
\to A_u$ for smooth curves of genus $g$, with incoming points $\{1,\dots,n\}$ 
and one outgoing point labelled by $0$, as follows:
\begin{equation}\label{diagZ}
{}_u^1Z^n_g(x_1,\dots,x_n) = E_u(-\psi_0)\Pi_u \left(
	Y_u(\kappa)* \Pi_u^{-1}E_u^{-1}(\psi_1)(x_1)*\dots*
	 \Pi_u^{-1}E_u^{-1}(\psi_n)(x_n)\right).
\end{equation}
The contribution of $n$ to $\kappa_0 = 2g+n-1$ gives a factor of $\alpha^{n/2}$ in 
${}_u\tilde{Z}$ and has the virtue of correcting the $n$ operations $*$ into 
the multiplication $\cdot_u$, cf.~\eqref{pimult}. We now differentiate in $u$.  
 
\begin{proposition}\label{varyE}
$E_u$ and $Y_u$ verify the following systems of ODE's in $u$, $\forall{v}\in T_uU$: 
\begin{align}
\tag{\ref{varyE}.a} \frac{\partial(E_u\Pi_u)}{\partial{v}}(\provomega)\circ 
	\Pi_u^{-1}& = \left[E_u(\provomega), \frac{(v\cdot_u)}{\provomega}\right];  \\
\tag{\ref{varyE}.b} \frac{\partial{Y_u}(z)}{\partial{v}}*Y_u(z)^{-1} &= 
	-Y_u(z)*\Pi_u^{-1}E_u(z)^{-1}\left(\frac{v}{z}\right) + 
		Y_u(0)*\Pi_u^{-1}\left(\frac{v}{z}\right). 
\end{align}
\end{proposition}

\noindent Before turning to the proof, the following comments might be 
helpful.

\begin{remark}\label{clarifeq}
\begin{trivlist}\itemsep0ex
\item (i) We use the flat structure of $TU$ to differentiate 
$\Pi_u$ and $E_u$.
\item (ii) Since $E=\Id \pmod{\provomega}$, the commutator in equation~(\ref
{varyE}.a) is regular at $\provomega=0$, where we obtain, with $E_{u,1}$ 
denoting the $\provomega$-linear term of $E_u$,
\[
\partial_v\Pi_u\circ\Pi_u^{-1} = 
	\left[{E}_{u,1}, (v\cdot_u) \right].
\]
By substituting this for the derivative of $\Pi$, (\ref{varyE}.a) can be 
expressed as a \emph{non-linear} ODE system in $E$ alone; $\Pi$ can then be 
recovered from $E$.
\item (iii) The second term on the right in equation~(\ref{varyE}.b) removes 
the pole present in the first term. 
\item (iv) Let $C_g^1:=M_g^1\times_{M_g}M_g^1$ be the universal curve over $M_g^1$ 
and note that $\int_{C_g^1}^{M_g^{1}} \psi^{j} =\kappa_{j-1}$, or zero if $j=0$. 
Because $\partial_v{Y}(\kappa) * Y^{-1}$ is linear homogeneous in 
the $\kappa$'s, we can write the ODE's for $Y_u(\kappa)$ explicitly:
\begin{equation} \tag{\ref{varyE}.c}
\frac{\partial{Y_u}(\kappa)}{\partial{v}}* Y_u(\kappa)^{-1} = 
	- \int_{C_g^1}^{M_g^{1}} Y_u(\psi)*\Pi^{-1}E^{-1}(\psi)(v).
\end{equation} 
Indeed, we will prove the equation in this form.
\item (v) A coordinate-free form of equation~(\ref{varyE}.b) is found 
in Proposition~\ref{vacdiff} below.
\end{trivlist}
\end{remark}

\begin{proof}
Proving the proposition will require us to find the variation of 
\eqref{diagZ} with $n=1$. However, to keep the formulas simple, we 
first write out the variation with $n=0$. It will then be straightforward 
to describe the additional terms for general $n$. We also drop the 
$u$-subscript from the notation when no confusion arises. 

From \eqref{diagZ}, 
\begin{equation}\label{varyZbar}
\partial_v({}^1Z) = \partial_v(E\Pi)(-\psi_0)\left(Y(\kappa)\right) 
	+ E(-\psi_0)\Pi\left(\partial_vY(\kappa)\right).
\end{equation}
This same variation is also, by definition, an integral along the 
universal curve:
\[
-\int_{C_g^1}^{M_g^{1}} {E}(-\psi_0)\Pi\left
	(Y(\kappa)*\Pi^{-1}E^{-1}(\psi)(v)\right) -v\cdot_u\frac
	{1-E(-\psi_0)}{\psi_0}\Pi\left(Y(\kappa)\right); 
\]
the second term is the boundary correction to $\ol{Z}$ on the diagonal 
section $\sigma_0$ of $M_g^1\times_{M_g}M_g^1$. The requisite picture for this 
correction attaches a three-pointed $\bP^1$ to $C_g^1$ at its output $\sigma_0$; 
this $\bP^1$ absorbs $v$ at the second input, and the output is read at the third 
point. 

Using the familiar formula $\kappa_j = \varphi^*\kappa_j + \psi^j$ upstairs, 
the integral above (without sign) becomes 
\[
{E}(-\psi_0)\Pi\left(Y(\kappa)*\int Y(\psi) * \Pi^{-1}E^{-1}
	(\psi)(v)\right) + \frac{E(-\psi_0) -1}{\psi_0}
	\left(v\cdot_u\Pi(Y(\kappa))\right);
\]
the second term comes from the correction to $\psi_0$ on the diagonal 
$\sigma_0$, and all the $\kappa$'s now live on the base $M_g^1$. 
All in all, we get 
\begin{equation}\label{lemmacompute}
\partial_v ({}^1Z)= \left[(v\cdot_u), \frac{E(-\psi_0)}
	{\psi_0}\right] \circ\Pi\left(Y(\kappa)\right) - {E}(-\psi_0)
	\circ\Pi\left(Y(\kappa)*\int Y(\psi)*\Pi^{-1}E^{-1}(\psi)(v)\right)
\end{equation}
and comparing with formula \eqref{varyZbar} \emph{suggests} a separation 
into two identities, namely (\ref{varyE}.a), with $z=-\psi_0$, and 
(\ref{varyE}.c). However, in order to \emph{prove} the proposition, 
we must:

\begin{itemize}\itemsep0ex
\item  consider $n=1$ in the variation of \eqref{diagZ}, in order to allow the 
insertion of arbitrary arguments in the first operator, in place of $\Pi(Y)$; 
\item justify the splitting of the one resulting identity into two pieces.
\end{itemize}
Taking $n=1$ changes \eqref{lemmacompute} as follows: $Y(\kappa)$ is 
replaced by $Y(\kappa)*\Pi^{-1}E^{-1}(\psi_1)(x_1)$, and an additional term,  
\[
E(-\psi_0)\left(\tilde{Z}\cdot_u\left[\frac{E^{-1}(\psi_1)}{\psi_1}, 
	v\cdot_u\right]\right),
\]
appears from the correction of $\psi_1$ along $\sigma_1$ and from 
the boundary contribution of $\sigma_1$ to $\ol{Z}$, just as explained 
in the case of $\psi_0$. Likewise, \eqref{varyZbar} changes by inserting 
$*\Pi^{-1}E^{-1}(\psi_1)(x_1)$ after $Y(\kappa)$ and $\partial_v{Y}(\kappa)$, 
and by the addition of 
\[
E(-\psi_0)\Pi\left(Y(\kappa)*\partial_v(E\Pi)^{-1}(\psi_1)(x_1)\right).
\]
Splitting the identity into separate ones will now complete the proof. 
This is accomplished by setting the  $\kappa$'s or $\psi$'s, which 
are now independent variables, selectively to zero. \emph{A priori}, 
this leaves a constant term ambiguity. That, however, is resolved by 
noting that the constant term of the first ODE, $\partial_v\Pi\circ\Pi^{-1}$, 
is a skew matrix, whereas the operator $\partial_v{Y}*$ is purely diagonal; 
so  there is no possible mixing of constant terms. 
\end{proof}

\subsection{Flat vacuum preserved.} If $\ol{Z}$ verifies the flat 
vacuum condition (\ref{gwconst}.ii), then the identity vector $\mathbf{1}
\in A_0$ remains the identity in the algebra structure at all $u$: indeed, 
in the formula for ${}_u\ol{Z}{}_0^3(1,a,b)$ in Def.~\ref{zedu}, all integrals 
with $m\neq 0$ vanish, because the integrand is lifted from the lower 
moduli space missing the first marked point:
\[
\ol{Z}{}_0^{3+m}(\mathbf{1},a,b,u,\dots) = 
\varphi^*\ol{Z}{}_0^{2+m}(a,b,u,\dots).
\]
Moreover, each ${}_u\ol{Z}$ then satisfies the flat vacuum condition 
$\varphi^*{}_u\ol{Z}{}_g^{n}(x_1,\dots) = {}_u\ol{Z}{}_g^{n+1}
(\mathbf{1},x_1,\dots)$, because of the ``base change" identity 
\[\begin{split}
\varphi^*\int_{\ol{M}{}_g^{n+m}}^{\ol{M}{}_g^{n}} \ol{Z}{}_g^{n+m}
	(x_1,\dots,x_n,u,\dots,u) &=
	\int_{\ol{M}{}_g^{n+1+m}}^{\ol{M}{}_g^{n+1}} \varphi^*\ol{Z}{}_g^{n+m}
	(x_1,\dots,x_n,u,\dots,u)\\
	&=\int_{\ol{M}{}_g^{n+1+m}}^{\ol{M}{}_g^{n+1}} \ol{Z}{}_g^{n+1+m}
	(\mathbf{1},x_1,\dots,x_n,u,\dots,u)
\end{split}\] 
confirming condition (\ref{gwconst}.ii) term-by-term in the sum 
\eqref{zedu}. Note that it is the \emph{absence} of $\psi$ in $u$ 
which carries the argument here: the vacuum, of course, is not 
preserved by $\BH^+$-translations.

\subsection{Vacuum differential equation.} \label{vacdiffeq}
The ODE's for $Y(z)$ have a 
cleaner, equivalent form in terms of the vacuum vector $\mathbf{v}(z)$ of the 
theory.

\begin{proposition}\label{vacdiff}
At each $u\in U$ and for any $v\in T_uU$, 
$\displaystyle\frac{\partial\mathbf{v}(z)}{\partial{v}} = \frac{v}{z}\cdot_u (\mathbf{1} - 
\mathbf{v}(z)). 
$
\end{proposition}
\begin{proof}
$Y(z)$ and $\mathbf{v}(z)$ are related by $\mathbf{v}(z) = E(z)\Pi\big(Y(z)^{-1}\big)$ 
(Proposition~\ref{flatvac}). Direct computation gives the following (we omit 
the argument $z$, when it is not set to zero):
\begin{align*}
\frac{\partial E\Pi(Y^{-1})}{\partial{v}} &= \frac{\partial{E\Pi}}{\partial{v}}
		(Y^{-1}) + \frac{v - E\Pi\left(Y^{-1}*Y(0)*\Pi^{-1}(v)\right)}{z} \\ 
		&= \frac{\partial{E\Pi}}{\partial{v}}
		(Y^{-1}) + \frac{v - E\left(\Pi(Y^{-1})\cdot{v}\right)}{z} \\
		&= \frac{E\left(v\cdot\Pi(Y^{-1})\right) - v\cdot E\Pi(Y^{-1}) + v - 
		E\left(\Pi(Y^{-1})\cdot{v}\right)}{z}\\
		& = \frac{v - v\cdot\mathbf{v}}{z}, 
\end{align*}
having used \eqref{pimult} and the relation $\Pi(Y(0)) =\alpha$ to convert $*$ 
to the product in $A_u$.
\end{proof}

Proposition~\ref{vacdiff}  provides the following 
formula for $\mathbf{v}(z)$ in terms of derivatives of~$\mathbf{1}$. 
Let $\partial_\mathbf{1}$ be the operator of differentiation, in flat coordinates, 
along the vector field $\mathbf{1}$. 

\begin{corollary}\label{vacfrom1}
$\mathbf{v}(z) = (1+z\partial_\mathbf{1})^{-1}(\mathbf{1}) 
	= \sum\nolimits_k (-1)^k z^k\cdot \partial_\mathbf{1}^k(\mathbf{1}).$ \qed
\end{corollary}
\noindent Thus, $\mathbf{v}$ is determined by the Frobenius manifold, and 
in particular $\mathbf{v}\equiv\mathbf{1}$ if the identity is flat. Conversely, if
$\mathbf{v}(z)\equiv \mathbf{1}$ at some point $u$, then $\partial\mathbf{1}/
\partial{v} =0$ at $u$ for all $v$, by Prop.~\ref{vacdiff}, and induction 
shows the vanishing of all higher derivatives of $\mathbf{1}$.  

\subsection{Homogeneity and the Euler vector field.}\label{eulerfield}
Assume that we are given a vector field $\xi$ on our Frobenius manifold 
$U\subset A$, whose Lie derivative action on $T_uU$ we denote by $\clL$. 
We call $U$ \emph{homogeneous} (or \emph{conformal}) of weight 
$d$ with \emph{Euler vector field} $\xi$ if the ($u$-dependent) 
multiplication operator on $T_uU$ and the quadratic form $\beta$ are 
homogeneous with weights $1$ and $2-d$, respectively. 

Since flat coordinates remain flat under the $\xi$-flow, it follows that 
$\xi$ must be affine-linear in any flat coordinates $x^j$ on $A$:  
\[
\xi = \xi_0 - \mu_j^i \cdot{x}^j\partial_i + (1-d/2) x^j \partial_j.
\]
The matrix $\mu_j^i$ contributes an infinitesimal rotation 
about $0$ in $A$, and the last term is the conformal scaling. The 
action of $\clL$ on the flat frame of vector fields, commonly denoted 
$\ad_\xi$, is given by $\mu + \left(\frac{d}{2}-1\right)\Id$. 

Following Dubrovin, we can reformulate homogeneity by viewing the space of 
sections $\Gamma(U; TU)$ as a Frobenius algebra over the ring $\bC[U]$ 
of functions on $U$. Differentiation by $\xi$ gives a derivation of 
$\bC[U]$, and the shifted operator $\clL^+:= \clL +\Id$ defines a 
compatible derivation of the algebra $\Gamma(U; TU)$. The metric has 
$\clL^+$-weight $(-d)$, and in general the $\clL^+$-weights of the 
basic objects in $A$ are eminently more reasonable than their 
$\clL$-weights, cf.~Table~\ref{basicwts}.

\begin{table}[htdp]
\begin{center}
\begin{tabular}{c|c|c|c}
object & $\clL$-weight & $\clL^+$-weight & reason \\
\hline \hline
product & $1$ & $0$ & definition \\
$\beta$ & $2-d$ & $-d$ & definition\\
$\mathbf{1}\in A$ & $-1$ & $0$ & $\mathbf{1}\cdot{x} = x$ \\
projector $P$ & $-1$ & $0$ & $P\cdot{P}=P$ \\
$\theta_i$ & $-d$ & $-d$ & $\beta(P,P)$ \\
$\theta: A\to \bC$ & $1-d$ & $-d$ & $\beta(\mathbf{1},.)$ \\
$\alpha_u$ & $d-1$ & $d$ & $\theta (x\cdot\alpha) =  \mathrm{Tr}_A(x\cdot)$\\
$(\alpha_u\cdot)$ & $d$ & $d$ & \\
\hline
\end{tabular}
\end{center}
\caption{Some basic weights}
\label{basicwts}
\end{table}%

View now the CohFT data ${}_u\ol{Z}{}_g^n: A^{\otimes n} \to H^{\bullet}
\big(\ol{M}{}_g^{n}\big)$ as a collection of $n$-ary tensor fields  on $U$, 
with values in $H^{\bullet}(\ol{M}{}_g^{n})$. Using the Lie action $\clL^+$ 
and weighting the cohomology of $\ol{M}$ by half the degree, we can extend 
the notion of homogeneity to the entire CohFT:
\begin{definition} \label{cohfthomog}
The CohFT ${}_u\ol{Z}$ is \emph{homogeneous of weight $d$} under the 
vector field $\xi$ if each tensor field $\ol{Z}{}_g^n: (TU)^{\otimes n} \to H^{2\bullet}
(\ol{M}{}_g^{n})$ is $\clL^+$-homogeneous with weight $(g-1)d$. 
\end{definition}

\noindent By considering $g=0$ and the values $n=3$ and $4$, we 
recover the Frobenius manifold homogeneity condition. Conversely, Manin's 
genus-zero reconstruction theorem shows that the latter implies the seemingly 
stronger property \eqref{cohfthomog}, in genus $g=0$, for all $n$.

\begin{example}\label{gromwitextend}
In the Gromov-Witten theories of \S\ref{gromwit}, the series 
\begin{equation}\label{totalgw}
GW_{g,u}^n := \sum_{\delta\in H_2(X;\bZ)}\mathrm{e}^
	{\langle{u}|\delta\rangle}\cdot GW^n_{g,\delta}
\end{equation}
gives a (possibly formal) function on the group $H^2(X;\bC^{\times})$, 
expressed in the Fourier modes $\mathrm{e}^u$. This group is a disjoint 
union of tori, each labelled by a character of the torsion subgroup of 
$H_2(X;\bZ)$. The \emph{divisor equation} (see for instance \cite
{leepan, giv2}) 
\[
\int{GW}^{n+1}_\delta(\dots,u)= 
-\langle{u}|\delta\rangle \cdot {GW}^n_\delta(\dots),\quad
	\text{for}\quad u\in H^2(X),
\] 
where we integrate along the last forgetful map, ensures that the family
${}_u\ol{Z}:= GW_u$ is its own $u$-variation along the $H^2$ torus 
directions, in the sense of Definition~\ref{zedu}. Near any chosen base-point, 
$H^2(X;\bC^{\times})$ can be identified with $U\cap H^2(X;\bC)\subset 
{A}$ by means of a translated exponential map. Subject to convergence, 
we can extend the family $GW_u$ to an open set $U$ of $A = H^{ev}(X)$, 
starting from our base point. If convergence fails, we treat $H^2(X;
\bC^{\times})\times H^{ev,\neq 2}(X)$ as a formal Frobenius manifold. 
The dimension formula \eqref{dimf} for the spaces of stable maps 
ensures that the family $GW_u$ obtained from \eqref{totalgw} is 
homogeneous of weight $d=\dim_\bC{X}$ with respect to the Euler field 
\[
\xi_{GW} = c_1(X) + \sum\nolimits_j \left(1-\frac{\deg(x^j)}{2}\right)
			\frac{\partial}{\partial{x}_j}
\] 
in a homogeneous basis $x^j$ of $H^\bullet(X)$. Thus, $\mu = (\deg-d)/2$. 
\end{example}

We conclude by describing the homogeneity condition in terms of the 
data $E_u, \tilde{Z}_u$. 
\begin{proposition}\label{homogcohft}
In a homogeneous semi-simple CohFT, $E_u(\provomega)$, $\tilde{Z}_u^+$ 
and $\mathbf{v}_u(z)$ are invariant under the shifted Lie action 
$\clL^+$ of the Euler field $\xi$. 
\end{proposition}
\noindent Recall that $\provomega$ has weight $1$, so we are saying 
that the $\provomega^j$th Taylor coefficient in $E_u$ has weight $(-j)$. 
The same applies to the coefficient $a_j$ of $\kappa_j$ in 
$\log\tilde{Z}^+$. It is not difficult to show that, for a vector field 
$\xi$ of the form in \S\ref{eulerfield}, these conditions are also 
sufficient for homogeneity of $\ol{Z}$, but we will not use that fact.  
\begin{proof}
The operator ${}_u^1Z_g^1$ for smooth surfaces must have weight 
$gd= (g-1)d+2 + (d-2)$, the last term being the added weight of 
replacing an input by an output. In particular, ${}^1\tilde{Z}{}_g^1 
= (\alpha^g \tilde{Z}{}^+\cdot)$ has weight $gd$, whereas $(\alpha\cdot)$ 
has weight $d$; this settles $(\tilde{Z}_u^+\cdot)$. Next, since 
${}_u^1Z_{g,1} = E(-\psi_0) \circ{}^1\tilde{Z}{}_g^1$, 
\[
\clL({}_u^1Z_{g,1}) = \clL(E(-\psi_0))\circ {}^1\tilde{Z}{}_g^1 
	+ E(-\psi_0)\circ \clL({}^1\tilde{Z}{}_g^1),
\]
showing that the first term vanishes, so $\clL(E(-\psi_0))=0$. The 
final statement follows from the relation 
\[
E(z)^{-1}(\mathbf{v}(z)) = \left.(\tilde{Z}^+)^{-1}\right|_{\kappa_j = z^j}.
\]
\end{proof}

\section{Reconstruction}\label{reconst}

I now explain the reconstruction of semi-simple cohomological field 
theories from genus zero data, confirming a conjecture of Givental's 
for Gromov-Witten theory \cite{giv}. In the case of homogeneous theories 
with flat vacuum, I also give a concrete variant which uses less input: 
the Euler vector field plus the Frobenius \emph{algebra} at a single 
semi-simple point of the Frobenius manifold (Theorem~1). This more 
economical recipe is implicit in Dubrovin's paper \cite{dub}. The present 
section is largely a review and adaptation of Givental's relevant work.

\subsection{Reconstruction from the Frobenius manifold: Givental's 
conjecture.}\label{frobreconst}
Let $\mathbf{u}$ be the vector of \emph{canonical coordinates}, for which 
the associated vector fields $\partial/\partial{u}^i$ are the projectors 
$P_i$ in the multiplication at the respective point. As shown in \cite{dub}, 
the existence of such coordinates follows from the integrability 
conditions of \S\ref{frobman}. Clearly, the $u_i$ are unique up to constant 
shifts. In the case of homogeneous Frobenius manifolds, a preferred 
choice of canonical coordinates is given by the eigenvalues of the operator 
$(\xi\cdot_u)$ of multiplication by the Euler vector field $\xi$.

\begin{proposition}\label{Fwith*}\begin{trivlist}\itemsep0ex
\item (i) The linear map $d\mathbf{u}: T_uU \to \bC^N$ is given by 
$\Pi_u^{-1}\circ(\alpha_u^{-1/2}\cdot_u)$.
\item (ii) The system of ODE's in (\ref{varyE}.a) is 
equivalent to 
\[
\frac{\partial{F}}{\partial{v}} = -\frac{(v\cdot_u)}{z}\circ{F},\quad\text{with}
	\quad F(z) = E_u(z)\circ\Pi_u\circ\exp\left(-\frac{\mathbf{u}*}{z}\right).
\]
\end{trivlist}\end{proposition}
\begin{proof}
The first part merely rewrites the defining property of $\mathbf{u}$: 
$d\mathbf{u}$ takes the projector frame to the standard frame of $\bC^N$.  
For the second claim, use the chain rule and the relation $\Pi\circ
(\frac{\partial\mathbf{u}}{\partial{v}}*) = (v\cdot_u)\circ\Pi$, which 
in turn is a consequence of part~(i) and of formula~\eqref{pimult}. 
\end{proof}

\begin{remark}\label{FwithEuler}\begin{trivlist}\itemsep0ex
\item (i) Letting $\xi = \sum_i u_i\partial/\partial{u}_i$ in canonical 
coordinates, an alternative expression for $F$ is 
\[
F(z) = E_u(z)\circ\exp\left(-\frac{(\xi\cdot_u)}{z}\right)\circ\Pi_u.
\]
In the homogeneous case, $\xi$ is the Euler vector field.
\item (ii) Usually, $E(z)$ does not converge; so $F(z)$ may not belong to 
any symplectic loop group, but only to a thickened version of it (analogous 
to the space of Laurent series infinite in both directions). One such thickening 
can be constructed as a moduli of (twisted) principal $\GL(A)$-bundles over $\bP^1$, 
with formal sections at $0$ and at $\infty$. This variety has no group 
structure, but is a homogeneous space for a (left and a right) loop group 
action, and this suffices to make the ODE meaningful.  

\end{trivlist}\end{remark}

The system of ODE's in Proposition~\ref{Fwith*}.ii is that of 
\cite[pp.1269--1270]{giv}, with the change of notation $\Psi = \Pi$, 
$R(z)=\Pi^{-1}E(z)\Pi$. Recall:
\begin{proposition}[\cite{dub, giv}]
The system of Proposition~\ref{Fwith*}.ii has solutions in which
$R\equiv \Id\pmod{\provomega}$ satisfies the symplectic condition 
$R_u(\provomega)R_u^*(-\provomega)=\Id$. These solutions are unique 
up to right multiplication by a matrix series $H(\provomega) = \exp
\left(H_1\provomega + H_3\provomega^3 +\dots\right)$ with constant 
\emph{diagonal} matrices $H_{2i+1}$. In the homogeneous case, there 
is a \emph{unique} solution with $R$  invariant under the Euler 
field. \qed
\end{proposition}

\noindent The proof of the proposition, for which we refer to Givental 
\cite{giv}, is closely related to the reconstruction procedure 
we will give below, in the homogeneous case. The ambiguity in $R$ reflects 
the possibility of a $z$-dependent shift in the canonical coordinates; the 
parity constraint comes from the symplectic condition. In terms of $E$, 
this ambiguity is the right composition with the operator 
of multiplication by a ``symplectic'' unit in $A\bbrak{\provomega}$. 
Note that Euler invariance of $R$ and $E$ are equivalent because of the 
relation $\clL(\Pi) = (d/2 -1)\Pi$. 

\begin{corollary}
A semi-simple homogeneous CohFT is determined from its Frobenius manifold,  
by the unique Euler-invariant solution $E$ of the ODE \eqref{varyE} and 
the vacuum \eqref{vacfrom1}.  
 \qed
\end{corollary}

\subsection{Ambiguity for inhomogeneous theories.} \label{ambiginhom}
The inhomogeneous theories corresponding to a given semi-simple 
Frobenius manifold are related geometrically by  \emph{Hodge bundle twists}. 
More precisely, let  $\mu_j = \mathrm{ch}_j \Lambda$ be the Chern components 
of the Hodge bundle~$\Lambda$, whose fibres are the spaces of global 
differentials along the universal curve, with simple poles allowed at 
the marked sections~$\sigma_i$. Recall that the classes~$\mu_j$ vanish 
for even $j$. We can construct a \emph{Hodge CohFT} based on $A$ by 
choosing any odd power series $h(\provomega)= \sum_j h_{2j-1}
\provomega^{2j-1}$ in $A\bbrak{\provomega}$ and setting
\[
{}^n\ol{Z}{}_g[h] := \text{($n$th co-power of) } \alpha^g\cdot
	\exp\Big\{\sum\nolimits_j h_{2j-1}\frac{(2j)!}{B_{2j}}\cdot\mu_{2j-1}\Big\},
\]
with the Bernoulli numbers 
$B_{2j}$. Basic properties of the Hodge bundle ensure that $\ol{Z}[h]$ 
is a CohFT with flat vacuum: namely, $\Lambda$ is primitive\footnote{When 
there are no marked points, we must normalise the bundle by virtually 
subtracting a trivial line.} under restriction to 
the boundaries of $\ol{M}{}_g^n$, and changes under forgetful pull-back by 
the addition of a trivial line. Note, in addition, that ${}^n\ol{Z}{}_0[h]$ 
is the trivial genus-zero theory on $A$, because the Hodge bundle $\Lambda$ is 
trivial there. Givental's calculation in \cite[\S2.3]{giv}, summarised in 
Part (i) of the next proposition (and re-derived below), identifies the theory for us. 

\begin{proposition}\label{ambig}
\begin{trivlist}\itemsep0ex
\item (i) The theory $\ol{Z}[h]$ is the transform $T_\provomega^{-1}
\circ\exp h(\provomega)\circ {T}_{\provomega}(I_A)$ of the trivial 
$A$-theory.
\item (ii) All cohomological Field theories with flat vacuum based 
on a fixed semi-simple, pointed Frobenius manifold are classified by 
matrices $E\circ\exp{h}(z)\in \BSp^+_L$, with arbitrary $h$ but the \emph{same} $E$. 
That is, they have the form $T_\provomega^{-1}E(\provomega)\exp h
(\provomega){T}_{\provomega}(I_A)$, with a fixed $E$. \qed
\end{trivlist}\end{proposition} 

\noindent
Let us revisit the flat vacuum condition \eqref{zfrome} in light 
of statement (i). Over $M_g$, we find 
\[
\sum\nolimits_j h_{2j-1}\frac{(2j)!}{B_{2j}}\cdot\mu_{2j-1} =
	\sum\nolimits_j h_{2j-1}\kappa_{2j-1},
\]
recovering the Riemann-Roch identities $\mu_{2j-1} = \frac{B_{2j}}{(2j)!}
\cdot\kappa_{2j-1}$ over $M_g$. These identities, in turn, prove 
statement~(i), because $\tilde{Z}$ determines $E$ when the latter is 
a multiplication operator in the Frobenius algebra and $\mathbf{v}\equiv \mathbf{1}$. 

\subsection{Rank one theories: a conjecture of Manin and Zograf.}\label{manzogsect}
When $A$ has rank $1$, we can give a closed formula for all possible CohFT's 
(which are necessarily semi-simple). 

Taking logarithms converts the FTFT factorisation axiom for the classes 
$\ol{Z}{}_g^n$ into the primitivity condition. Manin and Zograf 
conjectured in \cite{manzog} that the $\kappa_j$ ($j\ge 0$) and the 
$\mu_j$ ($j>0$, odd) were the only primitive 
classes on the $\ol{M}{}_g^n$; consequently, they proposed that any 
rank $1$ theory should  have the form 
\begin{equation}\label{manzogfor}
{}^n\ol{Z}{}_g = \exp\left\{\sum\nolimits_{j\ge 0} a_j\kappa_j + 
	\sum\nolimits_{j>0} b_{2j-1}\mu_{2j-1} \right\}\cdot \exp(a_0)^{\otimes n},
\end{equation}
for freely chosen constants $a_j, b_j \in \bC$. (Note that $\exp(a_0)$ is the 
normalised canonical vector.)
\begin{proposition}\label{manzogconj}
Formula \eqref{manzogfor} describes all possible rank one CohFT's. 
Flat vacuum theories are those with $a_j =0$ for $j>0$. 
\end{proposition}
\begin{proof}
The symplectic condition forces the element $E(\provomega)$ in our 
classification to have the form $\exp h(z)$. A general translation vector 
$\zeta$ in our classification inserts unrestricted $\kappa$-class 
combinations in \eqref{manzogfor}, but the flat vacuum 
condition fixes the $a_j$ to be zero.
\end{proof}

\subsection{Classification of homogeneous CohFT's.} 
Since the family ${}_u\ol{Z}$ of theories is constructed from its special 
value at $u=0$, we can describe the homogeneity condition in terms 
of the Euler field 
\[
\xi = \xi_0 -\mu_j^ix^j \partial_i +  (1-d/2)x^i\partial_i
\] 
and the classification datum $E$; as always, $(\xi_0\cdot)$ denotes the 
operator of multiplication by the (constant vector) $\xi_0$ in $A$. We focus 
on the important special case of flat vacuum theories, and show that 
they are completely determined by the Frobenius algebra structure and 
the Euler field. 

\begin{proposition}\label{eulrecursion}
The CohFT  $\ol{Z}$ with flat vacuum defined by $E$ is homogeneous 
of weight $d$ for $\xi$ iff
\[
\mu(\mathbf{1})= - \frac{d}{2}\cdot\mathbf{1} \quad\text{ and }\quad
\left[(\xi_0\cdot), E_{k+1}\right] + (\mu+k)E_k = 0.
\]
\end{proposition}
\begin{remark}\begin{trivlist}\itemsep0ex
\item (i) Without the flat vacuum assumption, the first equation must be 
replaced by the differential equation
\[
\frac{d\mathbf{v}(z)}{dz} + \frac{\mu+d/2}{z}\big(\mathbf{v}(z)\big) = 
			\frac{\xi_0}{z^2}\cdot\left(\mathbf{v}(z)-\mathbf{1}\right).
\]
The calculation follows the same steps as the proof of the proposition.
At a generic point where $\xi_0$ is invertible in the algebra (that is, 
away from the canonical coordinate axes), the Taylor coefficients of 
$\mathbf{v}$ are recursively determined by this equation.
\item (ii) The second recursion is equivalent to an ODE for the 
expression $F(z)$ of Remark~\ref{FwithEuler}.ii, 
\[
\frac{dF}{dz} + \frac{\mu}{z}\circ{F} = \frac{(\xi\cdot)}{z^2}\circ{F}.
\]
(iii) For $k=0$, we find $\mu= [E_1, (\xi\cdot)]$. When $(\xi\cdot)$ 
has repeated eigenvalues (on the big diagonal in canonical coordinates), 
solvability of this equation places constraints on $\mu$. 
In a general Frobenius manifold, one can expect semi-simplicity to  
fail on the big diagonal. However, the requisite constraint on $\mu$ 
\emph{must} hold at all \emph{semi-simple} diagonal points, because 
the solution $E_u$ exists there.  
\end{trivlist}\end{remark}

\begin{proof}
First, $\mathbf{1} = -\clL(\mathbf{1}) = -\partial(\mathbf{1})/
\partial\xi -\mu(\mathbf{1}) +(1-d/2)\cdot\mathbf{1}$; flatness of 
$\mathbf{1}$, $\partial(\mathbf{1})/\partial\xi =0$, gives the first 
relation. Next, $\clL(E_k) = -kE_k$ from Proposition~\ref{homogcohft}. But 
\[
 \clL(E_k) = \frac{\partial{E}_k}{\partial\xi} +\mu\circ{E}_k - E_k\circ\mu,
\]
whereas according to equation~(\ref{varyE}.a),
\[
 \frac{\partial{E}_k}{\partial\xi} = [E_{k+1}, (\xi\cdot)] - 
	E_k\circ \frac{\partial\Pi}{\partial\xi}\circ\Pi^{-1}.
\]
The normal canonical frame $\Pi$ scales with weight $(d/2-1)$ under 
the Euler flow; since
\[
\clL(\Pi) = \frac{\partial\Pi}{\partial\xi} + \mu\circ\Pi +(d/2-1)\Pi,
\]
we have $\partial_\xi\Pi\circ\Pi^{-1} = -\mu$ and combining the equations 
proves necessity of the conditions:
\[
 -kE_k = [E_{k+1}, (\xi\cdot)] - E_k \circ\partial_\xi(\Pi)\circ\Pi^{-1} 
 	+\mu\circ{E}_k - E_k\circ\mu = [E_{k+1}, (\xi\cdot)] + \mu\circ{E}_k.
\]

Conversely, the same calculations show that the two conditions imply 
the $\clL$-homogeneity of ${}_u\ol{Z}{}_g^n$ to first order at $u=0$, 
\[
\clL({}_u\ol{Z}{}_g^n)\Big|_{u=0} = (gd-d + n)\ol{Z}{}_g^n. 
\]
We now check that Euler homogeneity at any other point is a formal 
consequence. Recall from \S\ref{eulerfield} the action $\ad_\xi$ of 
$\clL$ on the flat frame of $TU$ and its multi-linear extension to 
tensors. Also, denote by $\Delta$ half the degree operator on 
$H^\bullet(\ol{M})$; it was implicit in Definition~\ref
{cohfthomog}. At a point $u$, $\xi$ has the value $\xi_u = \xi_0 - 
\ad_\xi(u)$ and 
\[
(\clL-\Delta)({}_u\ol{Z}{}_g^n) = 
	\partial_{\xi_u}({}_u\ol{Z}{}_g^n) 
		- {}_u\ol{Z}{}_g^n\circ\ad_\xi
	= \int_{\ol{M}{}_g^{n+1}}^{\ol{M}{}_g^n} \iota(\xi_0-\ad_\xi(u))
	{}_u\ol{Z}{}_g^{n+1} - {}_u\ol{Z}{}_g^n\circ\ad_\xi.
\]
Substitute now formula \eqref{zedu} for ${}_u\ol{Z}{}$, this becomes
\[
-\sum_m \frac{(-1)^m}{m!}\int_{\ol{M}{}_g^{n+m+1}}^{\ol{M}{}_g^n} 
	\left(\iota(u)^m\iota(\xi_0)\ol{Z}{}_g^{n+m+1} - 
	\iota(u)^m\iota(\ad_\xi(u))\ol{Z}{}_g^{n+m+1}\right)  - 
	{}_u\ol{Z}{}_g^n\circ\ad_\xi,
\]
and shifting the summation index $m\mapsto m+1$ in the second term 
of the sum converts this into  
\[
\sum_m \frac{(-1)^m}{m!}\int_{\ol{M}{}_g^{n+m}}^{\ol{M}{}_g^n} 
	\iota(u)^m\left(\partial_{\xi_0}\ol{Z}{}_g^{n+m} -
		\ol{Z}{}_g^{n+m}\circ\ad_\xi\right)
\]
By homogeneity at $u=0$, the integrand is $\iota(u)^m(\clL-\Delta)
\ol{Z}{}_g^{n+m} = \iota(u)^m(gd-d+n+m-\Delta)\ol{Z}{}_g^{n+m}$. 
Pulling $\Delta$ through the integral gives $(gd-d+n-\Delta){}_u
\ol{Z}{}_g^{n}$, proving homogeneity at $u$.
\end{proof}

\subsection{GW invariants from quantum cohomology.} 
As we now explain, Proposition~\ref{eulrecursion} determines 
$E$ from $A, \xi_0$ and $\mu$. In Gromov-Witten theory, we have:
\begin{theorem} \label{gwreconst}
 The Gromov-Witten classes $GW{}^n_{g,d}\in H^{ev}(\ol{M}{}_g^n)$ 
 of a compact symplectic manifold are uniquely determined by its 
 first Chern class and by the quantum multiplication law at any \emph
 {single} semi-simple point.
\end{theorem}

\begin{proof}
Assume first that the quantum multiplication operator $(\xi\cdot)$ 
at our chosen semi-simple point has distinct eigenvalues. Working in 
the normal canonical basis, 
the second equation in Proposition~\ref{eulrecursion} supplies the 
off-diagonal entries of $E_k$, once $E_{k-1}$ is known. Next, since 
$(\xi\cdot)$ is a diagonal matrix, the diagonal entries of the commutator 
$[(\xi\cdot),  E_{k+1}] = (\mu+k)E_k$ must vanish; since those of 
the skew matrix $\mu$ vanish as well, this fact determines the 
diagonal part of $E_k$ from its off-diagonal part. Finally, $E_0=\Id$.

In the general case, consider the block-decompositions of $\mu$ and of 
the $E_k$ corresponding to the eigenspaces of $(\xi\cdot)$. The first 
equation $[(\xi\cdot),  E_1] = \mu$ implies the vanishing of the diagonal 
blocks of $\mu$. This is a constraint which \emph{must} hold if $A$ is 
semi-simple. Given that, the off-diagonal blocks of $E_1$ are 
determined from those of $\mu$. The diagonal blocks are determined from the 
vanishing of those of $(\mu+\Id)E_1$ --- which must equal $[(\xi\cdot), 
 E_2]$ --- and in this way, the recursive determination of the $E_k$ 
proceeds as before. 
\end{proof}

{\small\noindent
\textsc{Constantin Teleman}\\
UC Department of Mathematics, 970 Evans Hall, Berkeley, CA 94720\\
\texttt{teleman@math.berkeley.edu}}

\begin{thebibliography}{ABC}
\bibitem[A1]{ab} L.~Abrams: Two-dimensional topological quantum field 
	theories and Frobenius algebras. \emph{J.\ Knot Theory Ramifications} 
	\textbf{5}  (1996), 569--587
\bibitem[A2]{ab2}L.~Abrams: The quantum Euler class and the quantum 
	cohomology of the Grassmannians. \emph{Israel J.\ Math.} \textbf{117} 
	(2000), 335--352
\bibitem[B]{bayer} A.~Bayer: Semisimple quantum cohomology and blowups. 
	\emph{Int.\ Math.\ Res.\ Not.} \textbf{2004}, 2069--2083
\bibitem[BT]{botil}C.-F.~B\"odigheimer, U.~Tillmann: Stripping and 
	splitting decorated mapping class groups. \textit{Cohomological 
	methods in homotopy theory} (Bellaterra, 1998), 47--57, Progr.\ Math. 
	\textbf{196}, Birkh\"auser, Basel, 2001.
\bibitem[C]{cost} K.~Costello: Topological conformal field theories 
	and Calabi-Yau categories. \textit{Adv. Math.} \textbf{210}, 
	(2007),
\bibitem[Ci]{ciolli} G.~Ciolli: On the quantum cohomology of some Fano 
	threefolds and a conjecture of Dubrovin. \emph{Internat.\ J.\ Math.}
	\textbf{16} (2005), 823--839
\bibitem[CG]{tomsasha} T.~Coates, A.~Givental: Quantum cobordism and 
	formal group laws. In: \textit{The unity of mathematics}, 155--171, 
	\textit{Progr.\ Math.} \textbf{244}, Birkh\"auser, Boston, 2006
\bibitem[CKS]{cks} Y.~Chen, M.~Kontsevich, A.~Schwartz: Symmetries of 
	WDVV equations. \textit{Nuclear Phys.~B}  \textbf{730}  (2005), 
	352--363. 
\bibitem[D]{dub} B.\ Dubrovin: Geometry of 2D topological field theories. 
	In: Integrable systems and Quantum Groups (Montecatini Terme, 1993), 
	\textit{Lecture Notes in Math.} \textbf{1620}, Springer, Berlin, 1996, 
	120--348
\bibitem[FOOO]{fooo} K.\ Fukaya, Y.-G.\ Oh, H.\ Ohta, K.\ Ono: Lagrangian 
	Floer theory on compact toric manifolds I, II: Preprints, arxiv:0802.1703, 
	0810.5654 
\bibitem[Ge]{get} E.~Getzler: Batalin-Vilkovisky algebras and $2$-dimensional
	topological field theories. \textit{Commun.\ Math.\ Phys.} \textbf{159} 
	(1994), 265--285
\bibitem[G1]{giv}A.~Givental: Semi-simple Frobenius structures in higher 
	genus. \textit{Internat. Math.~Res.~Notices} \textbf{23} (2001), 
	1265--1286
\bibitem[G2]{giv2}A.~Givental: On the WDVV equation in quantum $K$-theory. 
	Dedicated to William Fulton on the occasion of his 60th birthday. 
	\textit{Michigan Math.~J.} \textbf{48} (2000), 295--304 
\bibitem[G3]{giv3}A.~Givental: Gromov-Witten invariants and quantization 
	of quadratic Hamiltonians. \textit{Mosc.~Math.~J.} \textbf{1} (2001), 
	551--568
\bibitem[H]{har} J.L.~Harer: Stability of the homology of the mapping 
	class groups of orientable surfaces. \emph{Ann.\ of Math.\ (2)} 
	\textbf{121} (1985), 215--249
\bibitem[HMT]{hmt}C.~Hertling, Yu.~Manin and C.~Teleman: \emph{An 
	update on semi-simple quantum cohomology and F-manifolds.} arxiv:0803.2769
\bibitem[I]{iv} N.V.~Ivanov: Stabilization of the homology of Teichm\"uller 
	modular groups. (Russian) \emph{Algebra i Analiz} \textbf{1} (1989), 
	110--126;  translation in \emph{Leningrad Math.\ J.} \textbf{1} (1990),
	675--691
\bibitem[KK] {kabkim}A.~Kabanov, T.~Kimura: A change of coordinates on the 
	large phase space of quantum cohomology. \textit{Comm.\ Math.\ Phys.} 
	\textbf{217} (2001), 107--126
\bibitem[KKP]{kkp}L.~Katzarkov, M.~Kontsevich, T.~Pantev: Hodge theoretic 
	aspects of Mirror Symmetry. In: From Hodge theory to integrability 
	and TQFT tt*-geometry, \emph{Proc.~Sympos.~Pure Math.} \textbf{78}, 
	Amer.~Math.~Soc., 2008, 87--174
\bibitem[KM1]{km}M.~Kontsevich, Yu.~Manin: Gromov-Witten classes, quantum 
	cohomology, and enumerative geometry. \textit{Comm.\ Math.\ Phys.} 
	\textbf{164} (1994), 525--562
\bibitem[KM2]{km2} M.~Kontsevich, Yu.~Manin: Relations between the 
	correlators of the topological sigma-model coupled to gravity.
	\textit{Commun.\ Math\ Phys.} \textbf{196} (1998), 385--398
\bibitem[L]{lo}E.~Looijenga: Stable cohomology of the mapping class group 
	with symplectic coefficients and of the universal Abel-Jacobi map. 
	\textit{J.~Algebraic Geom.} \textbf{5} (1996), 135--150
\bibitem[LP]{leepan} Y-P.~Lee, R.~Pandharipande: \emph{Frobenius 
	manifolds, Gromov--Witten theory and Virasoro constraints}. 
	Book in preparation.
\bibitem[M]{man} Yu.~Manin: \emph{Frobenius manifolds, quantum cohomology 
	and moduli spaces}. AMS, 1999
\bibitem[MT]{madtil} I.~Madsen, U.~Tillmann: The stable mapping class group 
	and $Q(\mathbb{CP}^\infty_+)$. \textit{Invent.~Math.} \textbf{145} 
	(2001), 509--544
\bibitem[MMS]{mms} M.~Markl, S.~Merkulov, S.~Shadrin: Wheeled PROPs, 
	graph complexes and the master equation. \emph{J.~Pure Appl.~Algebra} 
	\textbf{213} (2009), 496--535.
	\bibitem[MW]{madw} I.~Madsen, M.~Weiss: The stable mapping class group and 
	stable homotopy theory. \textit{European Congress 
	of Mathematics}, 283--307, Eur.~Math.~Soc., Z\"urich, 2005
\bibitem[MZ]{manzog} Yu.~Manin, P.~Zograf: Invertible cohomological field 
	theories and Weil-Petersson volumes. \textit{Ann.\ Inst.\ Fourier 
	(Grenoble)} \textbf{50} (2000), 519--535. 
\bibitem[S]{seg} G.B.\ Segal: \textit{Topological Field Theory.} Notes of 
	lectures at Stanford University (1998), \texttt
	{http://www.cgtp.duke.edu/ITP99/segal}
\bibitem[Su]{sul}D.~Sullivan: Lectures on String Topology (AIM 2005, 
	Morelia 2006). \textit{String Topology: Background and Present State}.
	arXiv:0710.4141 
\bibitem[Te]{ecm} C.~Teleman: Topological field theories in 2 dimensions. 
\emph{European Congress of Mathematics}, 197Ð210, Eur.\ Math.\ Soc., ZŸrich, 2010
\bibitem[Ti]{til} U.~Tillmann: On the homotopy of the stable mapping class 
	group, \emph{Invent.~Math.} \textbf{130} (1997), 257Ð275
\bibitem[W] {wit} E.~Witten: Topological Quantum Field theory. \textit
	{Commun.\ Math.\ Phys.} \textbf{117} (1988), 353--386
\end{thebibliography}
\end{document}